%% file: main.tex
\documentclass[11pt]{article}

\makeatletter
\newcommand\blfootnote[1]{%
  \begingroup
  \renewcommand\thefootnote{}\footnote{#1}%
  \addtocounter{footnote}{-1}%
  \endgroup
}
\newcommand{\tpitchfork}{%
  \vbox{
    \baselineskip\z@skip
    \lineskip-.52ex
    \lineskiplimit\maxdimen
    \m@th
    \ialign{##\crcr\hidewidth\smash{$-$}\hidewidth\crcr$\pitchfork$\crcr}
  }%
}

\makeatother
\usepackage{overpic}
\usepackage{orcidlink}

\usepackage[english]{babel}
\usepackage{mathtools}
\usepackage[top=0.75in, bottom=0.5in, left=0.8in, right=0.8in,includefoot,heightrounded]{geometry}
\usepackage[utf8]{inputenc}
\usepackage[procnames]{listings}
\usepackage{pdflscape}
\usepackage{fancyhdr} 
\usepackage{subfiles, comment}
\usepackage{bigints}
\usepackage{hyperref}
\definecolor{ao(english)}{rgb}{0.0, 0.5, 0.0}
\definecolor{DBrown}{RGB}{161, 104, 17}
\definecolor{RawSienna}{RGB}{125, 73, 49}
\definecolor{RedViolet}{RGB}{146, 43, 62}
\definecolor{Cerulean}{RGB}{29,172,214}
\hypersetup{
    colorlinks=true,
    linkcolor = black,
    citecolor = Cerulean,
    filecolor=magenta,      
    urlcolor=blue,
    }

\usepackage{biblatex}
\usepackage{dsfont}
\usepackage{amssymb, amsmath, amsbsy,amsthm,amscd}

\numberwithin{equation}{section}
\numberwithin{figure}{section}

\usepackage{authblk}

\usepackage{graphicx}
\usepackage{hyperref}
\usepackage{multirow}
\usepackage[colorinlistoftodos]{todonotes}
\usepackage{subcaption}
\usepackage{fancyhdr}
\usepackage{enumerate}
\usepackage[all]{xy}

\newtheorem{defn0}{Definition}[section]
\newtheorem{prop0}[defn0]{Proposition}
\newtheorem{thm0}[defn0]{Theorem}
\newtheorem{lemma0}[defn0]{Lemma}
\newtheorem{corollary0}[defn0]{Corollary}
\newtheorem{example0}[defn0]{Example}
\newtheorem{remark0}[defn0]{Remark}
\newtheorem{conjecture0}[defn0]{Conjecture}
\newtheorem{ansatz0}[defn0]{Ansatz}
\newtheorem{notation0}[defn0]{Notation}

\newenvironment{definition}{ \begin{defn0}}{\end{defn0}}
\newenvironment{proposition}{\bigskip \begin{prop0}}{\end{prop0}}
\newenvironment{theorem}{\bigskip \begin{thm0}}{\end{thm0}}
\newenvironment{lemma}{\bigskip \begin{lemma0}}{\end{lemma0}}
\newenvironment{corollary}{\bigskip \begin{corollary0}}{\end{corollary0}}

\newenvironment{remark}{ \bigskip \begin{remark0}}{\end{remark0}}
\newenvironment{conjecture}{\begin{conjecture0}}{\end{conjecture0}}

\newcommand{\transv}{\mathrel{\text{\tpitchfork}}}

\input{notation}
\mathtoolsset{showonlyrefs}
\title{Oscillatory motions, parabolic orbits and collision orbits in the planar circular restricted three-body problem}
\author[1,3]{Marcel Guardia \orcidlink{0000-0002-4802-3151}}
\author[2]{Jos\'e Lamas \orcidlink{0000-0002-1809-1823}}
\author[2,3]{Tere M.Seara \orcidlink{0000-0001-8421-8717}}

\affil[1]{Departament de Matem\`atiques i Inform\`atica, Universitat de Barcelona, Gran Via, 585, 08007 Barcelona, Spain}
\affil[2]{Departament de Matem\`atiques \& IMTECH, Universitat Polit\`ecnica de Catalunya, Diagonal 647, 08028 Barcelona, Spain}
\affil[3]{Centre de Recerca Matem\`atica, Campus de Bellaterra, Edifici C, 08193, Barcelona, Spain}

\date{}
\begin{document}

\maketitle

\begin{abstract}
\blfootnote{\textit{E-mail addresses}: \href{mailto:guardia@ub.edu}{guardia@ub.edu} (M.Guardia), \href{mailto:jose.lamas.rodriguez@upc.edu}{jose.lamas.rodriguez@upc.edu} (J.Lamas), \href{mailto:tere-m.seara@upc.edu}{tere-m.seara@upc.edu} (Tere M.Seara).}
In this paper we consider the planar circular restricted three body problem (PCRTBP), which models the motion of a massless body under the attraction of other two bodies, the primaries, which describe circular orbits around their common center of mass. In a suitable system of coordinates, this is a two degrees of freedom Hamiltonian system. The orbits of this system are either defined for all (future or past) time or eventually go to collision with one of the primaries. For orbits defined for all time, Chazy provided a classification of all possible asymptotic behaviors, usually called final motions.

By considering a sufficiently small mass ratio between the primaries, we analyze the interplay between collision orbits and various final motions and construct several types of dynamics.

In particular, we show that orbits corresponding to any combination of past and future final motions can be created to pass arbitrarily close to the massive primary. Additionally, we construct arbitrarily large ejection-collision orbits (orbits which experience collision in both past and future times) and periodic orbits that are arbitrarily large and get arbitrarily close to the massive primary. Furthermore, we also establish oscillatory motions in both position and velocity, meaning that as time tends to infinity, the superior limit of the position or velocity is infinity while the inferior limit remains a real number.
\\

\end{abstract}
\newpage

\section{Introduction}\label{Section: Introduction}

Understanding the long term dynamics of the 3 body problem is one of the longstanding questions in dynamical systems. Following the seminal works by Painlevé \cite{painleve1894leccons}, it is well known that orbits are either defined for all positive time or either two or three bodies collide (same happens regarding the past behavior).

Chazy in 1922 classified the (future and past) final motions of the 3 Body Problem. They are defined as the possible asymptotic behaviors that the orbits which are defined for all positive (negative) time can have as time tends to infinity (minus infinity).

Let us denote by $q_0, q_1, q_2$ the positions of the three bodies and  by $r_k$ the vector from  $q_i$ to  $q_j$ for $i\neq k$, $j\neq k$, $i<j$. The following theorem classifies the possible final motions of the 3 body problem in terms of the mutual distances $r_i$.

\begin{theorem}[Chazy, 1922, see also \cite{MR2269239}]\label{thm:chazy}
 Every solution of the 3 Body Problem defined for all (future) time belongs to one of the following seven classes.
\begin{itemize}
\item Hyperbolic (\textit{H}): $|r_i|\to\infty$, $|\dot r_i|\to c_i>0$, $i=0,1,2$, as $t\to\infty$.
\item Hyperbolic--Parabolic (\textit{HP$_k$}): $|r_i|\to\infty$, $i=0,1,2$, $|\dot r_k|\to 0$, $|\dot r_i|\to c_i>0$, $i\neq k$, as $t\to\infty$.
\item Hyperbolic--Elliptic,  (\textit{HE$_k$}): $|r_i|\to\infty$, $|\dot r_i|\to c_i>0$, $i=0,1,2$, $i\neq k$, as $t\to\infty$, $\sup_{t\geq t_0}|r_k|<\infty$.
\item Parabolic-Elliptic (\textit{PE$_k$}): $|r_i|\to\infty$, $|\dot r_i|\to 0$, $i=0,1,2$, $i\neq k$, as $t\to\infty$, $\sup_{t\geq t_0}|r_k|<\infty$.
\item Parabolic (\textit{P}): $|r_i|\to\infty$, $|\dot r_i|\to 0$, $i=0,1,2$, as $t\to\infty$.
\item Bounded (\textit{B}):  $\sup_{t\geq t_0}|r_i|<\infty$, $i=0,1,2$.
\item Oscillatory (\textit{OS}):  $\limsup_{t\to\infty}\sup_{i=0,1,2}|r_i|=\infty$ and $\liminf_{t\to\infty}\sup_{i=0,1,2}|r_i|<\infty$.
\end{itemize}
\end{theorem}
Note that this classification applies both when $t\to+\infty$ or $t\to-\infty$. To distinguish both cases we add a superindex $+$ or $-$ to each of the cases, e.g $H^+$ and $H^-$. Among the admissible final motions at a given energy level, it is known that any future and past combination of them can be obtained for almost all choices of masses (see \cite{MR0629685, guardia2022hyperbolic, MR2350333}). 




Besides the question of existence of such motions, there is the question about their ``abundance''. It turns out that for any combination of past and future final motion it is known whether the set has zero or positive measure except for $\mathcal{OS}^+\cap \mathcal{OS}^-$ (see \cite{MR2269239}).  As is pointed out in \cite{GorodetskiK12}, V. Arnol'd, in the conference in
honor of the 70th anniversary of Alexeev, posed the following question.

\begin{conjecture}
The Lebesgue measure of the set $\mathcal{OS}^+\cap \mathcal{OS}^-$ is equal to zero.    
\end{conjecture}

This conjecture is wide open. Another fundamental question is what is the topology of this set. For instance, in some restricted 3 body problems, the set of oscillatory motions has maximal Hausdorff dimension (see \cite{GorodetskiK12}). On the other hand,   a famous conjecture by Herman (see \cite{MR1648051}), if true,  would imply  that the set of bounded orbits is nowhere dense. 

The mentioned classification and conjectures refer to orbits defined for all time. That is, they exclude orbits hitting collisions. How ``large'' is the set of such orbits? 
Saari \cite{MR0321386,MR0295648} (see also \cite{fleischer2019improbability1,fleischer2019improbability2})
proved that the set of orbits hitting a collision  has zero measure. However, it  may form a topologically
“rich” set as stated in the following conjecture by by Alekseev \cite{MR0629685} (it actually might
be traced back to Siegel, Sec. 8, P. 49 in \cite{Siegel}).

\begin{conjecture}\label{conj_alexeev}
Is there an open subset $\mathcal{U}$ of the phase space such that for a dense subset of initial conditions the associated trajectories go to a collision?\end{conjecture}

The purpose of this paper is to analyze the interplay between the different types of final motions and collision orbits  in  the planar circular restricted $3$-body problem (PCRTBP).


\subsection{Main results}\label{subsec: Main results}
The  PCRTBP models the motion of an infinitesimal body $P_3$ under the attraction of two massive bodies called primaries, labeled as $P_1$ and $P_2$. Since  $P_3$ is considered to have zero mass does not exert any force onto the other two bodies, which are assumed to perform circular motions around their common center of mass and are coplanar with the motion of $P_3$. In this configuration one can normalize the masses of $P_1$ and $P_2$ as $m_1 = 1-\mu$ and $m_2 = \mu$ with $\mu \in (0,1/2]$. Since in this paper we will consider $\mu$ small, we will refer to $P_1$ as the Sun and to $P_2$ as Jupiter.

Taking the appropriate units and considering rotating coordinates, the PCRTBP is Hamiltonian with respect to 
\begin{equation}\label{eqn:Hamiltonian in cartesian synodical coordinates}
   \hamiltonianCartesianRotatingCoordinates(q,p;\mu) = \frac{||p||^2}{2} - q_1p_2+q_2p_1 - \frac{1-\mu}{||q+(\mu,0)||} - \frac{\mu}{ ||q-(1-\mu,0)||},
\end{equation}
where $q,p \in \mathds{R}^2$ and $P_1 = (-\mu,0)$ and $P_2 = (1-\mu,0)$ are the position of the primaries, which are fixed.
The PCRTBP is reversible with respect to the symmetry
\begin{equation}\label{eqn: Symmetry of the PCRTBP in cartesian synodical coordinates}
    (q_1,q_2,p_1,p_2;t) \to (q_1,-q_2,-p_1,p_2;-t).
\end{equation}
We define the collision with the Sun  as the set
\begin{equation}\label{eqn: Alpha and Omega}
    \begin{aligned}
         \mathcal{S} &= \{(q,p) \in \mathds{R}^4 \colon q = (-\mu,0)\}.
    \end{aligned}
\end{equation}
For the PCRTBP the Chazy classification  is reduced to 
\begin{itemize}
    \item Hyperbolic ($\mathcal{H}^\pm$) $\colon \underset{t\to \pm \infty}{\lim} ||q(t)|| = \infty$ and $\underset{t\to \pm \infty}{\lim}||\dot{q}(t)|| = c > 0$.
    \item Parabolic ($\mathcal{P}^\pm)$ $\colon \underset{t\to \pm \infty}{\lim} ||q(t)|| = \infty$ and $\underset{t\to \pm \infty}{\lim}||\dot{q}(t)|| = 0$.
    \item Bounded ($\mathcal{B}^\pm$) $ \colon \underset{t\to \pm \infty}{\limsup}\, ||q(t)|| < \infty$.
    \item Oscillatory ($\mathcal{OS}^\pm$) $\colon \underset{t\to \pm \infty}{\limsup}\, ||q(t)|| = \infty$ and $\underset{t\to \pm \infty}{\liminf}\,||q(t)|| < \infty$.
\end{itemize}
For any $\mu\in (0,1/2]$, it is known (see \cite{MR3455155,MR0573346}) that 
\[
X^{+}\cap Y^{-}\neq \emptyset\quad \text{where}\quad X,Y=\mathcal{H},\mathcal{P},\mathcal{B},\mathcal{OS}.
\]
Note that, when $\mu=0$, the PCRTBP is reduced to a Kepler problem and therefore in this case $\mathcal{OS}^\pm=\emptyset$ and $\mathcal{H}^+=\mathcal{H}^-$, $\mathcal{P}^+=\mathcal{P}^-$, $\mathcal{B}^+=\mathcal{B}^-$. On the contrary,  for any $\mu>0$, any combination of past and future final motion is possible.

Let us define also the ejection and collision orbits.
\begin{definition}\label{def: Ejection-Collision orbits in cartesian synodical coordinates}
Consider an orbit $(q(t),p(t))$ of the Hamiltonian  \eqref{eqn:Hamiltonian in cartesian synodical coordinates}. Then, we call to this orbit
\begin{itemize}
\item Ejection orbit (from the Sun) if there exists $t_0 \in \mathds{R}$ such that $\underset{t\to t_0^-}{\lim} q(t) = (-\mu,0)$.
\item Collision orbit (to the Sun) if there exists $t_1\in \mathds{R}$ such that $\underset{t\to t_1^+}{\lim} q(t) = (-\mu,0)$.
\item Ejection-collision orbit if it is both collision and ejection orbit.
\end{itemize}
\end{definition}

The main results of the present paper are the following.

\begin{theorem}\label{thm: Existence of parabolic, oscillatory and periodic orbits for small energies} 
There exists $\mu_0 > 0$  such that, for any $\mu\in (0,\mu_0)$,
\[
\overline{X^+\cap Y^-}\cap \mathcal{S} \neq\emptyset\qquad \text{where}\qquad  X,Y=\mathcal{H},\mathcal{P},\mathcal{B},\mathcal{OS}.
\]
Moreover, 
\begin{itemize}
 \item There exist orbits $({q}(t),{p}(t))$ of \eqref{eqn:Hamiltonian in cartesian synodical coordinates} which are oscillatory and get arbitrarily close to collision with the Sun. Namely, they satisfy
    \begin{equation*}
        \underset{t\to \pm \infty}{\limsup}\, ||q(t)|| = \infty \;\;\;\; \text{and} \;\;\;\; \underset{t\to \pm \infty}{\liminf}\, ||q(t) + (\mu,0)|| = 0.
    \end{equation*} 
    In particular, this also implies that
    \begin{equation}\label{def:oscillatoryspeed}
            \underset{t\to \pm \infty}{\limsup}\, ||\dot q(t)|| = \infty \;\;\;\; \text{and} \;\;\;\; \underset{t\to \pm \infty}{\liminf}\, ||\dot q(t) || <\infty.
\end{equation}
            
 \item For any $\varepsilon>0$,   there exists a periodic orbit $(q(t),p(t))$ of \eqref{eqn:Hamiltonian in cartesian synodical coordinates} satisfying
    \begin{equation*}
       \underset{t\in \mathds{R}}{\sup}\;||q(t)||\geq \varepsilon^{-1} \;\;\;\; \text{and} \;\;\;\;\underset{t\in \mathds{R}}{\inf}\,||q(t)+ (\mu,0)||\leq \varepsilon.
    \end{equation*}
\end{itemize}

\end{theorem}
In fact, this theorem will be a consequence of the following one. It constructs hyperbolic sets with symbolic dynamics in the PCRTBP, which contain the collision in its closure and also contains points arbitrarily far from collision (``close to infinity'').

\begin{theorem}\label{thm:chaos} 
There exists $\mu_0 > 0$  such that, for any $\mu\in (0,\mu_0)$, there exists a section $\Pi$ transverse to the flow of \eqref{eqn:Hamiltonian in cartesian synodical coordinates} such that the induced Poincar\'e map
\[
\mathbb{P} : U \subset \Pi \to \Pi
\]
has an invariant set $\mathcal{X}$ which is
homeomorphic to $\mathbb{N}^\mathbb{Z}$
and whose  dynamics $\mathbb{P}: \mathcal{X}\to \mathcal{X}$ is topologically conjugated to the shift $\sigma: \mathbb{N}^\mathbb{Z}\to \mathbb{N}^\mathbb{Z}$, $(\sigma\omega)_k=\omega_{k+1}$.
Moreover, this invariant set satisfies that its closure intersects $\mathcal{S}$ and contains ejection orbits to $\mathcal{S}$, orbits in $\mathcal{P}^+$ and orbits in $\mathcal{P}^-$.
\end{theorem}

This theorem implies that the ``set of chaotic motions'' accumulate both at collision with the Sun and at infinity. The construction of these hyperbolic sets accumulating to $\mathcal{P}^\pm$ was already achieved in \cite{MR3455155,MR0573346}. The main novelty of this theorem is that such sets can be also constructed accumulating to ejection and collision orbits.

We are also able to construct ``large'' ejection collision orbits.
\begin{theorem}\label{thm: Existence of infinite sequence of ECO orbits for small energies}
There exists $\eta >0$ and $\mu_0 > 0$  such that, for any $\mu\in (0,\mu_0)$ and  any energy $h \in (-\eta\mu,\eta\mu)$, one can find a sequence of ejection-collision orbits $\{z_k(t)\}_{k\in \mathds{N}}$, $z_k(t) = (q_k(t),p_k(t))$ in the energy level $\hamiltonianPolarRotatingCoordinatesCenteredSun(q,p;\mu) = h$ such that
\begin{equation*}
    \underset{k\to \infty}{\limsup}\; \left(\underset{t\in\mathds{R
    }}{\sup}\;||q_k(t)||\right) = +\infty.
\end{equation*}
\end{theorem}

Theorem \ref{thm: Existence of parabolic, oscillatory and periodic orbits for small energies} gives the existence of periodic orbits and oscillatory orbits which can pass as close to the Sun as determined and also as large as determined. 
It also implies, for instance, that, for any $\varepsilon>0$,  there exists a (forward and backward) parabolic orbit $(q(t),p(t))$ such that 
    \begin{equation*}
    \underset{t\in \mathds{R}}{\inf}\; ||q(t) + (\mu,0)||\leq \varepsilon.
    \end{equation*}
Concerning Theorem \ref{thm: Existence of infinite sequence of ECO orbits for small energies}, it provides the first proof of existence of ``large'' ejection-collision orbits. Indeed, all those provided by the previous results (see Section \ref{sec:literature} below) were arbitrarily small.

\subsection{Literature and previous results}\label{sec:literature}

The literature on the analysis of final motions and collision orbits is abundant. 

Concerning the combination of different past and future final motions, it can be traced back to the work by Sitnikov for the nowadays called Sitnikov problem \cite{MR0127389}, who showed in this model that all combinations of past and future motion was possible, and in particular constructed oscillatory motions. A decade later Moser \cite{moser2001stable} gave a new proof of the same results relying on dynamical systems tools and relating them to chaotic motions (Smale horseshoes). After Moser, his approach has been implemented in other restricted three body problems \cite{MR3455155,guardia2022hyperbolic,MR0573346} (see \cite{KaloshinGalanteIII,KaloshinGalanteI,KaloshinGalanteII, paradela2022oscillatory} for results using other methods). The first results for the (non-restricted) 3 body problem were obtained by Alekseev \cite{MR0249754,MR0276949,MR0276950}, which has recently been generalized in \cite{guardia2022hyperbolic}.

Concerning collision dynamics, as already mentioned,  Saari proved that the set of colliding orbits has zero Lebesgue measure \cite{MR0321386,MR0295648}. Even if Conjecture \ref{conj_alexeev} is wide open, a partial answer was given in \cite{MR3951693}, where the authors prove that the set of orbits leading to collision in the PCRTBP is $\mu^\alpha$-dense (for some $\alpha>0$) in some open set of phase space.

The first proof of existence of ejection-collision orbits for the PCRTBP was achieved by Lacomba and Llibre \cite{MR0949626,MR0682839} (see also \cite{MR1342132}), where they prove their existence for small mass ratio and large Jacobi constant, which implies that the orbits are very close to collision (in contraposition to those obtained in Theorem \ref{thm: Existence of infinite sequence of ECO orbits for small energies} which can make arbitrarily large excursions).
These results have been later generalized in  \cite{MR4110029,MR3693390,MR4518121}. More recently, relying on computer assisted proofs,  Capi\'nski, Kepley and Mireles \cite{MR4576879F} have constructed ejection-collision orbits involving the two primaries (and other orbits which close passages to both collisions) and Capi\'nski and Pasiut \cite{capinski2024oscillatory} have constructed orbits which oscillate between collision and a compact set of phase space away from collision (these orbits are oscillatory in velocity in the sense of \eqref{def:oscillatoryspeed} but not oscillatory in the sense of Chazy).


Another dynamics associated to collisions are the so-called punctured tori, that is invariant quasiperiodic tori (for the Levi-Civita regularized three body problems) which contain binary collisions \cite{MR0967629, MR1849229, MR1919782, MR3417880}.
Passages close to collision also allow constructing the so-called second species periodic solutions, which are periodic orbits such that the massless body has a  certain number of close encounters with the small primary  (see \cite{MR1335057, MR1805879, MR2245344, MR2331205}).


Triple collision is not possible in the restricted 3 body problem, but it is indeed possible in the full 3 body problem. It has been widely studied since the pioneering results by  McGehee \cite{MR0359459} (see also \cite{MR0586428,MR1013560,MR0571374,MR0636961,MR0640127}). The analysis of triple collisions has lead to the construction of a large variety of motions in the 3 body problem. In particular, to oscillatory motions both in position and velocity (see Moeckel \cite{MR1000223, MR2350333}). Note that this analysis requires close passages to triple collision and therefore one must be in the regime of total angular momentum very close to zero. On the contrary, in the present paper there are two bodies performing circular motion (and therefore have large angular momenta) whereas the other has small angular momentum.

\subsection{Main ideas for the proofs of Theorems \ref{thm: Existence of parabolic, oscillatory and periodic orbits for small energies}, \ref{thm:chaos} and \ref{thm: Existence of infinite sequence of ECO orbits for small energies}}\label{subsec: ideas for the proofs of the main results}

The orbits constructed in Theorems \ref{thm: Existence of infinite sequence of ECO orbits for small energies} and \ref{thm: Existence of parabolic, oscillatory and periodic orbits for small energies} rely on developing an invariant manifold theory for ``singular'' invariant objects that the PCRTBP possesses. Those are the collision (with the Sun) and infinity, which after some compactification can be seen as invariant objects for the regularized flow. This will require certain changes of coordinates (and time reparameterizations): one to deal with binary collision (similar to that considered by  McGehee to analyze triple collision, see \cite{MR0359459}) and a different one to compactify infinity, also first used by McGehee (see \cite{McGeheeInf}). These changes of variables are explained in Section \ref{Section: Analysis of the invariant manifolds in McGehee coordinates}.

After these transformations, on the one hand the collision set $\mathcal{S}$ (see \eqref{eqn: Alpha and Omega}) becomes an invariant torus which contains two circles that are normally hyperbolic invariant manifolds and are foliated by critical points. On the other hand, the ``parabolic infinity'', that is the limit of parabolic orbits, at a fixed energy level becomes a periodic orbit. This periodic orbit is degenerate (the linearization of the vector field at it vanishes) but it is well known that it possesses stable and unstable invariant manifolds (see \cite{MR0362403}).

We analyze these invariant manifolds and, relying on perturbative methods (suitable versions of Poincar\'e-Melnikov Theory) we prove that they intersect transversally. These intersections plus the local analysis close to collision and infinity lead to the construction of the different types of motions provided by Theorems \ref{thm: Existence of parabolic, oscillatory and periodic orbits for small energies} and \ref{thm: Existence of infinite sequence of ECO orbits for small energies}.

Let us be more precise.
\begin{enumerate}
\item We prove that the stable manifold of infinity intersects transversally the unstable manifold of the collision and the unstable manifold of infinity intersects transversally the stable manifold of the collision (see Section \ref{sec: Asymptotic formula for the distance between the manifolds}).
\item Relying on the local analysis close to infinity (at the $\mathcal{C}^1$ level) done in \cite{moser2001stable}, one can prove that the stable and unstable invariant manifolds of the collision set intersect transversally and that these intersections can be arbitrarily far away from the Sun (see Section \ref{sec: Proof of Theorem for ECO orbits}). This proves Theorem \ref{thm: Existence of infinite sequence of ECO orbits for small energies}.

\item We analyze the dynamics close to the collision set, and we prove a $\mathcal{C}^1$ Lambda lemma type statement for the passage close to collision (see Section \ref{sec: Dynamics close to collision}). This local analysis leads to transverse intersections between the stable and unstable manifolds of infinity close (but at a fixed distance) to collision. Proceeding as in \cite{moser2001stable}, one can construct hyperbolic sets with symbolic dynamics which contain the homoclinic points to infinity in its closure (but not containing the Sun in its closure). This leads to oscillatory motions passing close to collision (and combination of past and future different final motions), but not to oscillatory motions which have the Sun at its closure, and it does not imply Theorems \ref{thm:chaos} and  \ref{thm: Existence of parabolic, oscillatory and periodic orbits for small energies}.

\item To prove these theorems we have to further analyze the invariant manifolds of infinity and the collision. We rely on what we call, with a strong abuse of language, triple intersection of invariant manifolds. It is well known that stable invariant manifolds of different objects cannot intersect. So let us explain what do we mean by that. For an open interval of energy levels, we have transverse intersections of the stable manifold of collision with the unstable of infinity (and also the other way around). We say that we have triple intersection if, moreover, these two intersections belong to the stable/unstable leaf \emph{of the same equilibrium} point in the collision set. We prove that there exists an energy level where this happens (see Section \ref{sec: Proof of Theorem parabolic orbits}). Then, relying on the local analysis close to collision and the tools developed in Moser \cite{moser2001stable} one can construct the behaviors provided by Theorems \ref{thm:chaos} and \ref{thm: Existence of parabolic, oscillatory and periodic orbits for small energies}.
\end{enumerate}

\section*{Acknowledgments}
This work was partially supported by the grant PID-2021-122954NB-100 funded by MCIN/AEI/10.13039/
\\501100011033 and ``ERDF A way of making Europe''. 

M.Guardia has been supported by the European Research Council (ERC) under the
European Union’s Horizon 2020 research and innovation programme (grant agreement
No. 757802). M.Guardia was also supported by the Catalan Institution for Research and
Advanced Studies via an ICREA Academia Prize 2019.

J.Lamas has been supported by grant 2021 FI\_B 00117 under the European Social Fund.

Tere M.Seara has been supported by the Catalan Institution for Research and Advanced Studies via an ICREA Academia Prize 2019.

This work was also supported by the Spanish State Research Agency through the Severo Ochoa and Mar\'ia de Maeztu Program for Centers and Units of Excellence in R\&D(CEX2020-001084-M).

\section{Analysis of the invariant manifolds}\label{Section: Analysis of the invariant manifolds in McGehee coordinates}

The first step towards a proof of Theorems \ref{thm: Existence of parabolic, oscillatory and periodic orbits for small energies}, \ref{thm:chaos} and  \ref{thm: Existence of infinite sequence of ECO orbits for small energies}  consists on the analysis of the collision and infinity ``invariant sets''. To this end, in Section \ref{subsec: McGehee coordinates at infinity}, we introduce the so-called \textit{McGehee coordinates at infinity} (see for instance \cite{MR0573346,McGeheeInf,moser2001stable}) to give a proper definition of the \textit{parabolic infinity} set  and analyze the dynamics ``close'' to it. On the other hand, the Hamiltonian \eqref{eqn:Hamiltonian in cartesian synodical coordinates} is singular at the collision set $\mathcal{S}$. To regularize it, in Section \ref{subsec: McGehee coordinates at collision} we use the \textit{McGehee coordinates at collision} (see \cite{MR0359459,MR3693390}).




\subsection{The McGehee coordinates at infinity}\label{subsec: McGehee coordinates at infinity}

To define the McGehee infinity coordinates, we first  express the Hamiltonian \eqref{eqn:Hamiltonian in cartesian synodical coordinates} in (synodical) polar coordinates centered at the center of mass, defined by
\begin{equation}\label{eqn: Change from synodical cartesian to synodical polar centered at CM}
\begin{aligned}
    q_1 &=  \hat{r}\cos\hat{\theta}\;\;\;\; p_1 = \hat{R}\cos\hat{\theta} - \frac{\hat{\Theta}}{\hat{r}}\sin\hat{\theta},\\
    q_2&= \hat{r}\sin\theta\;\;\;\; p_2 = \hat{R}\sin\hat{\theta} + \frac{\hat{\Theta}}{\hat{r}}\cos\hat{\theta},
\end{aligned}
\end{equation}
which leads to the Hamiltonian

\begin{equation}\label{eqn:Hamiltonian Polar Rotating Coordinates centered at CM}
    \hamiltonianPolarRotatingCoordinatesCenteredCM(\rRotatingCenteredCM,\thetaRotatingCenteredCM,\RRotatingCenteredCM,\ThetaRotatingCenteredCM) = \frac{1}{2}\left(\RRotatingCenteredCM^2 + \frac{\ThetaRotatingCenteredCM^2}{\rRotatingCenteredCM^2}\right) - \frac {1}{\rRotatingCenteredCM} -\ThetaRotatingCenteredCM - \hat{V}(\rRotatingCenteredCM,\thetaRotatingCenteredCM;\mu),
\end{equation}
where
\begin{equation}\label{eqn:Perturbed Potential in polar rotating coordinates centered at CM}
    \hat{V}(\rRotatingCenteredCM,\thetaRotatingCenteredCM;\mu) = \frac{1-\mu}{\left(\rRotatingCenteredCM^2 + 2\rRotatingCenteredCM\mu\cos\thetaRotatingCenteredCM + \mu^2\right)^{\frac{1}{2}}} + \frac{\mu}{\left(\rRotatingCenteredCM^2 - 2\rRotatingCenteredCM(1-\mu)\cos\thetaRotatingCenteredCM + (1-\mu)^2\right)^{\frac 1 2}} - \frac {1}{\rRotatingCenteredCM},
\end{equation}
and the symmetry \eqref{eqn: Symmetry of the PCRTBP in cartesian synodical coordinates} becomes

\begin{equation}\label{eqn: Symmetry of the PCRTBP in synodical polar coordinates centered at CM}
    \left(\rRotatingCenteredCM,\thetaRotatingCenteredCM,\RRotatingCenteredCM,\ThetaRotatingCenteredCM;t\right) \to \left(\rRotatingCenteredCM,-\thetaRotatingCenteredCM,-\RRotatingCenteredCM,\ThetaRotatingCenteredCM;-t\right).
\end{equation}
Then, we define the McGehee  change of coordinates (see \cite{McGeheeInf}) as
\begin{equation}\label{eqn: McGehee Coordinates at infinity}
    \hat{r} = 2\mcGeheeRadiusInfinity^{-2},
\end{equation}
in which the parabolic infinity 
\[
\Alpha=\left\{(\hat{r},\hat{\theta},\hat{R},\hat{\Theta}) \colon\hat r=+\infty, \hat R=0, \hat{\theta}\in\mathbb{T}, \hat{\Theta}\in\mathbb{R}\right\}
\]
becomes
\begin{equation}\label{eqn: Infinity set in mcgehee coordinates at infinity}
    \Alpha = \Bigg\{(\xi,\hat{\theta},\hat{R},\hat{\Theta}) \colon \mathds{R}^+\times \mathds{T}\times \mathds{R}^2\colon \xi = 0, \hat{R} = 0\Bigg\}. 
\end{equation}
Applying the change of coordinates \eqref{eqn: McGehee Coordinates at infinity} to the equations of motion associated to Hamiltonian $\hat{\mathcal{H}}$ in \eqref{eqn:Hamiltonian Polar Rotating Coordinates centered at CM} leads to
\begin{equation}\label{eqn: Eq motion in McGehee coordinates of infinity}
    \begin{aligned}
        \frac{d\xi}{dt} &=-\frac{\hat{R}\xi^3}{4}\\
        \frac{d\hat{\theta}}{dt} &= \frac{\hat{\Theta}}{4}\xi^4 - 1\\
        \frac{d\hat{R}}{dt} &= - \frac{\xi^4}{4} + \frac{\hat{\Theta}^2}{8}\xi^6 - \frac{\xi^3}{4}\partial_{\xi}\hat V(\xi,\hat{\theta};\mu)\\
        \frac{d\hat{\Theta}}{dt} &= \partial_{\hat{\theta}} \hat V(\xi,\hat{\theta};\mu),
    \end{aligned}
\end{equation}
where

\begin{equation*}
    \hat{V}(\xi,\hat{\theta};\mu) = \frac{\xi^2}{2}\left(\frac{1-\mu}{\left(1+\xi^2\mu\cos\hat{\theta} + \frac{\xi^4}{4}\mu^2\right)^{\frac 1 2}} + \frac{\mu}{\left(1- \xi^2(1-\mu)\cos\hat{\theta} + \frac{\xi^4}{4}(1-\mu)^2\right)^{\frac 1 2}}-1\right).
\end{equation*}
Note that the change of variables \eqref{eqn: McGehee Coordinates at infinity} is not symplectic, so the new vector field is no longer Hamiltonian. Nevertheless, the Hamiltonian \eqref{eqn:Hamiltonian Polar Rotating Coordinates centered at CM} becomes a first integral of system \eqref{eqn: Eq motion in McGehee coordinates of infinity} and is given by

\begin{equation}\label{eqn: Hamiltonian function in mcgehee coordinates at infinity}
    \hat{H}(\xi,\hat{\theta},\hat{R},\hat{\Theta}) = \frac{1}{2}\left(\hat{R}^2 + \frac{\hat{\Theta}^2 \xi^4}{4}\right) - \frac{\xi^2}{2} - \hat{\Theta} - \hat{V}(\xi,\theta;\mu).
\end{equation}
From \eqref{eqn: Eq motion in McGehee coordinates of infinity}, one obtains that the manifold $\Alpha$ in \eqref{eqn: Infinity set in mcgehee coordinates at infinity} is foliated by periodic orbits as $\Alpha= \underset{\hat{\Theta}_0 \in \mathds{R}}{\bigcup}\Alpha_{\hat{\Theta}_0}$ with
\begin{equation}\label{eqn: Definition of Alpha_Theta0 in mcgehee coordinates at infinity}
    \Alpha_{\hat{\Theta}_0} = \Bigg\{(\xi,\hat{\theta},\hat{R},\hat{\Theta}) \colon \mathds{R}^+\times \mathds{T}\times \mathds{R}^2\colon \xi = 0, \hat{R} = 0, \hat{\Theta}=\hat{\Theta}_0\Bigg\}.
\end{equation}
In \cite{MR0362403} it was proven that these periodic orbits have stable and unstable manifolds, which we denote by $\stableManifoldInftymu$ and $\unstableManifoldInftymu$ respectively. 
In \cite{MR3455155} and \cite{MR0573346} it is shown that $\invariantManifoldsInfinity$ intersect transversally for any $\mu \in \left(0,\frac 1 2\right]$ if $\hat{\Theta}_0$ is big enough. Theorem \ref{thm: transverse intersection of the invariant manifolds} will guarantee this transversality for small values of $\hat{\Theta}_0$ and $\mu > 0$.

Note that the rates of convergence of the invariant manifolds $\invariantManifoldsInfinity$ are polynomial in $t$ and not exponential as in the case of normally hyperbolic invariant manifolds. For this reason, in \cite{MR3927089} and \cite{MR3455155}, the set $\Alpha$ in \eqref{eqn: Infinity set in mcgehee coordinates at infinity} is called a normally parabolic invariant manifold.
   


\subsection{The McGehee coordinates at collision}\label{subsec: McGehee coordinates at collision}
To study the collision set $\mathcal{S}$ in \eqref{eqn: Alpha and Omega}, we express first the Hamiltonian \eqref{eqn:Hamiltonian in cartesian synodical coordinates} in synodical polar coordinates centered at the primary we want to regularize, i.e, $\mathcal{S}$.

\begin{equation}\label{eqn: Change from synodical cartesian to synodical polar centered at P1}
    \begin{aligned}
        q_1 &= -\mu + r\cos\theta\;\;\;\; p_1 = R\cos\theta - \frac{\Theta}{r}\sin\theta,\\
        q_2&= r\sin\theta\;\;\;\;\;\;\;\;\;\;\;\;\;\; p_2 = R\sin\theta + \frac{\Theta}{r}\cos\theta.
    \end{aligned}
\end{equation}
In these new coordinates, the Hamiltonian $\mathcal{H}$ in \eqref{eqn:Hamiltonian in cartesian synodical coordinates} becomes
\begin{equation}\label{eqn: Hamiltonian function in rotating polar coordinates centered at $P_1$}
\hamiltonianPolarRotatingCoordinatesCenteredSun(r,\theta,R,\Theta;\mu) = \frac 1 2 \left(R^2 + \frac{\Theta^2}{r^2}\right) - \frac 1 r - \Theta - V(r,\theta,R,\Theta;\mu),
\end{equation}
where
\begin{equation}\label{eqn: Potential hamiltonian in rotating polar coordinates cetnered at P1}
    V(r,\theta,R,\Theta;\mu) = -\mu \left(\frac 1 r + R\sin \theta + \frac{\Theta}{r}\cos \theta - \frac{1}{\sqrt{1+r^2-2r\cos\theta}}\right).
\end{equation}
Moreover, the symmetry \eqref{eqn: Symmetry of the PCRTBP in cartesian synodical coordinates} now reads
\begin{equation}\label{eqn: Symmetry of the PCRTBP in synodical polar coordinates centered at P1}
    (r,\theta,R,\Theta;t) \to (r,-\theta,-R,\Theta;-t).
\end{equation}
In these coordinates, the collision set $\mathcal{S}$ in \eqref{eqn: Alpha and Omega} becomes $\{r=0\}$.
%
Following  \cite{MR0359459}, we perform the  transformation
\begin{equation}\label{eqn:McGehee map Collision}
    \begin{aligned}
    \mcGeheeTransformationCollision \colon \mathds{R}^+ \times \mathds{T} \times \mathds{R}^2 &\to \mathds{R}^+ \times \mathds{T} \times \mathds{R}^2\\
    (r,\theta,\mcGeheeRadialMomentum, \mcGeheeAngularMomentum) &\mapsto (r,\theta,R, \Theta) = \left(r,\theta, vr^{-\frac 1 2} - \mu\sin\theta, ur^{\frac 1 2} + r^2 - \mu r \cos\theta\right).
    \end{aligned}
\end{equation}
and the change of time
\begin{equation}\label{eqn:change of time McGehee Collision}
    dt = r^\frac 3 2 d\mcGeheeTimeCollision,
\end{equation} 
so that the equations of motion associated to the Hamiltonian $\mathcal{H}$ in \eqref{eqn: Hamiltonian function in rotating polar coordinates centered at $P_1$} become

\begin{equation}\label{eqn:Motion regularized McGehee Collision}
    \begin{aligned}
    r'&= r\mcGeheeRadialMomentum\\
    \theta' &= \mcGeheeAngularMomentum\\
    \mcGeheeRadialMomentum'&= \frac{\mcGeheeRadialMomentum^2}{2} + \mcGeheeAngularMomentum^2 + 2\mcGeheeAngularMomentum r^{\frac 3 2 } + r^3 - 1 + \mu \left[1-r^2\left(\cos \theta + \frac{r-\cos \theta}{(1+r^2-2r\cos\theta)^{\frac 3 2}}\right)\right]\\
    \mcGeheeAngularMomentum' &= -\frac{\mcGeheeAngularMomentum\mcGeheeRadialMomentum}{2} -2\mcGeheeRadialMomentum r^{\frac 3 2} + \mu r^2\sin\theta\left[1-\frac{1}{(1+r^2-2r\cos\theta)^{\frac 3 2}}\right],
    \end{aligned}
\end{equation}
where $\phantom{3}^{'}$ denotes $\frac{d}{d\mcGeheeTimeCollision}$. Observe that \eqref{eqn:Motion regularized McGehee Collision} is now regular at $r=0$.

The change of variables in \eqref{eqn:McGehee map Collision} is not symplectic but the Hamiltonian $\mathcal{H}$ in \eqref{eqn: Hamiltonian function in rotating polar coordinates centered at $P_1$} is still a first integral of \eqref{eqn:Motion regularized McGehee Collision}. 
Moreover, the energy level $\{\hamiltonianPolarRotatingCoordinatesCenteredSun = h\}$ is now given by $(\mathcal{H}-h) \circ \psi = 0$, where
\begin{equation*}\label{eqn:McGehee Collision map to the hamiltonian}
(\hamiltonianPolarRotatingCoordinatesCenteredSun-h) \circ \mcGeheeTransformationCollision(r,\theta,\mcGeheeRadialMomentum,\mcGeheeAngularMomentum) = - h+ \frac{\mcGeheeRadialMomentum^2+\mcGeheeAngularMomentum^2}{2r} -\frac{r^2}{2} -\frac{1-\mu}{r} + \mu\left[-\frac{\mu}{2} +r\cos\theta - \frac{1}{\sqrt{1+r^2-2r\cos\theta}}\right].
\end{equation*}
We now multiply by $r$ to remove the singularity, obtaining
\begin{equation}\label{eqn:Energy relation McGehee Collision}
\begin{aligned}
    \Tilde{\mcGeheeFirstIntegral}(r,\theta,\mcGeheeRadialMomentum,\mcGeheeAngularMomentum;\mu,h) =& -rh + \frac{\mcGeheeRadialMomentum^2+\mcGeheeAngularMomentum^2}{2} -\frac{r^3}{2} -1 + \mu + \mu r\left[-\frac \mu 2 + r\cos\theta - \frac{1}{\sqrt{1+r^2-2r\cos\theta}}\right].
\end{aligned}
\end{equation}
The orbits belonging to the hypersurface $\{\mathcal{H}(r,\theta,R,\Theta;\mu) = h\}$, including the ejection and collision ones, now lie in $\{\Tilde{\mcGeheeFirstIntegral}(r,\theta,v,u;\mu,h) = 0\}$. Therefore, we study \eqref{eqn:Motion regularized McGehee Collision} restricted to this manifold.
It is convenient to introduce a last change of coordinates

\begin{equation}\label{eqn:polar change of coordinates for u,v into rho alpha}
    \begin{aligned}
    \Tilde{\psi}\colon  \mathds{R}^+ \times \mathds{T}^2 \times \mathds{R}^+ &\to \mathds{R} \times \mathds{T}\times \mathds{R}^2 \\
    (s,\theta,\mcGeheePolarAngle,\mcGeheePolarRadius) &\mapsto (r,\theta,\mcGeheeRadialMomentum,\mcGeheeAngularMomentum) = \left(s^2,\theta, \sqrt{2(1-\mu) +\mcGeheePolarRadius}\sin\mcGeheePolarAngle, \sqrt{2(1-\mu) +\mcGeheePolarRadius}\cos\mcGeheePolarAngle\right),
\end{aligned}
\end{equation}
such that $\{\Tilde{\mcGeheeFirstIntegral}=0\}$ becomes
\begin{equation}\label{eqn:Relation rho with (r,theta,H,mu)}
    0= \mcGeheeFirstIntegral(\mcGeheeRegularRadius,\theta,\mcGeheePolarAngle, \mcGeheePolarRadius;\mu,h) = -\mcGeheePolarRadius+ 2\mcGeheeRegularRadius^2 h + \mcGeheeRegularRadius^6 -2\mu \mcGeheeRegularRadius^2\left[-\frac{\mu}{2} +\mcGeheeRegularRadius^2\cos\theta - \frac{1}{\sqrt{1+\mcGeheeRegularRadius^4-2\mcGeheeRegularRadius^2\cos\theta}}\right].
\end{equation}

\begin{remark}\label{remark: Sign of s}
Note that we have taken $r = s^2$ so, from now on, it is enough to analyze $s \geq 0$.    
\end{remark}
To study the motion in coordinates $(s,\theta,\alpha,\rho)$, we define the $3$-dimensional submanifold
\begin{equation}\label{eqn: 3dimensional submanifold McGehee in coordinates (r,theta,rho,alpha)}
    \generalMcGeheeInvariantManifold =\{(\mcGeheeRegularRadius,\theta,\mcGeheePolarRadius,\mcGeheePolarAngle) \in \mathds{R}^+\times \mathds{T} \times \mathds{R}\times \mathds{T}  \colon \mcGeheeFirstIntegral(\mcGeheeRegularRadius,\theta,\mcGeheePolarRadius,\mcGeheePolarAngle;\mu,h) = 0\},
\end{equation}
Therefore, using $(s,\theta,\alpha)$ as coordinates in $\generalMcGeheeInvariantManifold$, the vector field \eqref{eqn:Motion regularized McGehee Collision} writes

\begin{equation}\label{eqn:Motion in coordinates (s,theta,alpha) McGehee Collision}
    \begin{aligned}
        \mcGeheeRegularRadius'&=\frac{\mcGeheeRegularRadius}{2}\sqrt{2(1-\mu) + \mcGeheePolarRadius} \sin \mcGeheePolarAngle\\
        &\\
        \theta' &= \sqrt{2(1-\mu) + \mcGeheePolarRadius} \cos \mcGeheePolarAngle \\
        &\\
        \mcGeheePolarAngle' &= \frac{\mcGeheePolarRadius'}{2[2(1-\mu)+\mcGeheePolarRadius] \tan \mcGeheePolarAngle} + \frac{\sqrt{2(1-\mu) + \mcGeheePolarRadius}}{2}\cos \mcGeheePolarAngle + 2\mcGeheeRegularRadius^3 \\
        &- \frac{\mu \mcGeheeRegularRadius^4 \sin \theta}{\sqrt{2(1-\mu) + \mcGeheePolarRadius} \sin \mcGeheePolarAngle}\left[1-\frac{1}{(1+\mcGeheeRegularRadius^4-2\mcGeheeRegularRadius^2\cos\theta)^{\frac 3 2}}\right],
    \end{aligned}
\end{equation}
where $\rho$ can be obtained from (\ref{eqn:Relation rho with (r,theta,H,mu)}).

The collision manifold $\{r=0\}$ expressed in coordinates $(s,\theta,\alpha)$  becomes the invariant torus
\begin{equation}\label{eqn: Collision manifold}
    \collisionManifold = \{(0,\theta,\mcGeheePolarAngle) \colon \theta  \in \mathds{T}, \mcGeheePolarAngle \in \mathds{T}\} \subset \generalMcGeheeInvariantManifold
\end{equation}
whose dynamics is given by
\begin{equation}\label{eqn:Dynamics on the collision}
\left.
    \begin{aligned}
        \theta' &= \constantEquilibrium \cos \mcGeheePolarAngle\\
        \mcGeheePolarAngle' &= \frac{\constantEquilibrium}{2} \cos \mcGeheePolarAngle
    \end{aligned}
\right.
\quad \text{where    } m_0 = \sqrt{2(1-\mu)}.
\end{equation}
This system has two circles of critical points

\begin{equation}\label{eqn: Circles of equilibrium points for mu neq 0}
    \begin{aligned}
        \circleSplus= \bigg\{S_{\initialtheta}^+ = \left(0, \initialtheta, \frac \pi 2\right)\colon \initialtheta \in \mathds{T}\bigg\}, \;\;\;\;\; \circleSminus = \bigg\{S_{\initialtheta}^- = \left(0,\initialtheta, -\frac{\pi}{2}\right)\colon \initialtheta \in \mathds{T}\bigg\}.
    \end{aligned}
\end{equation}
Next lemma analyzes the dynamics of system \eqref{eqn:Motion in coordinates (s,theta,alpha) McGehee Collision} close to these circles.

\begin{lemma}\label{lemma: Linear part of motion close to S+-}
Consider the system \eqref{eqn:Motion in coordinates (s,theta,alpha) McGehee Collision} for $0\leq \mu \leq \frac 1 2$. Then

\begin{itemize}
    \item The invariant circles $\circlesSpm$ in \eqref{eqn: Circles of equilibrium points for mu neq 0} are normally hyperbolic. Moreover, they have $2$-dimensional stable and unstable manifolds $\invariantManifoldsSpm = \underset{\initialtheta \in \mathds{T}}{\bigcup}W_\mu^{u,s}(S_{\initialtheta}^\pm)$.
    
    \item $\stableManifoldCollisionSplusmu$ and $\unstableManifoldCollisionSminusmu$ are contained in $\collisionManifold$. Moreover, they coincide 
   
    \begin{equation*}
        \Gamma = \stableManifoldCollisionSplusmu = \unstableManifoldCollisionSminusmu \subset \collisionManifold.
    \end{equation*}
    Therefore

    \begin{equation*}
        \collisionManifold = \circleSplus \cup \circleSminus \cup \Gamma,
    \end{equation*}
    and $\Gamma$ is foliated by a family of heteroclinic orbits between $\equilibriumPointSminus$ and $\equilibriumPointSplus$, for $\initialtheta\in \mathds{T}$. The heteroclinic orbits in $\Gamma\cap \{\alpha\in (-\pi/2,\pi/2)\}$ can be parameterized as
    \begin{equation}\label{eqn:Heteroclinic Orbits inside the collision manifold}
    \begin{aligned}
        &\heteroclinicConnectionCollisionmu\left(\mcGeheeTimeCollision;\initialtheta\right) = \left(0,\thetaHeteroclinicConnectionCollisionmu\left(\mcGeheeTimeCollision;\initialtheta\right),\alphaHeteroclinicConnectionCollisionmu\left(\mcGeheeTimeCollision;\initialtheta\right)\right)
    \end{aligned}
    \end{equation}
    such that

    \begin{equation*}
        \begin{aligned}
            \thetaHeteroclinicConnectionCollisionmu\left(\mcGeheeTimeCollision;\initialtheta\right) &= \initialtheta + \pi +2\alphaHeteroclinicConnectionCollisionmu\left(\mcGeheeTimeCollision;\initialtheta\right) \\
            \alphaHeteroclinicConnectionCollisionmu\left(\mcGeheeTimeCollision;\initialtheta\right) &= 2\tan^{-1}\left(\tanh\left(\frac{\constantEquilibrium \tau}{4}\right)\right).
        \end{aligned}
    \end{equation*}
\item $W_\mu^s(S^-)$ and  $W_\mu^u(S^+)$ belong to $\mathcal{M}\setminus \collisionManifold$. When  $\mu=0$, in coordinates $(r,\theta, v,u)$, $W_0^u(S^+)$  can be parameterized by its trajectories as \begin{equation}\label{eqn: Unpertubed solution in regularized coordinates}
\left(\tilde r_h(\tau), \tilde \theta_h(\tau;\initialtheta),\tilde  v_h(\tau), \tilde u_h(\tau)\right) = \left(\constantValue e^{\sqrt{2}\tau}, \initialtheta - e^{\frac{3}{\sqrt{2}}\tau}, \sqrt{2}, -\constantValue^{\frac 3 2}e^{\frac{3}{\sqrt{2}}\tau}\right)
\end{equation}
with $\initialtheta\in\mathds{T}$ and $\tau \in \mathds{R}$ satisfying
\begin{equation*}
    \underset{\tau \to -\infty}{\lim} \left(\tilde r_h(\tau),\tilde \theta_h(\tau;\initialtheta), \tilde v_h(\tau), \tilde u_h(\tau)\right) = (0,\initialtheta,\sqrt{2},0) \in S_{\initialtheta}^+.
\end{equation*}
Symmetry \eqref{eqn: Symmetry of the PCRTBP in synodical polar coordinates centered at P1} gives us an analogous result for $W_0^s(S^-)$.
\end{itemize}
\end{lemma}



\begin{proof}
    The proof of the first item is a direct consequence of the expression of the differential of the vector field $F_\mu$ associated to  \eqref{eqn:Motion in coordinates (s,theta,alpha) McGehee Collision} evaluated at the equilibrium points $S_{\overline{\theta}}^{\pm}$ in \eqref{eqn: Circles of equilibrium points for mu neq 0}, which is given by
    \begin{equation*}
    D\vectorFieldMcGeheePolarCoordinatesmu(\circlesSpm_{\overline{\theta}}) = \begin{pmatrix} \pm\frac{\constantEquilibrium}{2} & 0 & 0\\ 0&0&\mp \constantEquilibrium \\ 0&0&\mp \frac{\constantEquilibrium}{2}\end{pmatrix},
    \end{equation*}
    where we recall that $\constantEquilibrium = \sqrt{2(1-\mu)}$, and whose corresponding eigenvalues are 
    \begin{equation*}\label{eqn:Eigenvalues of the motion (r,theta,alpha)}
   \eigenvaluerSpmGeneralVectorField = \pm \frac{\constantEquilibrium}{2}, \;\; \eigenvalueThetaSpmGeneralVectorField = 0, \;\; \eigenvalueAlphaSpmGeneralVectorField = \mp \frac{\constantEquilibrium}{2}.
\end{equation*}
The proof of the second item follows from integrating \eqref{eqn:Dynamics on the collision} and the third item from integrating \eqref{eqn:Motion regularized McGehee Collision} for $\mu=0$.
\end{proof}
\begin{remark}\label{remark: definition of the invariant manifolds of collision in rotating polar coordinates centered at P1}
The definitions of $\stableManifoldSminusmu$ and $\unstableManifoldSplusmu$ can be translated to coordinates $(r,\theta,R,\Theta)$ by means of the changes \eqref{eqn:McGehee map Collision} and \eqref{eqn:polar change of coordinates for u,v into rho alpha}. Abusing the notation, we denote the collision and ejection manifolds as
\begin{equation}\label{eqn: definition of the invariant manifolds of collision in rotating polar coordinates centered at P1}
\begin{aligned}
    \stableManifoldSminusmu =& \Bigg\{(r,\theta,R,\Theta) \in \mathds{R}^+ \times \mathds{T} \times \mathds{R}^2 \colon \exists t_* = t_*(r,\theta,R,\Theta) >0 \text{ such that }  \underset{t \to t_*^-}{\lim} \Phi_t^r(r,\theta,R,\Theta) = 0,\\
    &\underset{t \to t_*^-}{\lim}\Phi_t^R(r,\theta,R,\Theta) = - \infty\Bigg\},\\
    &\\
    \unstableManifoldSplusmu =& \Bigg\{(r,\theta,R,\Theta) \in \mathds{R}^+ \times \mathds{T} \times \mathds{R}^2 \colon \exists t_* = t_*(r,\theta,R,\Theta)<0 \text{ such that } \underset{t \to t_*^+}{\lim} \Phi_t^r(r,\theta,R,\Theta) = 0, \\
    &\underset{t \to t_*^+}{\lim}\Phi_t^R(r,\theta,R,\Theta) = + \infty\Bigg\},
\end{aligned}
\end{equation}
where $\Phi_t$ refers to the flow of the equations of motion associated to the Hamiltonian $\mathcal{H}$ in \eqref{eqn: Hamiltonian function in rotating polar coordinates centered at $P_1$}.

We stress that, although invariant, they are not stable and unstable manifolds of any invariant objects since $\circleSplus$ and $\circleSminus$ collapse to the singular set $\{r=0\}$.

\end{remark}

\subsection{The unperturbed case $\mu = 0$}\label{subsec: Flow of the 2-body rotating problem}

As we will see in Sections \ref{sec: Proof of Theorem for ECO orbits} and \ref{sec: Proof of Theorem parabolic orbits}, both proofs of Theorems \ref{thm: Existence of infinite sequence of ECO orbits for small energies} and \ref{thm: Existence of parabolic, oscillatory and periodic orbits for small energies} are based on the analysis of the invariant manifolds of infinity and collision respectively. To this end, the purpose of this section is to study them when  $\mu = 0$. Since both synodical polar coordinates \eqref{eqn: Change from synodical cartesian to synodical polar centered at CM} and \eqref{eqn: Change from synodical cartesian to synodical polar centered at P1} are identical when $\mu = 0$, in this section we will use the notation for the synodical polar coordinates $(r,\theta,R,\Theta)$ to study the dynamics.

When $\mu = 0$, the PCRTBP in synodical polar coordinates is defined by the integrable Hamiltonian 
\begin{equation}\label{eqn:Hamiltonian rotating polar coordinates for mu = 0}
    \hamiltonianPolarRotatingCoordinatesCenteredSun(r,\theta,R,\Theta;0) = \frac{1}{2} \left(R^2 + \frac{\Theta^2}{r^2}\right) - \frac 1 r - \Theta.
\end{equation}
The ejection and collision orbits of this Hamiltonian (see Definition \ref{def: Ejection-Collision orbits in cartesian synodical coordinates}) belong to $\{\Theta = 0\}$ and, at the energy level $\mathcal{H}=0$, correspond to ``heteroclinic connections'' between $\circleSpm$ in \eqref{eqn: Circles of equilibrium points for mu neq 0} and $\Alpha_0$ in \eqref{eqn: Definition of Alpha_Theta0 in mcgehee coordinates at infinity}.


\begin{lemma}\label{lemma: Flow of the 2 body problem}
The stable manifold $\stableManifoldInftymuO$ of the system associated to Hamiltonian \eqref{eqn:Hamiltonian rotating polar coordinates for mu = 0} can be written as

\begin{equation}
\stableManifoldInftymuO = \Big\{\orbitPolarRotatingUnperturbedProblemRpos(t,\initialtheta)\colon t >0, \initialtheta \in \mathds{T}\Big\}\quad \text{and}\quad \unstableManifoldInftymuO = \Big\{\orbitPolarRotatingUnperturbedProblemRneg(t,\initialtheta)\colon t<0,\initialtheta \in \mathds{T}\Big\}
\end{equation}
where the trajectories $\orbitPolarRotatingUnperturbedProblemRpos(t;\initialtheta)$ and $\orbitPolarRotatingUnperturbedProblemRneg(t;\initialtheta)$ are given by

\begin{equation}\label{eqn: Parabolic ejection and collision orbit unpertubed case}
\begin{aligned}
    \orbitPolarRotatingUnperturbedProblemRpos(t;\initialtheta)=& \left(\rOrbitPolarRotatingUnperturbedProblemRpos(t),\thetaOrbitPolarRotatingUnperturbedProblemRpos(t;\initialtheta),\ROrbitPolarRotatingUnperturbedProblemRpos(t),\ThetaOrbitPolarRotatingUnperturbedProblemRpos(t)\right) = \left(\constantValue t^{\frac 2 3}, \initialtheta-t,\sqrt{\frac{2}{\constantValue}}\frac{1}{t^{\frac 1 3}},0\right)\quad &\forall \;t>0,\\
    \orbitPolarRotatingUnperturbedProblemRneg(t;\initialtheta) =& \left(\rOrbitPolarRotatingUnperturbedProblemRneg(t),\thetaOrbitPolarRotatingUnperturbedProblemRneg(t;\initialtheta),\ROrbitPolarRotatingUnperturbedProblemRneg(t),\ThetaOrbitPolarRotatingUnperturbedProblemRneg(t)\right) = \left(\constantValue t^{\frac 2 3}, \initialtheta-t,-\sqrt{\frac{2}{\constantValue}}\frac{1}{|t|^{\frac 1 3}},0\right)\quad &\forall \;t<0,
\end{aligned}
\end{equation}
where $\constantValue= \frac{3^{\frac 2 3}}{2^{\frac 1 3}}$.

Therefore, 
\begin{equation}\label{eqn: Relation between the stable and unstable manifolds of infinity and collision for mu=0}
    \begin{aligned}
        \unstableManifoldInftymuO &= W_0^s(\circleSminus) = \underset{\initialtheta\in \mathds{T}}{\bigcup} W_0^s(\equilibriumPointSminus),\\
        \stableManifoldInftymuO &= W_0^u(\circleSplus) = \underset{\initialtheta\in \mathds{T}}{\bigcup} W_0^u(\equilibriumPointSplus).
    \end{aligned}
\end{equation}
where $W_0^{s,u}(\circlesSpm)$ are defined in \eqref{eqn: definition of the invariant manifolds of collision in rotating polar coordinates centered at P1} 

\end{lemma}

\subsection{The perturbed invariant manifolds of infinity}\label{subsec: Invariant manifolds of infinity in synodical polar coordinates}


The next proposition gives a parameterization, in the synodical polar coordinates \eqref{eqn: Change from synodical cartesian to synodical polar centered at CM}, of the invariant manifolds $W_{\mu}^{s,u}(\Alpha_{\hat{\Theta}_0})$ close to the unperturbed ones obtained in Lemma \ref{lemma: Flow of the 2 body problem}, for $\mu>0$ small enough and $\hat{\Theta}_0$ of order $\mu$. We provide the statement for the stable manifold. One can deduce an analogous result for the unstable one using that the system is reversible with respect to \eqref{eqn: Symmetry of the PCRTBP in synodical polar coordinates centered at CM}. 

\begin{proposition}\label{proposition: Parameterization of the invariant manifolds of infinity in synodical polar coordinates}
   Fix $a,b$ with $a<b \in \mathds{R}$, $\hat{D}\in (0,1/2)$, $\hat{W} > 0$ and  $\overline{\Theta}_0 \in [a,b]$. Then, there exists $\mu_0 >0$ small enough such that for $0<\mu<\mu_0$ and  $\hat{\Theta}_0 =\mu\overline{\Theta}_0$, a subset of the invariant manifold $\stableManifoldInftymu$, which we denote by $\reducedStableManifoldInfinitymu$, can be written as
    
    \begin{equation}\label{eqn: Parameretization of the reduced invariant manifolds of infinity}
        \reducedStableManifoldInfinitymu = \Big\{(\hat r,\hat \theta, \hat R, \hat \Theta) = \hat{\Upsilon}^+(\hat{w},\hat{\theta},\hat{\Theta}_0)\colon (\hat{w},\hat{\theta})\in  \hat{\mathcal{D}}^{+}_{\hat{W},\hat{D}}\Big\}
    \end{equation}
    where $\hat{\Upsilon}^+$ is a $C^\infty$ function of the form 
    \begin{equation}\label{eqn: Relation between the parameterization of the invariant manifolds of infinity and mu=0}
        \hat{\Upsilon}^+(\hat{w},\hat{\theta},\hat{\Theta}_0) = \left( \constantValue\hat{w}^{\frac 2 3} ,\hat{\theta},\hat{R}_\infty^s(\hat{w},\hat{\theta},\hat{\Theta}_0), \hat{\Theta}_\infty^s(\hat{w},\hat{\theta},\hat{\Theta}_0)\right),
    \end{equation}
    with
    \begin{equation}\label{eqn: Parameterization of the angular momentum for WinftysMu}
    \begin{aligned}
        \hat{R}_\infty^s(\hat{w},\hat{\theta},\hat{\Theta}_0) =&\sqrt{\frac{2}{\constantValue}}\frac{1}{\hat{w}^{\frac 1 3}} + \mathcal{O}_1(\mu)\\
       \hat{\Theta}_\infty^s(\hat{w},\hat{\theta},\hat{\Theta}_0) =& \hat{\Theta}_0- \mu\constantValue \bigintss_{+\infty}^{\hat{w}}\left( \frac{\hat{s}^{\frac 2 3}\sin(\hat{\theta}+\hat{w}-\hat{s})}{\left(\constantValue^2 \hat{s}^{\frac 4 3}- 2\constantValue\hat{s}^{\frac 2 3}\cos(\hat{\theta}+\hat{w}-\hat{s}) + 1\right)^{\frac 3 2}}\right.\\
       &\left.- \frac{\sin(\hat \theta + \hat w - \hat s)}{\constantValue^3 \hat{s}^{\frac 4 3}}\right)\;d\hat{s} + \mathcal{O}_2(\mu)\\
    \end{aligned}
    \end{equation}
    where $\constantValue= \frac{3^{\frac 2 3}}{2^{\frac 1 3}}$, 
    and the domain $\hat{\mathcal{D}}_{\hat{W},\hat{D}}^+$ is defined as
    \begin{equation*}
         \hat{\mathcal{D}}^{+}_{\hat{W},\hat{D}} = \left[\hat{W},+\infty\right) \times \hat{I}_{\hat{D}}^+(\hat{w})
    \end{equation*}
    with

    \begin{equation}\label{eqn: Domain Ihat}
    \begin{aligned}
        \hat{I}_{\hat{D}}^+(\hat{w}) = \begin{cases} 
        \mathds{T} - \left(\frac{\sqrt{2}}{3}-\hat{w} - \hat{D}, \frac{\sqrt{2}}{3} - \hat{w} + \hat{D}\right)\quad &\text{if}\quad \hat{W} \leq \hat{w} \leq \frac{\sqrt{2}}{3} (1-\mu)^{\frac 3 2} + \hat{D},\\
        \mathds{T}\quad &\text{if}\quad \hat{w} > \frac{\sqrt{2}}{3} (1-\mu)^{\frac 3 2} + \hat{D}.
        \end{cases}
    \end{aligned}
    \end{equation}

\end{proposition}

\begin{proof}
    We denote  by $\hat{F}$ the vector field associated to the Hamiltonian  \eqref{eqn:Hamiltonian Polar Rotating Coordinates centered at CM}, which can be written as
    \begin{equation*}
        \hat{F}(\hat{r},\hat{\theta},\hat{R},\hat{\Theta}) = \hat{F}_0(\hat{r},\hat{\theta},\hat{R},\hat{\Theta}) + \hat{F}_1(\hat{r},\hat{\theta},\hat{R},\hat{\Theta})
    \end{equation*}
    such that
    
    \begin{equation}\label{eqn: hat(F0) and hat(F1)}
        \begin{aligned}
            \hat{F}_0(\hat{r},\hat{R},\hat{\Theta}) =& \left(\hat{R},\frac{\hat{\Theta}}{\hat{r}^2}-1, \frac{\hat{\Theta}^2}{\hat{r}^3} - \frac{1}{\hat{r}^2}, 0\right)^T\\
            \hat{F}_1(\hat{r},\hat{\theta}) =& \left(0,0,\partial_{\hat{r}}V(\hat{r},\hat{\theta};\mu), \partial_{\hat{\theta}}V(\hat{r},\hat{\theta};\mu)\right)^T,
        \end{aligned}
    \end{equation}
    where $V(\hat{r},\hat{\theta};\mu)$ is the potential in \eqref{eqn:Perturbed Potential in polar rotating coordinates centered at CM}.
        The vector field $\hat{F}$ is regular at
    is regular for $\hat{r}\geq \hat{r}_0, \hat{\theta} \in \mathds{T}$ for any $\hat{r}_0 > 1$.

    Observe that for $\mu=0$, the 2-dimensional invariant manifold 
    $W_\mu^s(\Alpha_{\hat{\Theta}_0})$ can be parameterized 
    as a graph over $\hat r$, $\hat\theta$. For $\hat r>1$, we can make use of the  regularity with respect to parameters proven by McGehee in \cite{MR0362403} to ensure that the perturbed manifold is $\mu$-close to the unperturbed one given in Lemma \ref{lemma: Flow of the 2 body problem}, and therefore still a graph over the same variables. For convenience, we change the parameter $\hat r$ into  $\hat w$ by
    \[
\hat{r}= \constantValue\hat{w}^{\frac 2 3}, \quad \hat{w} > \hat{W}.
    \]
    Then, the graph parameterization of the perturbed invariant manifold is given by 
        \begin{equation}\label{eqn: Parameterization of the stable manifold of infinity r0}
       \hat{r}(\hat{w})= \constantValue\hat{w}^{\frac 2 3} ,\quad \hat{\theta} = \hat{\theta},\quad
        \hat{R}(\hat{w},\hat{\theta})= \sqrt{\frac{2}{\constantValue}}\frac{1}{\hat{w}^{\frac 1 3}} + \mathcal{O}_1(\mu),\quad \hat{\Theta}(\hat{w},\hat{\theta}) = \mathcal{O}_1(\mu).
    \end{equation}
    Hence, the proof consists on extending the previous parameterization of $\stableManifoldInftymu$ from $\{\hat{r}=\hat{r}_0>1\}$
    to a section $\{\hat{r} = \hat{r}_F\}$ satisfying $\mu < \hat{r}_F < 1-\mu$.

    By \eqref{eqn: Parameterization of the stable manifold of infinity r0}, we can fix $\hat{w}_0> \constantValue^{-\frac 3 2}$ such that $\hat{r}_0 = \hat{r}(\hat{w}_0)$ and we consider the set of points of the form
        \begin{equation}\label{eqn: Init conditions in hat(r0)}
       \stableManifoldInftymu \bigcap \{\hat{w} = \hat{w}_0\} = \Bigg\{\left(\hat{r}_0,\hat{\theta}_0,\hat{R}(\hat{w}_0,\hat{\theta}_0),\hat{\Theta}(\hat{w}_0,\hat{\theta}_0)\right),\hat{\theta}_0\in\mathds{T}\Bigg\}.
    \end{equation}
    Recall that in Section \ref{subsec: Flow of the 2-body rotating problem}, for $\mu = 0$ we have analyzed the invariant manifolds of infinity for the Hamiltonian $\mathcal{\hat{H}}$ in \eqref{eqn:Hamiltonian Polar Rotating Coordinates centered at CM} restricted to  the plane $\{\hat{\mathcal{H}}=\hat{\Theta}= 0\}$. Then, if we denote   
    \begin{equation*}
        z_0(t;\hat{\theta}_0) = \hat{\Phi}_0\left(t, \left(\hat{r}(\hat{w}_0),\hat{\theta}_0,\hat{R}(\hat{w}_0,\hat{\theta}_0),\hat{\Theta}(\hat{w}_0,\hat{\theta}_0)\right)\right) = \left(z_0^{\hat{r}},z_0^{\hat{\theta}},z_0^{\hat{R}},z_0^{\hat{\Theta}}\right),
    \end{equation*}
    we have that
    \begin{equation}\label{eqn: Unperturbed z0}
    \begin{aligned}
        z_0^{\hat{r}}(t) = \left(\hat{r}(\hat{w}_0)^{\frac 3 2} + \frac{3t}{\sqrt{2}}\right)^{\frac 2 3},\quad z_0^{\hat{\theta}}(t;\hat{\theta}_0) = \hat{\theta}_0 - t,\quad z_0^{\hat{R}}(t) = \sqrt{\frac{2}{z_0^{\hat{r}}(t)}},\quad z_0^{\hat{\Theta}}(t;\hat{\theta}_0) = 0,
    \end{aligned}
    \end{equation}
    where $\hat{\Phi}_0$ is the flow of the vector field $\hat{F}_0$ in \eqref{eqn: hat(F0) and hat(F1)}. 
    
   Then we can  compute $(t_0,\hat{\theta}_0^0)$ such that
   \begin{equation*}
       z_0^r(t_0;\hat{\theta}_0^0) = 1,\quad z_0^\theta(t_0;\hat{\theta}_0^0) = 0,
   \end{equation*}
    which corresponds to the position of the primary $P_2$ (when $\mu=0)$ and gives
    \begin{equation}\label{eqn: Expression of hat(theta)}
        \hat{\theta}_0^0 = t_0 = -\frac{\sqrt{2}}{3}\left(\hat{r}(\hat{w}_0)^{\frac 3 2} - 1\right)<0.
    \end{equation}
    Taking $0<\hat{d}< \frac 1 4$, we define the following set of ``bad'' initial conditions close to $\hat{\theta}_0^0$

    \begin{equation*}
        B_{\hat{d}}(\hat{\theta}_0^0) = \Big\{\hat{\theta}_0\in \mathds{T}\colon |\hat{\theta}_0-\hat{\theta}_0^0| < \hat{d}\Big\}.
    \end{equation*}
    Any $\hat{\theta}_0 \in B_{\hat{d}}(\hat{\theta}_0^0)$ satisfies that
    \begin{equation*}
        z_0^r(t_0;\hat{\theta}_0) = 1,\quad z_0^{\hat{\theta}}(t_0;\hat{\theta}_0) = \hat{\theta}_0-t_0 \in \left(-\hat{d},\hat{d}\right).
    \end{equation*}
    That is, the unperturbed flow sends $ B_{\hat{d}}(\hat{\theta}_0^0)$ to a neighborhood of $P_2$ at the section $\hat r=1$ of the form
    \begin{equation}\label{eqn: neighborhood of collision for mu=0}
        B_{\hat{d}}(0) = \Big\{\hat{\theta}\in\mathds{T}\colon |\hat{\theta}|<\hat{d}\Big\}.
    \end{equation}
    Now we consider $\mu > 0 $ small enough. In this case, for each point in \eqref{eqn: Init conditions in hat(r0)}, we denote its trajectory by
    \begin{equation}\label{eqn: zmu}
        z_\mu(t;\hat{\theta}_0) = \hat{\Phi}_\mu\left(t, \left(\hat{r}(\hat{w}_0),\hat{\theta}_0,\hat{R}(\hat{w}_0,\hat{\theta}_0),\hat{\Theta}(\hat{w}_0,\hat{\theta}_0)\right)\right) = \left(z_\mu^{\hat{r}},z_\mu^{\hat{\theta}},z_\mu^{\hat{R}},z_\mu^{\hat{\Theta}}\right),
    \end{equation}
    where $\hat{\Phi}_\mu$ is the flow of the vector field $\hat F$ in \eqref{eqn: hat(F0) and hat(F1)}.
    
    Note that, for $\mu = 0$, any point in $\stableManifoldInftymuO$ has $\frac{d}{dt}\hat{r} = \hat{R}>0$. Since for $\hat{\theta}\notin B_{\hat{d}}(0)$ the vector field $\hat{F}$ in \eqref{eqn: hat(F0) and hat(F1)} is regular with respect to $\mu$, we have the same behaviour at $\stableManifoldInftymu$ for $\mu > 0$ small enough. 
    
    Therefore, if we fix $\hat{D}= 2\hat{d} < \frac 1 2$ and we consider a set of initial conditions of the form

    \begin{equation*}
        B_{\hat{D}}^{\mathrm{init}}(\hat{\theta}_0^0) = \Big\{\hat{\theta} \in \mathds{T}\colon |\hat{\theta}_0 - \hat{\theta}_0^0| < \hat{D}\Big\}
    \end{equation*}
    then, for every $\hat{\theta}_0 \notin B_{\hat{D}}^{\mathrm{init}}(\hat{\theta}_0^0)$, there exists $t_F^\mu (\hat{\theta}_0) < 0$ such that $z_\mu^{\hat{r}}(t_F^\mu(\hat{\theta}_0),\hat{\theta}_0) = \hat{r}_F < 1-\mu$. In particular, if we denote by $t_\mu^*(\hat{\theta}_0) \in \left[t_F^\mu(\hat{\theta}_0),0\right]$ the time such that $z_\mu^{\hat{r}}(t_\mu^*(\hat{\theta}_0),\hat{\theta}_0) = 1-\mu$, we know that $z_\mu^{\hat{\theta}}(t_\mu^*(\hat{\theta}_0),\hat{\theta}_0) = \hat{\theta} \notin B_{\hat{d}}(0)$.
    
    Hence, the perturbed vector field $\hat{F}_1$ in \eqref{eqn: hat(F0) and hat(F1)} is uniformly bounded. Namely, there exist $C_1,C_2,C_3>0$ such that, for $t\in[t_F^\mu(\hat{\theta}_0),0]$ and $\hat{\theta}_0\notin B_{\hat{D}}(\hat{\theta}_0^0)$,
    \begin{equation*}
        \left|\left|\hat{F}_1\left(z_\mu^{\hat{r}}(t;\hat{\theta}_0), z_\mu^{\hat{\theta}}(t;\hat{\theta}_0)\right)\right|\right| \leq C_1 + \frac{C_2}{\left(\hat{r}(\hat{w}_0)-\mu\right)^3} + \mu \frac{C_3}{\left(1-\cos \hat{d}\right)^{\frac 3 2}}.
    \end{equation*}
    This implies that, for $\hat{\theta}_0 \notin B_{\hat{D}}(\hat{\theta}_0^0)$ we have

    \begin{equation}\label{eqn: zmu = z0+ O(mu)}
        z_\mu(t;\hat{\theta}_0) = z_0(t;\hat{\theta}_0) + \mathcal{O}_1(\mu).
    \end{equation}
    Therefore, the parameterization \eqref{eqn: Relation between the parameterization of the invariant manifolds of infinity and mu=0} is well-defined for $\hat{\theta}$ defined in \eqref{eqn: Domain Ihat} until the section $\{\hat{r}=\hat{r}_F\}$.

    The second part of the proof comes as a result of the fundamental theorem of calculus, along with the fact that 
    \begin{equation*}
        \underset{t\to +\infty}{\lim} z_\mu^{\hat{\Theta}}(t,\hat{\theta}_0) = \hat{\Theta}_0
    \end{equation*}
    by definition of $\reducedStableManifoldInfinitymu$. The $\mu$-expansion for the equation of $\frac{d}{dt}\hat\Theta$, which corresponds to the fourth component of $\hat{F}_1$ in \eqref{eqn: hat(F0) and hat(F1)}, the expression of the potential $\hat{V}$ in \eqref{eqn:Perturbed Potential in polar rotating coordinates centered at CM} and the approximation \eqref{eqn: zmu = z0+ O(mu)} allows us to write the $\mu$-expansion of the $\hat{\Theta}$-component of the curve $\hat{\Upsilon}^+$ in \eqref{eqn: Parameterization of the angular momentum for WinftysMu}.
\end{proof}

Although in Proposition \ref{proposition: Parameterization of the invariant manifolds of infinity in synodical polar coordinates} we give a parameterization of $\reducedStableManifoldInfinitymu$, in order to prove the main results stated in Section \ref{subsec: Main results}, we need to compare it with the invariant manifolds of collision (see Section \ref{sec: Asymptotic formula for the distance between the manifolds}) in a common set of coordinates. We do the comparison in polar coordinates centered at $P_1$ \eqref{eqn: Change from synodical cartesian to synodical polar centered at P1} and therefore we must reparameterize the invariant manifold of infinity $\reducedStableManifoldInfinitymu$ obtained in Proposition \ref{proposition: Parameterization of the invariant manifolds of infinity in synodical polar coordinates}.


\begin{proposition}\label{proposition: Parameterization of the invariant manifolds of infinity in rotating polar coordinates centered at P1}
    Fix $a,b \in \mathds{R}$ with $a<b$, $D\in (0,1/2)$, $W > 0$ and $\overline{\Theta}_0 \in [a,b]$. Then there exists $\mu_0 >0$ small enough such that for $0<\mu<\mu_0$ and $\hat{\Theta}_0 =\mu\overline{\Theta}_0$, a subset of the stable manifold $\stableManifoldInftymu$, which we denote by $\reducedStableManifoldInfinitymu$, can be written as

    \begin{equation*}
        \reducedStableManifoldInfinitymu =\Big\{(r,\theta,R,\Theta) = \Upsilon^+(w,\theta,\hat{\Theta}_0)\colon (w,\theta) \in \mathcal{D}^+_{W,D}\Big\},
    \end{equation*}
    where $\Upsilon^+$ is a $C^\infty$-function of the form
    \begin{equation}\label{eqn: Parameterization of the stable manifold of infinity in terms of (w,theta)}
        \Upsilon^+(w,\theta,\hat{\Theta}_0) = \left( \constantValue{w}^{\frac 2 3} ,\theta,R_\infty^s(w,\theta,\hat{\Theta}_0), \Theta_\infty^s(w,\theta,\hat{\Theta}_0)\right),
    \end{equation}
    where
    \begin{equation}\label{eqn: Angular momentum of Ws(infty) as a graph of (w,theta)}
    \begin{aligned}
        R_\infty^s(w,\theta,\hat{\Theta}_0) &=\sqrt{\frac{2}{\constantValue}}\frac{1}{{w}^{\frac 1 3}} + \mathcal{O}_1(\mu),\\
        \Theta_\infty^s(w,\theta,\hat{\Theta}_0) =&\hat{\Theta}_0 -\mu \constantValue\bigintss_{+\infty}^{\timeParametrization} \frac{s^{\frac{2}{3}} \sin(\theta+ w -s)}{\left(1+\constantValue^2s^{\frac 4 3} - 2\constantValue s^{\frac 2 3}\cos(\theta + w -s)\right)^{\frac{3}{2}}}\; ds \\
        &-\mu\sqrt{\frac{2}{\constantValue}}\bigintss_{+\infty}^{\timeParametrization} \frac{\cos(\theta + w-s)}{s^{\frac{1}{3}}}\; ds + \mathcal{O}_2(\mu),
    \end{aligned}
    \end{equation}
    with $\constantValue=  \frac{3^{\frac 2 3}}{2^{\frac 1 3}}$, and the domain $\mathcal{D}^+_{W,D}$ is defined as    
    \begin{equation}\label{eqn: parameter region (w,theta)}
        \mathcal{D}^+_{W,D} =  \left[W,+\infty\right) \times I_D^+(w)
    \end{equation}
    with
    \begin{equation}\label{eqn: domain Ip}
    \begin{aligned}
        I_D^+(w) = \begin{cases} \mathds{T} - \left(\frac{\sqrt{2}}{3}-w-D,\frac{\sqrt{2}}{3}-w+D\right)\quad &\text{if}\quad W\leq w \leq \frac{\sqrt{2}}{3}+D,\\\mathds{T}\quad &\text{if } w > \frac{\sqrt{2}}{3}+D. \end{cases}
    \end{aligned}
    \end{equation}
    Analogously, due to the symmetry \eqref{eqn: Symmetry of the PCRTBP in synodical polar coordinates centered at P1}, a subset of the unstable manifold $\unstableManifoldInftymu$, which we denote by $\reducedUnstableManifoldInfinitymu$, can be written as

    \begin{equation}
        \reducedUnstableManifoldInfinitymu = \Big\{\Upsilon^-(w,\theta)\colon (w,\theta) \in \mathcal{D}^-_{W,D}\Big\},
    \end{equation}
where    \begin{equation}\label{eqn: Parameterization of the unstable manifold of infinity in terms of (w,theta)}
        \Upsilon^-(w,\theta,\hat{\Theta}_0) = \Upsilon^+(-w,-\theta,\hat{\Theta}_0)
    \end{equation}
    and the domain $\mathcal{D}^-_{W,D}$ is defined as 
    \begin{equation*}
        \mathcal{D}^-_{W,D} = \left(-\infty,-W\right] \times I_D^-(w)
    \end{equation*}
    with
    \begin{equation}\label{eqn: domain Im}
    \begin{aligned}
        I_D^-(w) = \begin{cases}\mathds{T} - \left(-\frac{\sqrt{2}}{3} + w- D, - \frac{\sqrt{2}}{3} -w +D\right)\quad &\text{if } -\frac{\sqrt{2}}{3} - D \leq w \leq -W\\\mathds{T}\quad &\text{if }w < - \frac{\sqrt{2}}{3} - D \end{cases}
    \end{aligned}
    \end{equation}
    In particular, the $\Theta$-component of $\Upsilon^-$ can be written as 

    \begin{equation}\label{eqn: Angular momentum of Wu(infty) as a graph of (w,theta)}
        \Theta_\infty^u(w,\theta,\hat{\Theta}_0) = \Theta_\infty^s(-w,-\theta,\hat{\Theta}_0).
    \end{equation}



     

\end{proposition}

\begin{proof}
    We consider the transformation from rotating polar coordinates centered at $P_1$ \eqref{eqn: Change from synodical cartesian to synodical polar centered at P1} to the ones centered at the center of mass \eqref{eqn: Change from synodical cartesian to synodical polar centered at CM}
    \begin{equation}\label{eqn: Transformation from P1-CM}
    \begin{aligned}
       r &=\sqrt{\hat{r}^2 + 2\hat{r}\mu\cos\hat{\theta}+\mu^2}\\
       \theta &=\tan^{-1}\left(\frac{\hat{r}\sin\hat\theta}{\hat r \cos\hat{\theta}+\mu}\right)\\
       R&= \hat{R} \frac{r-\mu\cos\theta}{\sqrt{r^2-2\mu r\cos\theta+\mu^2}} - \mu \hat{\Theta} \frac{r\sin\theta}{r^2-2\mu r\cos\theta + \mu^2},\\
      \Theta&= \mu \hat{R} \frac{r\sin\theta}{\sqrt{r^2-2\mu r\cos\theta + \mu^2}} + \hat{\Theta} \frac{r(r-\mu\cos\theta)}{r^2-2\mu r\cos\theta + \mu^2}.
    \end{aligned}
    \end{equation}
    This transformation satisfies

    \begin{equation}\label{eqn: Approximation of the transformation rotating CM P1}
        \begin{aligned}
            r&= \hat r + \mu\cos\hat\theta + \mathcal{O}\left(\frac{\mu^2}{\hat r}\right),\\
            \theta &= \hat \theta - \mu \frac{\sin\theta}{\hat r} + \mathcal{O}\left(\frac{\mu^2}{\hat{r}^2}\right),\\
            R &= \hat{R} + \mathcal{O}_1(\mu),\\
            \Theta &=\hat \Theta\left(1 + \frac{\mu \cos\theta}{r}\right) + \mu \hat{R}\sin\theta + \mathcal{O}_2(\mu).
        \end{aligned}
    \end{equation}
%
For any $\hat{W}>0$, Proposition \ref{proposition: Parameterization of the invariant manifolds of infinity in synodical polar coordinates} gives us the parameterization of the invariant manifold in \eqref{eqn: Relation between the parameterization of the invariant manifolds of infinity and mu=0} such that $\hat{r} = \constantValue \hat{w}^{\frac 2 3}$. Therefore, for $\hat{w}> \hat{W}$, equation \eqref{eqn: Approximation of the transformation rotating CM P1} leads to
\[
r=\constantValue\hat{w}^{\frac 2 3}+\mathcal{O}_1(\mu).
\]
To have a graph parameterization analogous to the one in Proposition \ref{proposition: Parameterization of the invariant manifolds of infinity in synodical polar coordinates}, we define $w$ such that $r=\constantValue w^{\frac 2 3}$. Substituting in \eqref{eqn: Transformation from P1-CM} we obtain
%
\begin{equation}\label{eqn: Relation w hat(w)}
    \begin{aligned}
        \hat{w}(w,\theta) = \frac{\left(\constantValue^2 w^{\frac 4 3} - 2\constantValue\mu w^{\frac 2 3}\cos \theta +\mu^2\right)^{\frac 3 4}}{\constantValue^{\frac 3 2}} =  w + \mathcal{O}_1(\mu).
    \end{aligned}
\end{equation}
Hence, given $W > 0$, one can find $\hat{W} = W + \mathcal{O}_1(\mu)$ so that $\hat{w}> \hat{W}$ when $w > W$. From this result and \eqref{eqn: Transformation from P1-CM}, we can also relate the parameter $\hat \theta$ as follows

\begin{equation}\label{eqn: Relation w hat(theta)}
    \hat \theta (w,\theta) = \tan^{-1}\left(\frac{\constantValue w^{\frac 2 3}\sin \theta}{\constantValue w^{\frac 2 3}\cos \theta - \mu}\right) = \theta + \mathcal{O}_1(\mu).
\end{equation}
Regarding now $\hat{R},\hat{\Theta}$ as the parameterizations of the invariant manifold $\stableManifoldInftymu$ in \eqref{eqn: Parameterization of the angular momentum for WinftysMu}, substituting them in \eqref{eqn: Approximation of the transformation rotating CM P1} and changing the parameters $(\hat{w},\hat{\theta})$ to $(w,\theta)$ using the relations in \eqref{eqn: Relation w hat(w)} and \eqref{eqn: Relation w hat(theta)}, we obtain that
\begin{equation*}
\begin{aligned}
    R_\infty^s(w,\theta,\hat{\Theta}_0) =& \sqrt{\frac{2}{\constantValue}}\frac{1}{w^{\frac 1 3}} + \mathcal{O}_1(\mu)\\
    \\
    \Theta_\infty^s(w,\theta,\hat{\Theta}_0) =& \hat{\Theta}_\infty^s\left(\hat{w}(w,\theta),\hat{\theta}(w,\theta),\hat{\Theta}_0\right)\left(1+\frac{\mu\cos\theta}{\constantValue w^{\frac 2 3}}\right) + \mu \hat R_\infty^s\left(\hat{w}(w,\theta),\hat{\theta}(w,\theta),\hat{\Theta}_0\right) \sin\theta \\
    &+ \mathcal{O}_2(\mu) = \hat{\Theta}_\infty^s(w,\theta,\hat{\Theta}_0) + \mu \sqrt{\frac{2}{\constantValue}}\frac{\sin\theta}{w^{\frac 1 3}} + \mathcal{O}_2(\mu),
\end{aligned}
\end{equation*}
which implies
\begin{equation*}
\begin{aligned}
    \hat{\Theta}_\infty^s(w,\theta,\hat{\Theta}_0) =& \hat{\Theta}_0- \mu\constantValue \bigintss_{+\infty}^{w}\frac{s^{\frac 2 3}\sin(\theta + w- s)}{\left(1+\constantValue^2s^{\frac 4 3} - 2\constantValue s^{\frac 2 3}\cos(\theta + w -s)\right)^{\frac{3}{2}}}ds + \mu \frac{1}{\constantValue^2}\bigintss_{+\infty}^{w}\frac{\sin(\theta + w - s)}{s^{\frac 4 3}}ds\\
    &+ \mathcal{O}_2(\mu).
\end{aligned}
\end{equation*}
Integrating by parts the second integral and using the fact that $\frac{3}{\constantValue^2} = \sqrt{\frac 2 \constantValue}$ lead to \eqref{eqn: Angular momentum of Ws(infty) as a graph of (w,theta)}, completing the proof.
\end{proof}

\subsection{The invariant manifolds of the collision}\label{subsec: Invariant manifolds of collision in synodical polar coordinates}

This section is devoted to prove the following proposition, which gives a parameterization, in polar coordinates centered at $P_1$ (see \eqref{eqn: Change from synodical cartesian to synodical polar centered at P1}), of the invariant manifolds of the collision, that is   $W_\mu^u(S^+)$ and $W_\mu^s(S^-)$  (see Remark \ref{remark: definition of the invariant manifolds of collision in rotating polar coordinates centered at P1}) perturbatively from those  obtained in Lemma \ref{lemma: Flow of the 2 body problem} for $\mu=0$.

\begin{proposition}\label{proposition: Perturbed invariant manifolds of collision in synodical polar coordinates}
    Fix $r^*\in(0,1)$. There exists $\mu_0 > 0$ such that for $0<\mu<\mu_0$, the invariant manifolds $W_\mu^u(S^+)$, $W_\mu^s(S^-)$, written in polar coordinates centered at $P_1$ (see \eqref{eqn: Change from synodical cartesian to synodical polar centered at P1}),  intersect the section $r=r^*$. Moreover, the intersections are graphs over $\theta$ of the  form 
    \begin{equation}\label{def:invmancol_section}
    \begin{split}
(r,\theta,R,\Theta)&=(r^*, \theta, R_{S^+}^{u}(\theta), \Theta_{S^+}^{u}(\theta))\\
(r,\theta,R,\Theta)&=(r^*, \theta, R_{S^-}^{s}(\theta), \Theta_{S^-}^{s}(\theta)),
 \end{split}
    \end{equation}
   which depend smoothly on $\mu$. 
%
    
 Moreover, $\Theta_{\circleSplus}^u$ can be written as
 \begin{equation}\label{def:halfMelnikovCol}
    \begin{aligned}
\Theta_{\circleSplus}^u(\theta) &= \mu  \bigintss_{t(r^*)}^0 \left(\frac{\constantValue s^{\frac 2 3}\sin\left(\theta+t(r^*)-s\right)}{\left(1+\constantValue^2 s^{\frac 4 3} - 2\constantValue s^{\frac 2 3}\cos\left(\theta+t(r^*)-s\right)\right)^{\frac 3 2}} + \sqrt{\frac 2 \constantValue} \frac{\cos\left(\theta+t(r^*)-s\right)}{s^{\frac 1 3}}\right)\;ds + \mathcal{O}_1(\mu^2),
    \end{aligned}
    \end{equation}
    where 
    \[
    t(r^*)=\left(\frac{r^*}{\kappa}\right)^{3/2}\qquad \text{and}\qquad\constantValue = \frac{3^{\frac 2 3}}{2^{\frac 1 3}}
    .\]
  The expression for 
  $\Theta_{\circleSminus}^s(\theta)$ comes from the symmetry \eqref{eqn: Symmetry of the PCRTBP in synodical polar coordinates centered at P1} and is given by
    \begin{equation}\label{eqn: Angular momentum of Ws(S-) as a graph of theta}
\Theta_{\circleSminus}^s(\theta) = \Theta_{\circleSplus}^u(-\theta).
    \end{equation}
\end{proposition}

\begin{proof}
We provide the proof for the unstable manifold, since the statements for the stable  manifold are just a consequence of the symmetry.
We rely on McGehee coordinates to prove this proposition, and consider then the vector field \eqref{eqn:Motion in coordinates (s,theta,alpha) McGehee Collision} and the transverse section $\Sigma_{s^*} = \{s = s^*\}$ with $s^*=\sqrt{r^*}<1$. Note that the vector field is regular for $0<s<s^*$.

The first observation is that $\circleSplus$ is a normally hyperbolic invariant manifold  (by Lemma \ref{lemma: Linear part of motion close to S+-}) and, by Fenichel's theory \cite{MR0343314,MR0426056,MR0287106,MR1278264}, its unstable and stable invariant manifolds are regular and depend also smoothly on $\mu$.  As a consequence, since in the unperturbed case $\mu=0$ (see \eqref{eqn: Unpertubed solution in regularized coordinates}) the invariant manifold intersect the section $s=s^*=\sqrt{r^*}<1$ and, at the section, it is a graph over $\theta$, the same happens for $\mu$ small enough.
Returning to polar coordinates \eqref{eqn: Change from synodical cartesian to synodical polar centered at P1} yields the same result since they are regular away from $\mathcal{S}$. This gives the curve \eqref{def:invmancol_section}.




The second part of the proof is devoted to obtain the expression of $\Theta_{\circleSplus}^u$ in \eqref{def:halfMelnikovCol}. 
We perform a Melnikov-like approach. That is, we consider a point of the form $z(\theta_0)=(r^*, \theta_0, R_{S^+}^{u}(\theta_0), \Theta_{S^+}^{u}(\theta_0))$ (see \eqref{def:invmancol_section}) and describe its $\Theta$-component by a backward integral. 

Since polar coordinates become undefined at collision, we express $\Theta$ in McGehee coordinates \eqref{eqn:McGehee map Collision} as 
\[
\Theta=ur^{1/2}+r^2-\mu r\cos\theta.
\]
Then, by \eqref{eqn:Motion regularized McGehee Collision}, its evolution in these coordinates is described by the equation
\begin{equation*}
    \frac{d}{d\tau}\Theta = \mu \left(r^{\frac 5 2}\sin\theta - \frac{r^{\frac 5 2}\sin\theta}{\left(r^2-2r\cos\theta + 1\right)^{\frac 3 2}} - rv\cos\theta + ru\sin\theta\right).
\end{equation*}
Now, we consider the trajectory departing from the point $z(\theta_0)$ expressed in McGehee coordinates, 
\begin{equation*}
    {\varphi}_\mu(\tau;\theta_0) = (r(\tau;\theta_0), \theta(\tau;\theta_0), v(\tau;\theta_0), u(\tau;\theta_0)).
\end{equation*}
Since $\underset{\tau \to -\infty}{\lim}r(\tau;\theta_0) = 0$ exponentially and therefore  $\underset{\tau \to -\infty}{\lim}\Theta(\tau;\theta_0) = 0$, one has 
\[
\begin{split}
\Theta_{S^+}^{u}(\theta_0) = \mu \int_{-\infty}^0\Bigg(& r(\tau;\theta_0)^{\frac 5 2}\sin\theta(\tau;\theta_0) - \frac{r^{\frac 5 2}\sin\theta(\tau;\theta_0)}{\left(r(\tau;\theta_0)^2-2r(\tau;\theta_0)\cos\theta(\tau;\theta_0) + 1\right)^{\frac 3 2}} \\
&-r(\tau;\theta_0)v(\tau;\theta_0)\cos\theta(\tau;\theta_0) + r(\tau;\theta_0)u(\tau;\theta_0)\sin\theta(\tau;\theta_0)\Bigg)d\tau.
\end{split}
\]
When $\mu=0$, shifting time to fix the initial condition in \eqref{eqn:Motion regularized McGehee Collision}, one has  
\begin{equation*}
    {\varphi}_0(\tau;\theta_0) = \left(r^*e^{\sqrt{2}\tau}, \theta_0+\left(\frac{r^*}{\kappa}\right)^{3/2}\left(1-e^{\frac{3}{\sqrt{2}}\tau}\right),\sqrt{2},-(r^*)^{3/2}e^{\frac{3}{\sqrt{2}}\tau}\right).
\end{equation*}
Moreover, since we are considering points in the invariant manifold, we have
\begin{equation}\label{eqn: Perturbed parabolic ejection solution in regularized coordinates}
    \begin{aligned}
        r(\tau;\theta_0) &= r^*e^{\sqrt{2}\tau} + \mathcal{O}_1(\mu e^{C\tau})\\
        \theta(\tau;\theta_0) &= \theta_0+\left(\frac{r^*}{\kappa}\right)^{3/2}\left(1-e^{\frac{3}{\sqrt{2}}\tau}\right)+ \mathcal{O}_1(\mu )\\
        v(\tau;\theta_0) &= \sqrt{2} + \mathcal{O}_1(\mu)\\
        u(\tau;\theta_0) &= -(r^*)^{3/2}e^{\frac{3}{\sqrt{2}}\tau} + \mathcal{O}_1(\mu ),
    \end{aligned}
\end{equation}
where $C>0$ is an adequate constant.
Then, we obtain 
\begin{equation*}
\begin{aligned}
    \Theta_{\circleSplus}^u(\theta_0) 
    =& -\mu \bigintss_{-\infty}^0 \Bigg[
    \frac{(r^*)^{\frac 5 2}e^{\frac{5}{\sqrt{2}}\tau}\sin\left(\theta_0+\left(\frac{r^*}{\kappa}\right)^{3/2}\left(1-e^{\frac{3}{\sqrt{2}}\tau}\right)\right)}{\left((r^*)^2e^{2\sqrt{2}\tau} - 2(r^*)e^{\sqrt{2}\tau}\cos\left(\theta_0+\left(\frac{r^*}{\kappa}\right)^{3/2}\left(1-e^{\frac{3}{\sqrt{2}}\tau}\right)\right) + 1\right)^{\frac 3 2}} \\
    &+ \sqrt{2}r^*e^{\sqrt{2}\tau}
    \cos\left(\theta_0+\left(\frac{r^*}{\kappa}\right)^{3/2}\left(1-e^{\frac{3}{\sqrt{2}}\tau}\right)\right)\Bigg]\;d\tau\\
    &+ \mathcal{O}_2(\mu).
\end{aligned}
\end{equation*}
To recover formula \eqref{def:halfMelnikovCol}, it is enough to apply the change of variables $\kappa s^{2/3}=r^* e^{\sqrt{2}\tau}$.

\end{proof}


\section{The distance between the invariant manifolds}\label{sec: Asymptotic formula for the distance between the manifolds}

Once we have characterized the invariant manifolds of infinity and collision, we analyze their intersections for $\mu > 0$ small enough at a section $r=r^*=\delta^2$ for some small $0<\delta<1$ independent of $\mu$. To this end we define, for a fixed energy $h = -\hat{\Theta}_0 = -\mu\overline{\Theta}_0$, the following curves 

\begin{equation}\label{eqn: Notation for the intersection of the invariant manifolds with the section r=2}
\begin{aligned}
    \Delta_\infty^{s,u}(\mu) =& \reducedInvariantManifoldsInfinity\;\bigcap\; \Sigma_{h} = \{r=\delta^2, \mathcal{H}(\delta^2,\theta,R,\Theta;\mu) = h\},\\
    \Delta_{\circleSplus}^u(\mu) =& \unstableManifoldSplusmu \;\bigcap\; \Sigma_{h} = \{r=\delta^2, \mathcal{H}(\delta^2,\theta,R,\Theta;\mu) = h\},\\
    \Delta_{\circleSminus}^s(\mu) =& \stableManifoldSminusmu \;\bigcap\; \Sigma_{h} = \{r=\delta^2, \mathcal{H}(\delta^2,\theta,R,\Theta;\mu) = h\},
\end{aligned}    
\end{equation}
where 
$\mathcal{H}$ is the Hamiltonian \eqref{eqn: Hamiltonian function in rotating polar coordinates centered at $P_1$}. This definition only involves the first intersection with the section $\Sigma_{h}$.
Note that for a fixed $h=-\hat{\Theta}_0$, the surface $\hamiltonianPolarRotatingCoordinatesCenteredSun(\delta^2,\theta,R,\Theta;\mu) = h$ is $2$-dimensional and can be locally defined by the coordinates $(\theta,\Theta)$. Moreover, by Propositions \ref{proposition: Parameterization of the invariant manifolds of infinity in rotating polar coordinates centered at P1} and \ref{proposition: Perturbed invariant manifolds of collision in synodical polar coordinates}, the curves $\Delta_\infty^{s,u}(\mu)$, $\Delta_{\circleSplus}^{u}(\mu)$ and $\Delta_{\circleSminus}^s(\mu)$ can be written as graphs with respect to $\theta$. We characterize the distance between the curves using the $\Theta$-component.

Hence, we can define $ \Delta_{\circleSplus}^u(\mu)$ (the definition of $\Delta_{\circleSminus}^s(\mu)$ comes from the symmetry \eqref{eqn: Symmetry of the PCRTBP in synodical polar coordinates centered at P1}) as
\begin{equation}\label{eqn: definition of the intersection of the invariant manifolds of collision with delta2 as graphs}
         \Delta_{\circleSplus}^u(\mu) = \Big\{(\theta,\Theta_{\circleSplus}^u(\theta)), \theta\in\mathds{T}\Big\},\quad
        \Delta_{\circleSminus}^s(\mu) = \Big\{(\theta,\Theta_{\circleSminus}^s(\theta)),\theta\in\mathds{T}\Big\},
\end{equation}
where $\Theta_{\circleSplus}^u$, $\Theta_{\circleSminus}^s$  are defined in Proposition \ref{proposition: Perturbed invariant manifolds of collision in synodical polar coordinates} for $r^* = \delta^2$. 

Note that the constant $t(r^*)$ introduced in Proposition  \ref{proposition: Perturbed invariant manifolds of collision in synodical polar coordinates}  is given by 
\begin{equation}\label{eqn: Time perturbed collision to reach delta2} 
    t(\delta^2)= \frac{\sqrt{2}}{3}\delta^3.
\end{equation}
Then, we have
\begin{equation}\label{eqn: Value of the angular momentum of Wu(S+) at r=2}
\begin{aligned}
    \Theta_{S^+}^u(\theta) = \mu I_{\circleSplus}^u(\theta) + \mathcal{O}_2(\mu)
\end{aligned}
\end{equation}
with
\begin{equation*}
    I_{\circleSplus}^u\left(\theta\right) = \mathcal{I}_{\circleSplus}^u\left(\theta + \frac{\sqrt{2}}{3} \delta^3\right)
\end{equation*}
where
\begin{equation}\label{eqn: Integral I_S+^u}
\begin{aligned}
    \mathcal{I}_{\circleSplus}^u(\alpha) =\constantValue\bigintss_{\frac{\sqrt{2}}{3}\delta^3}^0 \frac{s^{\frac 2 3}\sin\left(\alpha - s\right)}{\left(1+\constantValue^2 s^{\frac 4 3} - 2\constantValue s^{\frac 2 3}\cos\left(\alpha - s\right)\right)^{\frac 3 2}}\;ds + \sqrt{\frac 2 \constantValue} \bigintss_{\frac{\sqrt{2}}{3}\delta^3}^0\frac{\cos\left(\alpha - s\right)}{s^{\frac 1 3}}\;ds.
\end{aligned}
\end{equation}
The expression for $\Theta_{\circleSminus}^s(\theta)$ is given by the symmetry in \eqref{eqn: Angular momentum of Ws(S-) as a graph of theta}.

On the other hand, by Proposition \ref{proposition: Parameterization of the invariant manifolds of infinity in rotating polar coordinates centered at P1}, the curves $\Delta_\infty^{s,u}(\mu)$ are defined as

\begin{equation}\label{eqn: definition of the intersection of the invariant manifolds of infinity with delta2 as graphs}
\begin{aligned}
    \Delta_\infty^s(\mu) =& \Big\{(\theta,\Theta_\infty^s(w_\Sigma,\theta,\hat{\Theta}_0)), \theta\in I_D^+(w_\Sigma)\Big\},\\
    \Delta_\infty^u(\mu) =& \Big\{(\theta,\Theta_\infty^s(w_\Sigma,\theta,\hat{\Theta}_0)), \theta\in I_D^-(w_\Sigma)\Big\},
\end{aligned}
\end{equation}
where $I^\pm_D(w)$ are defined in \eqref{eqn: domain Ip} and \eqref{eqn: domain Im}, $\Theta_\infty^{\ast}(w,\theta,\hat{\Theta}_0)$,  $\ast=u,s$, are in \eqref{eqn: Angular momentum of Ws(infty) as a graph of (w,theta)} and \eqref{eqn: Angular momentum of Wu(infty) as a graph of (w,theta)}, and $w_\Sigma$ is the value such that $r^*=\constantValue w^{\frac 2 3} = \delta^2$, that is
\begin{equation}\label{eqn: definition of wSigma}
    w_\Sigma = \frac{\sqrt{2}}{3}\delta^3.
\end{equation}
Therefore we obtain

\begin{equation}\label{eqn: Value of the angular momentum of the Ws(infty) at r=2}
\begin{aligned}
    \Theta_\infty^s(w_\Sigma,\theta,\hat{\Theta}_0) = \hat{\Theta}_0 -\mu I_\infty^s(\theta)+ \mathcal{O}_2(\mu)
\end{aligned}
\end{equation}
with
\begin{equation*}
    I_\infty^s(\theta) = \mathcal{I}_\infty^s\left(\theta + \frac{\sqrt{2}}{3}\delta^3\right)
\end{equation*}
and
\begin{equation}\label{eqn: Integral I_infty^s}
    \mathcal{I}_\infty^s(\alpha) =\constantValue\bigintss_{+\infty}^{\frac{\sqrt{2}}{3}\delta^3} \frac{s^{\frac{2}{3}} \sin\left(\alpha -s\right)}{\left(1+\constantValue^2s^{\frac 4 3} - 2\constantValue s^{\frac 2 3}\cos(\alpha -s)\right)^{\frac{3}{2}}}\; ds +\sqrt{\frac{2}{\constantValue}}\bigintss_{+\infty}^{\frac{\sqrt{2}}{3}\delta^3} \frac{\cos(\alpha-s)}{s^{\frac{1}{3}}}\; ds.
\end{equation}
Now we measure the distance between the invariant manifolds of infinity and collision. Note that the sign of the radial velocity $R$ for the curves $\Delta_\infty^s(\mu)$ and $\Delta^u_{\circleSplus}(\mu)$ is positive, and negative for the other two. 
Hence, we define the transverse sections
        \begin{equation}\label{eqn: definition of Sigma2Rpos and Sigma2Rneg}
    \begin{aligned}
        \sectionrTwohRpos &= \{(r,\theta,R,\Theta) \colon r = \delta^2, \mathcal{H}(\delta^2,\theta,R,\Theta;\mu) = h , R > 0\} \subset \sectionrTwo_{h},\\
        \sectionrTwohRneg &= \{(r,\theta,R,\Theta) \colon r = \delta^2, \mathcal{H}(\delta^2,\theta,R,\Theta;\mu) = h , R < 0\} \subset \sectionrTwo_{h}.
    \end{aligned}
    \end{equation}
In these sections, we provide an asymptotic formula for the distance between $\Delta_\infty^{s,u}(\mu)$ and $\Delta_{\circleSpm}^{u,s}(\mu)$ at the sections $\sectionrTwohRpos$ and $\sectionrTwohRneg$ respectively for a given angle $\theta$. That is,
\begin{equation}\label{eqn: Explicit formula for the distance}
d_+(\theta,\hat{\Theta}_0) = \Theta_\infty^s(\theta,\hat{\Theta}_0) - \Theta_{\circleSplus}^u(\theta),\quad d_-(\theta,\hat{\Theta}_0) = \Theta_\infty^u(\theta,\hat{\Theta}_0) - \Theta_{\circleSminus}^s(\theta).
\end{equation}
The next theorem, whose proof is straightforward, gives an asymptotic formula for this distance.
\begin{theorem}\label{thm: asymptotic formula for the distance}

Fix $\eta> 0$. There exists $\mu_0 > 0$ such that for any $\mu \in (0,\mu_0)$ and $\hat{\Theta}_0 \in [-\eta \mu, \eta\mu]$, the distances $\distancepos$ and $\distanceneg$ in \eqref{eqn: Explicit formula for the distance} are given by

\begin{equation}\label{eqn: Asymptotic formula for the distances}
\begin{aligned}
    \distancepos(\theta,\hat{\Theta}_0) &= \hat{\Theta}_0 +\mu \melnikovpos\left(\theta \right) + \mathcal{O}_2(\mu)\quad \text{for }\theta\in I_D^+(w_\Sigma),\\
    \distanceneg(\theta; \hat{\Theta}_0) &= \hat{\Theta}_0 - \mu \melnikovneg(\theta) + \mathcal{O}_2(\mu)\quad \text{for }\theta\in I_D^-(w_\Sigma)
\end{aligned}
\end{equation}
where $w_\Sigma$ is defined in \eqref{eqn: definition of wSigma} and
\begin{equation}\label{eqn: translation melnikov}
    M_\pm(\theta)=\mathcal{M}_\pm\left(\theta \pm w_\Sigma\right)
\end{equation}
where
\begin{equation}\label{eqn: melnikov function}
 \mathcal{M}_+\left(\alpha\right)=\constantValue \bigintss_{0}^{+\infty} \frac{s^{\frac{2}{3}} \sin\left(\alpha-s\right)}{\left(1+\constantValue^2s^{\frac 4 3} - 2\constantValue s^{\frac 2 3}\cos\left(\alpha-s\right)\right)^{\frac{3}{2}}}\; ds +\sqrt{\frac{2}{\constantValue}}\bigintss_{0}^{+\infty} \frac{\cos\left(\alpha-s\right)}{s^{\frac{1}{3}}}\; ds,
\end{equation}
with $\constantValue = \frac{3^{\frac 2 3}}{2^{\frac 1 3}}$. The functions $M_-(\theta)$ and $\mathcal{M}_-\left(\alpha\right)$ are given by the symmetry
\begin{equation}\label{eqn: Relation between M- and M+}
    M_-(\theta) = -M_+(-\theta), \qquad \mathcal{M}_-\left(\alpha\right) = -\mathcal{M}_+\left(-\alpha\right).
\end{equation}
\end{theorem}

Once we have an asymptotic formula for the distance between the invariant manifolds of infinity and collision, we analyze their first order to  prove that they intersect transversally. 
The next lemma is proven in Section \ref{sec: Computer Assisted Proofs} and it is computer assisted.


\begin{lemma}\label{lemma: derivative of the melnikov functions not 0 for almost every theta}
    For every $\theta_+ \in B_+ = [-1.72851,-0.583065] \cup [-0.407155,0.0578054] \cup [0.921743,4.15633]$
    \begin{equation*}
        \frac{d}{d\theta} \mathcal{M}_+\left(\theta_+\right) \neq 0.
    \end{equation*}
    Moreover,
    \begin{equation*}
        \frac{d}{d\theta} \mathcal{M}_+\left(0\right)<0.
    \end{equation*}
    The result $\mathcal{M}_-'(\theta_-)\neq 0$ for $\theta_- \in B_-$, where $B_- = 2\pi - B_+$, comes as a consequence of \eqref{eqn: translation melnikov} and \eqref{eqn: Relation between M- and M+}.
\end{lemma}
Note that $B^\pm \subset I_{D}^\pm(w_\Sigma)$ in \eqref{eqn: domain Ip} and \eqref{eqn: domain Im} (with $w_\Sigma$ given in \eqref{eqn: definition of wSigma}) for $D = \frac{3}{10}$. Smaller values of $D$ will lead to larger intervals of $B_\pm$, at the cost of increased computational time to obtain the derivatives.

The following theorem, which is a direct consequence of Theorem \ref{thm: asymptotic formula for the distance}, Lemma \ref{lemma: derivative of the melnikov functions not 0 for almost every theta} and the Implicit Function Theorem, gives the transversality of the intersection between the invariant manifolds $\invariantManifoldsSpm$ and $\invariantManifoldsInfinity$ for some values of $\hat{\Theta}_0$.
\begin{theorem}\label{thm: transverse intersection of the invariant manifolds}
Consider $U_+ \subset B_+ \subset I_{\frac{3}{10}}^+(w_\Sigma)$ an open set satisfying that, for any $\theta^+\in U_+$, $\theta^+-w_\Sigma\in I_{\frac{3}{10}}^+(w_\Sigma)$ and $-\theta^+ + w_\Sigma\in I_{\frac{3}{10}}^-(w_\Sigma)$. Then, for any $\theta^+ \in U_+$, there exists $\mu_0>0$ such that, for any $\mu\in(0,\mu_0)$ and energy $h = -\hat{\Theta}_0 = -\mu\overline{\Theta}_0= \mu M_+(\theta^+-w_\Sigma)$,
\begin{itemize}
    \item $\reducedStableManifoldInfinitymu \transv \unstableManifoldSplusmu$ at a point \[p_>(\theta^+,\mu) = (\theta_>(\theta^+,\mu),\Theta_\infty^s(\theta_>(\theta^+,\mu),\hat{\Theta}_0))=(\theta_>(\theta^+,\mu),\Theta_{S^+}^u(\theta_>(\theta^+,\mu))),\] where $\theta_>(\theta^+,\mu) = \theta^+ - w_\Sigma + \mathcal{O}_1(\mu)$.
    \item $\reducedUnstableManifoldInfinitymu\transv \stableManifoldSminusmu$ at a point \[p_<(\theta^+,\mu) =  (\theta_<(\theta^+,\mu),\Theta_\infty^u(\theta_<(\theta^+,\mu),\hat{\Theta}_0))= (\theta_<(\theta^+,\mu), \Theta_{S^-}^s(\theta_<(\theta^+,\mu))),\] where $\theta_<(\theta^+,\mu) = -\theta^+ + w_\Sigma + \mathcal{O}_1(\mu)$.
    \item Both points $p_>(\theta^+,\mu)$ and $p_<(\theta^+, \mu)$ belong to the surface $\mathcal{H}(\delta^2,\theta,R,\Theta;\mu) = h = \mu M_+(\theta^+-w_\Sigma)$.
\end{itemize}
    
    

\end{theorem}

\begin{proof}
Finding a transverse intersection between $\reducedStableManifoldInfinitymu$ and $\unstableManifoldSplusmu$ is equivalent to find a non-degenerate zero of $d_+(\theta,\hat{\Theta}_0)$ (see \eqref{eqn: Asymptotic formula for the distances}). To this end, consider $\theta^+ \in U_+$, denote $\hat{\Theta}_0 = \mu\overline{\Theta}_0$ with $\overline{\Theta}_0 = - M_+(\theta^+-w_\Sigma)$ and define
    
    \begin{equation*}
    \begin{aligned}
        \mathcal{F}_+(\theta,\mu) = \mu^{-1}\cdot d_+(\theta, \hat{\Theta}_0) = -M_+(\theta^+-w_\Sigma) + M_+(\theta) + \mathcal{O}(\mu),
    \end{aligned}
    \end{equation*}
    which satisfies
    \begin{itemize}
        \item $\mathcal{F}_+(\theta^+-w_\Sigma,0) = 0$.
        \item $\frac{d}{d\theta}\mathcal{F}_+(\theta^+-w_\Sigma,0) = M_+'(\theta^+-w_\Sigma) = \mathcal{M}_+'(\theta^+) \neq 0$ due to \eqref{eqn: translation melnikov} and Lemma \ref{lemma: derivative of the melnikov functions not 0 for almost every theta}.
    \end{itemize}
    Then, the Implicit Function Theorem ensures that there exists $\mu_0 > 0$ such that, for any $\mu\in(0,\mu_0)$, there is $\theta_>(\mu)$ with $\theta_>(0) = \theta^+-w_\Sigma$ such that $d_+(\theta_>(\mu), \mu\overline{\Theta}_0) = 0$.

    The procedure to obtain a transverse intersection between $\reducedUnstableManifoldInfinitymu$ and $\stableManifoldSminusmu$ at the same energy level is completely analogous.


\end{proof}

\section{Proof of Theorem \ref{thm: Existence of infinite sequence of ECO orbits for small energies}}\label{sec: Proof of Theorem for ECO orbits}

To prove Theorem \ref{thm: Existence of infinite sequence of ECO orbits for small energies}, we show that there exist ejection-collision orbits  ``close'' to the invariant manifolds of infinity. To this end, we rely on a suitable return map.

We fix $\mu\in(0,\mu_0)$ so that Theorem \ref{thm: transverse intersection of the invariant manifolds} holds and we denote by $p_> = \left(\theta_>, \Theta_>\right)$ the point of intersection of $\curveIntersectionStableInfinitySectionrTwohRPos(\mu)$ and $\curveIntersectionUnstableSplusSectionrTwohRPos(\mu)$ at $\sectionrTwohRpos$ (see \eqref{eqn: Notation for the intersection of the invariant manifolds with the section r=2} and \eqref{eqn: definition of Sigma2Rpos and Sigma2Rneg}). Since $p_> \in\curveIntersectionStableInfinitySectionrTwohRPos(\mu)$, then $\Theta_> = \Theta_\infty^s(\theta_>)$, where $\Theta_\infty^s(\theta)$ is defined in \eqref{eqn: Value of the angular momentum of the Ws(infty) at r=2}. We consider a sufficiently small ``rectangle'' $U \subset \sectionrTwohRpos$ defined as 
\begin{equation*}\label{eqn: neighborhood of the intersection upper section proof thm 1.2}
    U = \bigg\{(\theta,\Theta) \in \sectionrTwohRpos \colon |\theta - \theta_>| < \varepsilon, \Theta_\infty^s(\theta) - \varepsilon < \Theta <\Theta_\infty^s(\theta)\bigg\},
\end{equation*}
for some $\varepsilon>0$ small enough. Then, we define the  Poincaré map
\begin{equation*}\label{eqn: Poincare map from sectionrTwohRpos to sectionrTwohRneg}
\begin{aligned}
    \poincareMapRposRneg \colon U \subset \sectionrTwohRpos &\to \sectionrTwohRneg\\
    p= (\theta, \Theta) &\mapsto \poincareMapRposRneg(p) = \Phi_{\mu}(t_\mu(p), (\delta^2,\theta,R(\theta,\Theta;h,\mu),\Theta)),
\end{aligned}
\end{equation*}
where $\Phi_\mu$ is the flow of the equations of motion associated to the Hamiltonian $\mathcal{H}$ in \eqref{eqn: Hamiltonian function in rotating polar coordinates centered at $P_1$}; $R(\theta,\Theta;h,\mu)$ is the radial velocity, which can also be computed from \eqref{eqn: Hamiltonian function in rotating polar coordinates centered at $P_1$}; $\sectionrTwohRpos$ and $\sectionrTwohRneg$ are the sections defined in \eqref{eqn: definition of Sigma2Rpos and Sigma2Rneg} and $t_\mu(p)$ is the time needed for the orbit with initial condition at $p \in U$ to reach $\sectionrTwohRneg$. By construction, $t_\mu(p)$ is well-defined and finite for $p \in U$ (but becomes unbounded as $p$ gets closer to the invariant manifold $\curveIntersectionStableInfinitySectionrTwohRPos(\mu)$).





Now we consider the $C^1$- curve
\begin{equation*}    
\gamma_{u,>} = U \cap \curveIntersectionUnstableSplusSectionrTwohRPos(\mu),
\end{equation*}
that, by Theorem \ref{thm: transverse intersection of the invariant manifolds}, intersects transversally $\curveIntersectionStableInfinitySectionrTwohRPos(\mu)$ at $p_>$.

Since the points of the curve $\gamma_{u,>}$ are close enough to the point $p_>$, they are close to $\curveIntersectionStableInfinitySectionrTwohRPos(\hat{\Theta}_0,\mu)$. Hence, the study of the image $\mathcal{P}(\gamma_{u,>})$  is  reduced to the analysis of the dynamics ``close'' to $\Alpha_{\hat{\Theta}_0}$. Thus, one can easily adapt the approach done by Moser for the Sitnikov problem in \cite{moser2001stable} (see also \cite{MR3455155}) to this case, and prove that the image $\poincareMapRposRneg(\gamma_{u,>})$ ``spirals'' towards $\Delta^u_\infty(\mu)$ in \eqref{eqn: Notation for the intersection of the invariant manifolds with the section r=2}. In particular, we have that

\begin{equation}\label{eqn: Lambda Lemma Moser}
    \Delta_\infty^u(\mu) \subset \overline{\poincareMapRposRneg(\gamma_{u,>})}.
\end{equation}
Theorem \ref{thm: transverse intersection of the invariant manifolds} ensures that $\Delta^s_{\circleSminus}(\mu)$ and $\curveIntersectionUnstableInfinitySectionrTwohRNeg(\mu)$ intersect transversally in $\sectionrTwohRneg$ at a point that we denote by $q_< \in \stableManifoldSminusmu \transv \unstableManifoldInftymu$. 
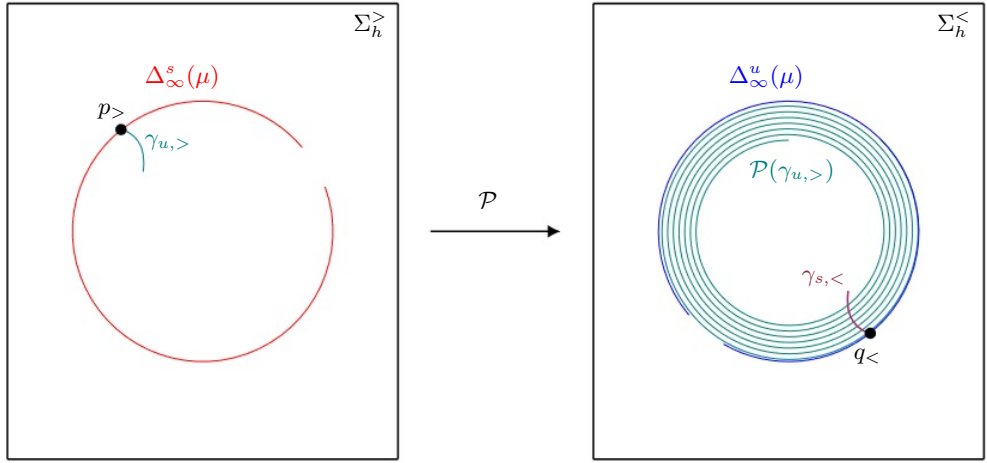
\begin{figure}[h!]
    \centering
\begin{overpic}[scale=0.6]{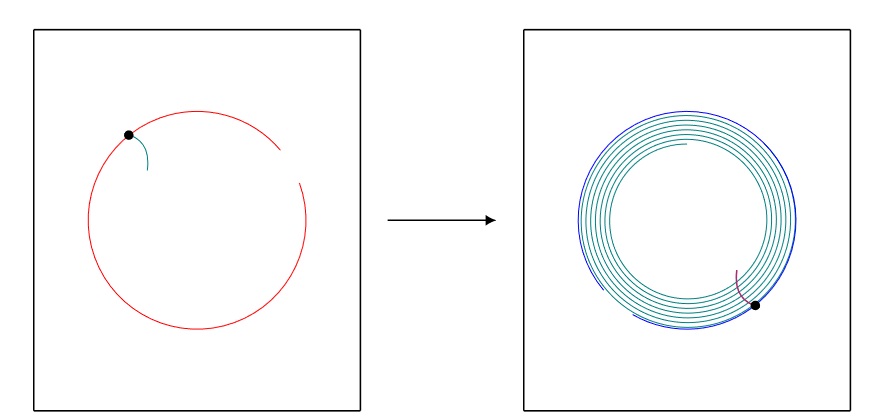}
\put (12.5,34){\footnotesize $p_>$}
\put (85,10.5){\footnotesize $q_<$}
\put(17,37) {\color{red} \footnotesize $\Delta_\infty^s(\mu)$}
\put(73,37) {\color{blue} \footnotesize $\Delta_\infty^u(\mu)$}
\put(17,31) {\color{teal} \footnotesize $\gamma_{u,>}$}
\put(75,28) {\color{teal} \footnotesize $\mathcal{P}(\gamma_{u,>})$}
\put (49,25) {\footnotesize $\mathcal{P}$}
\put (80,18) {\color{RedViolet} \footnotesize $\gamma_{s,<}$}
\put (37,42) {\footnotesize $\Sigma_h^>$}
\put (93,42) {\footnotesize $\Sigma_h^<$}
\end{overpic}
\caption{Representation of the ``spiraling effect'' of the transition map $\mathcal{P}$ on the curve $\gamma_{u,>}\subset \Delta_{\circleSplus}^u(\mu)\subset \Sigma_h^>$ and its transverse intersections with the curve $\gamma_{s,<}\subset \Delta_{\circleSminus}^s(\mu)\subset \Sigma_h^<$.}
\label{fig:Transition map close to infinity}
\end{figure}

Therefore, if we consider a $C^1$-curve $\gamma_{s,<}$ defined as follows
\begin{equation*}
    \gamma_{s,<} = V \cap \Delta^s_{\circleSminus}(\mu),
\end{equation*}
where $V$ is a sufficiently small neighborhood of $q_<$,
%
there exists a sequence $q_k \in \gamma_{s,<}\cap \mathcal{P}(\gamma_{u,>})$ such that $\underset{k\to+\infty}{\lim} q_k = q_<$ (see Figure \ref{fig:Transition map close to infinity}). In particular, the closer is $q_k$ to $\Delta_\infty^u(\mu)$, the larger is the time $t_\mu(p_k)$, where $p_k$ is such that $\mathcal{P}(p_k) = q_k$.
Since $\mathcal{P}(\gamma_{u,>}) \subset \unstableManifoldSplusmu$, the points $q_k\in \unstableManifoldSplusmu \cap\stableManifoldSminusmu$ and, therefore, give rise to ejection-collision orbits.

\section{Dynamics close to collision}\label{sec: Dynamics close to collision}
To prove Theorem \ref{thm: Existence of parabolic, oscillatory and periodic orbits for small energies}, it is necessary to study first the dynamics close to collision, which is the purpose of this section. To this end, we work with the coordinates $(s,\theta,\alpha)$ in \eqref{eqn:polar change of coordinates for u,v into rho alpha} and in a neighborhood of the collision manifold \eqref{eqn: Collision manifold} (defined by $s=0$) of the form
\begin{equation}\label{eqn: neighborhood of collision in McGehee coordinates}
\mathcal{B}_\delta = \big\{(s,\alpha,\theta) \in \generalMcGeheeInvariantManifold\colon s \in (0,2\delta)\big\},   
\end{equation}
where $\generalMcGeheeInvariantManifold$ is the $3$ dimensional manifold \eqref{eqn: 3dimensional submanifold McGehee in coordinates (r,theta,rho,alpha)} and $\delta>0$ is small enough and defined in \eqref{eqn: Notation for the intersection of the invariant manifolds with the section r=2}.

A first step towards a proper description of the dynamics close to collision consists on ``symplifying'' equations \eqref{eqn:Motion in coordinates (s,theta,alpha) McGehee Collision} in a neighborhood of the circles of equilibrium points $\circlesSpm$ in \eqref{eqn: Circles of equilibrium points for mu neq 0} by means of suitable changes of coordinates. This is done in Section \ref{subsec: Local straightening change of coordinates}, where we perform two changes of coordinates, each of them defined around the circles $\circlesSpm$ respectively, which straighten the invariant manifolds of $\circlesSpm$ (see Lemma \ref{lemma: Linear part of motion close to S+-}). In Section \ref{subsec: Transition map close to collision}, we define a transition map between $\Sigma_h^<$ and $\Sigma_h^>$ and show that it sends transverse curves to $\stableManifoldSminusmu$ to transverse curves to $\unstableManifoldSplusmu$.

\subsection{Local straightening of the invariant manifolds}\label{subsec: Local straightening change of coordinates}

We consider the following neighborhoods of the circles $\circlesSpm$ in \eqref{eqn: Circles of equilibrium points for mu neq 0} respectively

\begin{equation}\label{eqn: neighborhoods of S+- in coordinates (s,theta,alpha)}
    \begin{aligned}
        \mathcal{B}_\delta^+ &=\bigg\{(s,\alpha,\theta) \in \mathcal{B}_\delta \colon \alpha \in \left(\frac{\pi}{2}-\delta, \frac{\pi}{2}+\delta\right)\bigg\},\\
        \mathcal{B}_\delta^- &= \bigg\{(s,\alpha,\theta)\in \mathcal{B}_\delta \colon \alpha \in \left(-\frac{\pi}{2}-\delta, -\frac{\pi}{2} + \delta\right)\bigg\}.
    \end{aligned}
\end{equation}
The following proposition ensures the existence of two sets of coordinates defined in these neighborhoods straightening the invariant manifolds $\invariantManifoldsSpm$.

\begin{proposition}\label{proposition: Straightening local diffeomorphisms and motions}
    Fix any $\mu_0\in (0,1/2]$ and take $\delta>0$ small enough. Then, for any $\mu\in[0,\mu_0]$, there exists a change of variables

    \begin{equation}\label{eqn: Diffeo (s,alpha,theta)-(s,Tilde(beta),z)}
        \begin{aligned}
            \mapStraightenSminus\colon \mathcal{B}_{\delta}^- &\to (0,2\delta) \times (-2\delta,2\delta) \times \mathds{T}\\
            (s,\alpha,\theta) &\mapsto \mapStraightenSminus(s,\alpha,\theta) = (s,\angleStraightWsSminus,\straightenFiberWsSminus)
        \end{aligned}
    \end{equation}
    which transforms  the system \eqref{eqn:Motion in coordinates (s,theta,alpha) McGehee Collision} into
        \begin{equation}\label{eqn: Eq motion in coordinates (s,Tilde(beta),z)}
        \begin{aligned}
            s' &= -\frac{\constantEquilibrium}{2}s\left(1 + \Tilde{g_1}(s,\angleStraightWsSminus,\straightenFiberWsSminus)\right)\\
            \angleStraightWsSminus' &= \frac{\constantEquilibrium}{2}\angleStraightWsSminus\left(1 + \Tilde{g_2}(s,\angleStraightWsSminus,\straightenFiberWsSminus)\right)\\
            \straightenFiberWsSminus' &= s\angleStraightWsSminus\left(4\lambda(\mu,h)s +  \Tilde{g_3}(s,\angleStraightWsSminus,\straightenFiberWsSminus)\right)
        \end{aligned}
    \end{equation}
    with $m_0 = \sqrt{2(1-\mu)}$, $\lambda(\mu,h) = \frac{\mu^2 + 2h+2\mu}{4\constantEquilibrium}$ and

    \begin{equation}\label{eqn: Approximation of the nonlinear part for (s,Tilde(beta),z)}
        \Tilde{g_1}(s,\angleStraightWsSminus,\straightenFiberWsSminus) = \mathcal{O}_2(s,\angleStraightWsSminus);\quad \Tilde{g_2}(s,\angleStraightWsSminus,\straightenFiberWsSminus) = \mathcal{O}_1(s^2,\angleStraightWsSminus);\quad \Tilde{g_3}(s,\angleStraightWsSminus,\straightenFiberWsSminus) = \mathcal{O}_3(s,\angleStraightWsSminus).
    \end{equation}
    Moreover, in these coordinates, $S^-=\{(0,0,z), z\in \mathds{T}\}$ and its invariant manifolds become

    \begin{equation}\label{eqn: straightening Wus S-}
        \begin{aligned}
            \stableManifoldSminusmu &= \{\angleStraightWsSminus = 0\},\\
            \unstableManifoldCollisionSminusmu &= \{s = 0\}.\\
        \end{aligned}
    \end{equation}
    Analogously, there exists a diffeomorphism

    \begin{equation}\label{eqn: Diffeo (s,alpha,theta)-(s,iota,w)}
        \begin{aligned}
            \mapStraightenSplus\colon \mathcal{B}_{\delta}^+ &\to (0,2\delta) \times (-2\delta,2\delta) \times \mathds{T}\\
            (s,\alpha,\theta) &\mapsto \mapStraightenSplus(s,\alpha,\theta) = (s,\angleStraightWuSplus,\straightenFiberWuSplus)
        \end{aligned}
    \end{equation}
    so that in the coordinates $(s,\angleStraightWuSplus,\straightenFiberWuSplus)$, system \eqref{eqn:Motion in coordinates (s,theta,alpha) McGehee Collision} becomes

    \begin{equation}\label{eqn: Eq motion in coordinates (s,iota,w)}
        \begin{aligned}
            s' &= \frac{\constantEquilibrium}{2}s\left(1 + \Tilde{j_1}(s,\angleStraightWuSplus,\straightenFiberWuSplus)\right)\\
            \angleStraightWuSplus' &= -\frac{\constantEquilibrium}{2}\angleStraightWuSplus\left(1 + \Tilde{j_2}(s,\angleStraightWuSplus,\straightenFiberWuSplus)\right)\\
            \straightenFiberWuSplus' &= s\angleStraightWuSplus\left(-4\lambda(\mu,h)s + \Tilde{j_3}(s,\angleStraightWuSplus,\straightenFiberWuSplus)\right),
        \end{aligned}
    \end{equation}
    where
    \begin{equation}\label{eqn: Approximation of the nonlinear part for (s,Tilde(iota),w)}
        \Tilde{j_1}(s,\angleStraightWuSplus,\straightenFiberWuSplus) =  \mathcal{O}_2(s,\angleStraightWuSplus);\quad \Tilde{j_2}(s,\angleStraightWuSplus,\straightenFiberWuSplus) = \mathcal{O}_1(s,\angleStraightWuSplus);\quad \Tilde{j_3}(s,\angleStraightWuSplus,\straightenFiberWuSplus) = \mathcal{O}_3(s,\angleStraightWuSplus).
    \end{equation}
    In these coordinates, $\circleSplus = \{(0,0,w), w\in \mathds{T}\}$ and its invariant manifolds become
    \begin{equation}\label{eqn: straightening Wus S+}
        \begin{aligned}
            \unstableManifoldSplusmu &= \{\angleStraightWuSplus = 0\},\\
            \stableManifoldCollisionSplusmu &= \{s = 0\}.
        \end{aligned}
    \end{equation}
\end{proposition}

\begin{proof}
We start by straightening the local invariant manifolds of $\circleSminus$. Lemma \ref{lemma: Linear part of motion close to S+-} implies that $\unstableManifoldCollisionSminusmu$ is already straightened in the coordinates $(s,\theta,\mcGeheePolarAngle)$, since it lies in the hyperplane $\{s=0\}$. Hence, we only need to straighten $\stableManifoldSminusmu$ and the fibers $W_\mu^u(\circleSminus_\theta)$ and $W_\mu^s(\circleSminus_\theta)$. We define the following change of coordinates

\begin{equation}\label{eqn: change of coordinates alpha-beta for Sminus}
\begin{aligned}
    \angleSminus &= \mcGeheePolarAngle + \frac{\pi}{2}\\
    y &= \theta - 2\left(\alpha + \frac{\pi}{2}\right),
\end{aligned}
\end{equation}
to straighten the fibers $W_\mu^u(\circleSminus_\theta)$. In coordinates $(s,\beta,y)$, system \eqref{eqn:Motion in coordinates (s,theta,alpha) McGehee Collision} becomes

\begin{equation}\label{eqn: Eq motion (s,beta,y) taylor s}
    \begin{aligned}
        s' &=-\frac{\constantEquilibrium}{2} \cos\angleSminus s -\lambda(\mu,h)\cos\angleSminus s^3 + \mathcal{O}_5(s)\\
        \angleSminus'&=  \frac{\constantEquilibrium}{2}\sin\angleSminus - 3\lambda(\mu,h)\sin\angleSminus s^2 + 2s^3 + \mathcal{O}_4(s)\\
        y' &= 4\lambda(\mu,h)\sin\angleSminus s^2 - 4s^3 + \mathcal{O}_4(s),
    \end{aligned}
\end{equation}
where $\constantEquilibrium = \sqrt{2(1-\mu)}$ and $\lambda(\mu,h) = \frac{\mu^2 + 2h+2\mu}{4\constantEquilibrium}$. 

Moreover, the  circle $\circleSminus$ in \eqref{eqn: Circles of equilibrium points for mu neq 0} and $\mathcal{B}_\delta^-$ in \eqref{eqn: neighborhoods of S+- in coordinates (s,theta,alpha)} become
\begin{equation}\label{eqn:  Description of the circle S- and its neighborhood Bdelta- in coordinates (s,theta,beta)}
\begin{aligned}
    \circleSminus &= \bigg\{S_\theta^- = \left(0,0, y\right)\colon y \in \mathds{T}\bigg\},\\
    \mathcal{B}_\delta^- &= \bigg\{(s,\beta,y) \colon s\in(0,2\delta), \beta \in \left(-\delta,\delta\right), y\in\mathds{T}\bigg\}.
\end{aligned}
\end{equation}
Expanding around $s=0, \beta=0$, we write system \eqref{eqn: Eq motion (s,beta,y) taylor s} as

\begin{equation}\label{eqn: Reduced equations of motion in coordinates (s,beta,y)}
    \begin{aligned}
        s' &= -\frac{\constantEquilibrium}{2}s\left(1+g_1(s,\angleSminus,y)\right)\\
        \angleSminus' &= \frac{m_0}{2}\angleSminus + g_2(s,\angleSminus,y)\\
        y' &= s g_3(s,\angleSminus,y),
    \end{aligned}
\end{equation}
where

\begin{equation}\label{eqn: Approximation of g_1 and g_3}
    \begin{aligned}
        g_1(s,\angleSminus,y) &= \frac{2\lambda(\mu,h)}{\constantEquilibrium}s^2 - \frac 1 2 \angleSminus^2 + \mathcal{O}_4(s,\angleSminus),\\
        g_2(s,\angleSminus,y) &= \mathcal{O}_1\left(s^3,\beta s^2, \beta^3\right),\\
        g_3(s,\angleSminus,y) &= 4\lambda(\mu,h) s\angleSminus - 4s^2 + \mathcal{O}_4(s,\angleSminus).
    \end{aligned}
\end{equation}
Note that in these coordinates $W_\mu^u(S^-_{y_0}) = \{s=0,y=y_0\in \mathds{T}\}$.

On the other hand, to straighten $\stableManifoldSminusmu$, we use the fact that it is tangent to the plane $\{\angleSminus = 0\}$. Therefore, by the stable manifold theorem, $\stableManifoldSminusmu$ can be written as the graph of a function as
\begin{equation}\label{eqn: graph of W^s(S-)}
    \stableManifoldSminusmu = \bigg\{(s,\angleSminus,y) \in \ballCollisionSminus \colon \angleSminus = \graphOfStableManifoldSminus(s,y)\bigg\},
\end{equation}
where $\graphOfStableManifoldSminus$ satisfies
\begin{equation}\label{eqn: Approximation of psi-(s,y) in terms of s}
        \graphOfStableManifoldSminus(s,y) = \mathcal{O}_2(s).
\end{equation}
Hence if we perform the change of coordinates

\begin{equation}\label{eqn: Change of coordinates beta-Tilde(beta)}
    (s,\angleStraightWsSminus,y) = (s,\angleSminus - \graphOfStableManifoldSminus(s,y),y),
\end{equation}
then  $\stableManifoldSminusmu$ becomes $\{\Tilde{\beta} = 0\}$ and, using  the expressions of $g_1,g_2,g_3$ in \eqref{eqn: Approximation of g_1 and g_3},  \eqref{eqn: Reduced equations of motion in coordinates (s,beta,y)} becomes

\begin{equation}\label{eqn: Eq motion in coordinates (s,Tilde(beta),y)}
    \begin{aligned}
    s' &= -\frac{\constantEquilibrium}{2}s\left(1+\overline{g}_1(s,\angleStraightWsSminus,y)\right)\\
    \angleStraightWsSminus'&=\frac{\constantEquilibrium}{2}\angleStraightWsSminus\left(1 + \overline{g}_2(s,\angleStraightWsSminus,y)\right)\\
    y' &= s\overline{g}_3(s,\angleStraightWsSminus,y),
    \end{aligned}
\end{equation}
with
\begin{equation}\label{eqn: Expression of tilde g_1,G_2 and g_3 in coordinates (s,Tilde(beta),y)}
\begin{aligned}
    \overline{g}_1(s,\angleStraightWsSminus,y) &= g_1(s,\angleStraightWsSminus + \graphOfStableManifoldSminus(s,y),y) = \frac{2\lambda(\mu,h)}{\constantEquilibrium}s^2 - \frac 1 2 \angleStraightWsSminus^2 + \mathcal{O}_4(s,\angleStraightWsSminus),\\
    \overline{g}_2(s,\angleStraightWsSminus,y) &= \partial_{\Tilde\beta} g_2 (s,\Tilde\beta +\psi_-(s,y),y)\big|_{\Tilde\beta = 0} + \mathcal{O}_1(\Tilde{\beta})= \mathcal{O}_1(s^2,\angleStraightWsSminus),\\
    \overline{g}_3(s,\angleStraightWsSminus,y) &= g_3(s,\angleStraightWsSminus + \graphOfStableManifoldSminus(s,y),y) = 4\lambda(\mu,h)s\angleStraightWsSminus - 4s^2 + \mathcal{O}_4(s,\angleStraightWsSminus).
\end{aligned}
\end{equation}
Finally, we straighten the fibers $\localStableInvariantManifoldFiberSminusReferencey$. Indeed, Lemma \ref{lemma: Linear part of motion close to S+-} allows us to define $\stableManifoldSminusmu$ as the union of fibers $\localStableInvariantManifoldFiberSminusReferencey$, that is
\begin{equation*}
    \stableManifoldSminusmu = \underset{y_0\in \mathds{T}}{\bigcup}\localStableInvariantManifoldFiberSminusReferencey.
\end{equation*}
Moreover, since $\stableManifoldSminusmu$ corresponds to the plane $\{\angleStraightWsSminus = 0\}$, the stable manifold theorem ensures that $W_\mu^s(S^-_{y_0})$ can be described by the graph of a function as

\begin{equation}\label{eqn: Graph of W_y0^s}
    W_\mu^s(S^-_{y_0}) = \bigg\{(s,0,y) \in \mathcal{B}_\delta^- \colon y = \graphOfFiberStableManifoldSminus(s,y_0)\bigg\},
\end{equation}
for some function $\graphOfFiberStableManifoldSminus$ satisfying

\begin{equation}\label{eqn: Approximation of chi(s,z) in terms of s}
    \graphOfFiberStableManifoldSminus(s,y_0) = y_0 + a_0(y_0)s^2 + a_1(y_0) s^3 + \mathcal{O}_4(s).
\end{equation}
Now we consider the following change of coordinates
\begin{equation}\label{eqn: Change of coordinates y-z}
    (s,\Tilde{\beta},y) =\left(s,\Tilde{\beta}, \graphOfFiberStableManifoldSminus(s,\straightenFiberWsSminus)\right),
\end{equation}
so that system \eqref{eqn: Eq motion in coordinates (s,Tilde(beta),y)} becomes

\begin{equation}\label{eqn: Eq motion not reduced in coordinates (s,Tilde beta,z)}
    \begin{aligned}
        s' &= -\frac{\constantEquilibrium}{2}s\left(1 + \Tilde{g}_1(s,\angleStraightWsSminus,\straightenFiberWsSminus)\right)\\
        \angleStraightWsSminus' &= \frac{\constantEquilibrium}{2}\angleStraightWsSminus \left(1+\Tilde{g}_2(s,\angleStraightWsSminus,\straightenFiberWsSminus)\right)\\
        \straightenFiberWsSminus' &= \frac{s}{\partial_{\straightenFiberWsSminus}\graphOfFiberStableManifoldSminus(s,z)}\left(\Tilde{g}_3(s,\angleStraightWsSminus,\straightenFiberWsSminus) + \frac{\constantEquilibrium}{2}\partial_s \graphOfFiberStableManifoldSminus(s,z)\left(1+ \Tilde{g}_1(s,\angleStraightWsSminus,\straightenFiberWsSminus)\right)\right) = s \cdot Z(s,\angleStraightWsSminus,\straightenFiberWsSminus),
    \end{aligned}
\end{equation}
with
\begin{equation}\label{eqn: Expression for overline g_1,G_2 and g_3 in (s,Tilde(beta),z)}
    \begin{aligned}
    \Tilde{g}_1(s,\angleStraightWsSminus,\straightenFiberWsSminus) &= \overline{g}_1(s,\angleStraightWsSminus,\graphOfFiberStableManifoldSminus(s,\straightenFiberWsSminus)) = \frac{2\lambda(\mu,h)}{\constantEquilibrium} s^2- \frac 1 2 \angleStraightWsSminus^2+  \mathcal{O}_4(s,\angleStraightWsSminus)\\
    \Tilde{g}_2(s,\angleStraightWsSminus,\straightenFiberWsSminus) &= \overline{g}_2(s,\angleStraightWsSminus,\graphOfFiberStableManifoldSminus(s,\straightenFiberWsSminus)) =  \mathcal{O}_1(s^2,\angleStraightWsSminus)\\
    \Tilde{g}_3(s,\angleStraightWsSminus,\straightenFiberWsSminus) &= \overline{g}_3(s,\angleStraightWsSminus,\graphOfFiberStableManifoldSminus(s,\straightenFiberWsSminus)) = 4\lambda(\mu,h)s\angleStraightWsSminus - 4s^2 +  \mathcal{O}_4(s,\angleStraightWsSminus).
    \end{aligned}
\end{equation}
The invariance property ensures that $Z(s,0,z) = 0$. Therefore

\begin{equation*}
    \Tilde{g}_3(s,0,z) + \frac{\constantEquilibrium}{2}\partial_s\graphOfFiberStableManifoldSminus(s,z)\left(1 + \Tilde{g}_1(s,0,z)\right) = 0.
\end{equation*}
Using \eqref{eqn: Approximation of chi(s,z) in terms of s}, \eqref{eqn: Expression for overline g_1,G_2 and g_3 in (s,Tilde(beta),z)} and comparing coefficients, we have

\begin{equation*}
    \constantEquilibrium a_0(z) \cdot s + \left(-4+\frac{3\constantEquilibrium}{2}a_1(z)\right)s^2 + \mathcal{O}_3(s) = 0.
\end{equation*}
Hence

\begin{equation}\label{eqn: Values a0, a1}
    a_0(z) = 0,\;\;\;\; a_1(z) = \frac{8}{3\constantEquilibrium}.
\end{equation}
Moreover, the fact that $Z(s,0,z) = 0$ implies that $Z(s,\angleStraightWsSminus,z)$ cannot have independent terms in neither $s$ nor $z$. Therefore, $Z(s,\angleStraightWsSminus,z)$ in \eqref{eqn: Eq motion not reduced in coordinates (s,Tilde beta,z)} can be computed ignoring the independent terms in $s$ for $\Tilde{g_1}$ and $\Tilde{g_3}$ from \eqref{eqn: Expression for overline g_1,G_2 and g_3 in (s,Tilde(beta),z)}. Namely

\begin{equation*}
\begin{aligned}
    Z(s,\angleStraightWsSminus,z) &= \frac{1}{1+\mathcal{O}_4(s)}\left(4\lambda(\mu,h)s\angleStraightWsSminus + \mathcal{O}_4(s,\angleStraightWsSminus) + \left(4s^2+ \mathcal{O}_3(s)\right)\left(-\frac 1 2 \angleStraightWsSminus^2 + \mathcal{O}_4(s,\angleStraightWsSminus)\right)\right)\\
    &=\angleStraightWsSminus\left(4\lambda(\mu,h)s + \mathcal{O}_3(s,\angleStraightWsSminus)\right),
\end{aligned}
\end{equation*}
\\
which implies that we can rewrite system \eqref{eqn: Eq motion not reduced in coordinates (s,Tilde beta,z)} as in \eqref{eqn: Eq motion in coordinates (s,Tilde(beta),z)}, completing the proof.


The procedure to straighten the invariant manifolds for $\circleSplus$ is completely analogous to the one explained for $\circleSminus$, involving the following changes instead:
\begin{itemize}
    \item The changes to translate $\circleSplus$ and to straighten the fibers $\localUnstableInvariantManifoldFiberSplusReferencey$
    \begin{equation} \label{eqn:changes iota-x}
    \begin{aligned}
        \angleSplus &= \mcGeheePolarAngle - \frac{\pi}{2},\\
        x &= \theta - 2\left(\mcGeheePolarAngle - \frac{\pi}{2}\right).
    \end{aligned}
    \end{equation}

    \item The change to straighten the unstable manifold $\unstableManifoldSplusmu$
    \begin{equation}\label{eqn: change iota-tilde(iota)}
        \angleStraightWuSplus = \angleSplus - \graphOfUnstableManifoldSplus(s,x),
    \end{equation}
    where $\graphOfUnstableManifoldSplus(s,x)$ has analogous properties as $\graphOfStableManifoldSminus(s,y)$ in \eqref{eqn: Approximation of psi-(s,y) in terms of s}, that is
    \begin{equation}\label{eqn: Approximation of psi+(s,x)}
    \graphOfUnstableManifoldSplus(s,x) = \mathcal{O}_2(s).
    \end{equation}

    \item The change to straighten the fibers $\localUnstableInvariantManifoldFiberSplusReferencey$
    \begin{equation}\label{eqn: change x-w}
        x = \graphOfFiberUnstableManifoldSplus(s,\straightenFiberWuSplus),
    \end{equation} 
    where $\graphOfFiberUnstableManifoldSplus(s,\straightenFiberWuSplus)$ has similar properties as $\graphOfFiberStableManifoldSminus(s,\straightenFiberWsSminus)$ from \eqref{eqn: Approximation of chi(s,z) in terms of s}, implying that
    \begin{equation}\label{eqn: Approximation of chi(s,w)}
        \graphOfFiberUnstableManifoldSplus(s,\straightenFiberWuSplus) = \straightenFiberWuSplus + \mathcal{O}_3(s).
    \end{equation}
\end{itemize}
\end{proof}
\begin{remark}\label{remark: Translation from (s,Tilde(beta),z) to (s,Tilde(iota),w)}
We want to stress the relation between coordinates $(s,\Tilde{\beta},z)$ and $(s,\Tilde{\iota},w)$, which will be of major importance for the computations in Section \ref{subsec: Transition map close to collision}. It follows from \eqref{eqn: change of coordinates alpha-beta for Sminus} and \eqref{eqn:changes iota-x} that
\begin{equation}\label{eqn: Relation iota-beta, x-y}
\begin{aligned}
    \iota &= \beta + \pi,\\
    x &= y \mod 2\pi.
\end{aligned}
\end{equation}
Therefore, one can use \eqref{eqn: Change of coordinates beta-Tilde(beta)} and \eqref{eqn: Relation iota-beta, x-y}, along with equations \eqref{eqn:changes iota-x} and \eqref{eqn: change iota-tilde(iota)}, to relate $\Tilde{\iota}$ and $\Tilde{\beta}$ as 
\begin{equation*}
    \Tilde{\iota} = \Tilde{\beta}+\pi + (\psi_-(s,y)-\psi_+(s,x)),
\end{equation*}
where we recall 
\begin{equation*}
    \begin{aligned}
        \psi_-(s,y) = \mathcal{O}_2(s),\;\; \psi_+(s,x) = \mathcal{O}_2(s).
    \end{aligned}
\end{equation*}
Additionally, one can relate the variables $w$ and $z$ using \eqref{eqn: Change of coordinates y-z} and  \eqref{eqn: change x-w}, to obtain
\begin{equation}\label{eqn: Transformation w-z}
    w = z+ \mathcal{O}_3(s).
\end{equation}
\end{remark}



\subsection{Transition map close to collision}\label{subsec: Transition map close to collision}
The main purpose of this section is to define a transition map from $\Sigma_{h}^<$ to $\Sigma_{h}^>$ and prove that it sends transverse curves to $\stableManifoldSminusmu$ to transverse curves to $\unstableManifoldSplusmu$.
To this end, we consider the submanifold $\generalMcGeheeInvariantManifold$ (see \eqref{eqn: 3dimensional submanifold McGehee in coordinates (r,theta,rho,alpha)}) for a fixed $h$ and we define, in the straightened coordinates $(s,\Tilde{\beta},z)$ and $(s,\Tilde{\iota},w)$, the following  sections


\begin{equation}\label{eqn: Definition of Tilde(Cs) and Tilde(Cu)}
    \begin{aligned}
        \Tilde{\Sigma}_{h}^< &= \{(s,\Tilde{\beta},z) \colon s=\delta, -\delta < \Tilde{\beta} < \delta, z \in \mathds{T}\} =\Tilde{\Sigma}_{h}^{<,+} \cup \Tilde{\Sigma}_{h}^{<,-},\\
        \Tilde{\Sigma}_{h}^{>}&= \{(s,\Tilde{\iota},w) \colon s=\delta, -\delta < \Tilde{\iota} < \delta, w \in \mathds{T}\} = \Tilde{\Sigma}_{h}^{>,+} \cup \Tilde{\Sigma}_{h}^{>,-},
    \end{aligned}
\end{equation}
which are transverse to the invariant manifolds. We also define

\begin{equation}\label{eqn: definition of Tilde(Csu+-)}
    \begin{aligned}
       \Tilde{\Sigma}_{h}^{<,+} &= \{(s,\Tilde{\beta},z) \colon s=\delta, 0 < \Tilde{\beta} < \delta, z \in \mathds{T}\},\\
        \Tilde{\Sigma}_{h}^{<,-} &= \{(s,\Tilde{\beta},z) \colon s=\delta, -\delta < \Tilde{\beta} < 0, z \in \mathds{T}\},\\
        \Tilde{\Sigma}_{h}^{>,+} &= \{(s,\Tilde{\iota},w) \colon s=\delta, 0 < \Tilde{\iota} < \delta, w \in \mathds{T}\},\\
        \Tilde{\Sigma}_{h}^{>,-} &= \{(s,\Tilde{\iota},w) \colon s=\delta, -\delta < \Tilde{\iota} < 0, w \in \mathds{T}\}.
    \end{aligned}
\end{equation}
This section is devoted to prove the following theorem.
\begin{theorem}\label{thm: Transition map close to collision}
    Let $\delta, \sigma > 0$ be small enough with $0 < \sigma \ll \delta$. Consider a curve $\Tilde{\gamma}_\mathrm{in}$ of the form
    \begin{equation}\label{eqn: transverse curve gamma_in(nu)}
         \left(s,\Tilde{\beta}, z\right) =\Tilde{\gamma}_{\mathrm{in}}(\nu) = \left(\delta,\nu,z_{\mathrm{in}}(\nu)\right),\qquad \nu\in (-\sigma,\sigma)
    \end{equation}
    where $z_{\mathrm{in}}$ is a $C^1$ function, which is transverse to $W_\mu^s(S^-)\cap \Tilde{\Sigma}_{h}^{<}$ at 
    \begin{equation}\label{eqn: Intersection point ps Theorem 5.3}
        \Tilde{p}_s =\Tilde{\gamma}_{\mathrm{in}}(0) = \left(\delta, 0, z_{\mathrm{in}}(0)\right).
    \end{equation}
    Then, if we restrict the curve $\Tilde{\gamma}_\mathrm{in}$ for $\nu\in (0,\sigma)$, we have
    \begin{itemize}
    \item The transition map $\Tilde{f}\colon \Tilde{\Sigma}_{h}^{<,+} \to \Tilde{\Sigma}_{h}^{>,-}$ maps the curve $\Tilde{\gamma}_\mathrm{in}$    to a curve $\Tilde{\gamma}_{\mathrm{out}} \subset \Tilde{\Sigma}_{h}^{>}$, parameterized as
        \begin{equation*}
    \begin{aligned}
    \Tilde{\gamma}_{\mathrm{out}}(\nu) = \Tilde{f}(\Tilde{\gamma}_{\mathrm{in}}(\nu))= \left(s,\Tilde{\iota}, w\right) = \left(\delta, -\nu +\mathcal{O}_1(\delta \nu), z_{\mathrm{in}}(\nu) + \mathcal{O}_1(\delta^2\nu,\nu^2)\right)\qquad \nu\in (0,\sigma).
    \end{aligned}
    \end{equation*}
Moreover, it has a well-defined limit 
        \begin{equation}\label{eqn: Image of the intersection point Theorem 5.3}
        \Tilde{p}_u =   \lim_{\nu\to 0^+} \Tilde{\gamma}_{\mathrm{out}}(\nu) = \left(\delta, 0, z_{\mathrm{in}}(0)\right) \in \unstableManifoldSplusmu\cap\Tilde{\Sigma}_h^>.
    \end{equation}
    \item 
    The tangent vector $\Tilde{\gamma}_{\mathrm{out}}'(\nu)$ has a well-defined limit as $\nu\to 0$, which is of the form
   \begin{equation*}
    \begin{aligned}
        \Tilde{\gamma}_{\mathrm{out}}'(0) =&\lim_{\nu\to 0^+}\Tilde{\gamma}_{\mathrm{out}}'(\nu)=\left(0,-1 +\mathcal{O}_1(\delta), z_{\mathrm{in}}'(0) + \mathcal{O}_1(\delta)\right).\\
    \end{aligned}
    \end{equation*}   
 and, therefore, $\Tilde{\gamma}_{\mathrm{out}}$   is transverse to $W_\mu^u(S^+)\cap \Tilde{\Sigma}_{h}^{>}$ at $\Tilde{p}_u$.

    \end{itemize}
\end{theorem}
\begin{remark}\label{rem:extensioftilde}
In view of Theorem    \ref{thm: Transition map close to collision}, and in particular of \eqref{eqn: Image of the intersection point Theorem 5.3}, we can extend continuously the map $\tilde f$ to points in $(\delta,0,z)\in W^s_{\mu}(S^-) \cap \Tilde{\Sigma}_h^>$ as
\[
\tilde f (\delta,0,z)=(\delta,0,z) \in W^u_{\mu}(S^+)\cap \Tilde{\Sigma}_h^>.
\]
\end{remark}

To prove Theorem \ref{thm: Transition map close to collision}, we are going to consider the following intermediate sections

\begin{equation}\label{eqn: Intermediate subsets transition map C1pm and C2pm}
\begin{aligned}
    \Tilde{\Sigma}_{1}^{\pm} &= \{(s,\angleStraightWsSminus,\straightenFiberWsSminus) \colon 0<s<\delta, \angleStraightWsSminus = \pm \delta, \straightenFiberWsSminus\in \mathds{T}\},\\
    \Tilde{\Sigma}_{2}^{\pm} &= \{(s,\angleStraightWuSplus,\straightenFiberWuSplus) \colon 0 <s<\delta, \angleStraightWsSminus(s,\angleStraightWuSplus,\straightenFiberWuSplus) = \pi \pm \delta , \straightenFiberWuSplus\in \mathds{T}\},
\end{aligned}    
\end{equation}
where $\angleStraightWsSminus(s,\Tilde{\iota},w)$ is obtained from Remark \ref{remark: Translation from (s,Tilde(beta),z) to (s,Tilde(iota),w)} (see Figure \ref{fig:Dynamics close to collision}). We split the transition map into three maps: from $ \Tilde{\Sigma}_{h}^{<,+}$ to $ \Tilde{\Sigma}_{1}^{+}$, from $ \Tilde{\Sigma}_{1}^{+}$ to $ \Tilde{\Sigma}_{2}^{-}$ and from $ \Tilde{\Sigma}_{2}^{-}$ to $\Tilde{\Sigma}_{h}^{>,-}$. The other case, i.e., the transition map from $\Tilde{\Sigma}_{h}^{<,-}$ to $\Tilde{\Sigma}_{h}^{>,+}$, is built analogously.






The following lemma (whose proof is done in Appendix \ref{appendix: A lambda lemma for the transition map close to collision}), provides information regarding the  first transition map.

\begin{lemma}\label{lemma: A lambda lemma for the transition map close to collision}
The transition map 
\begin{equation}\label{eqn: Map from Tilde(Cs) to Tilde(C1)}
    \begin{aligned}
        \mathcal{T}_{s,1}^+ \colon\Tilde{\Sigma}_{h}^{<,+}&\to  \Tilde{\Sigma}_{1}^{+}\\
        (\delta,\Tilde{\beta}_0,z_0) &\mapsto (s_1,\delta,z_1)
    \end{aligned}
\end{equation}
takes the curve $\Tilde{\gamma}_{\mathrm{in}}(\nu)$ defined in \eqref{eqn: transverse curve gamma_in(nu)} with $\nu\in (0,\sigma)$ to a curve $\Tilde{\gamma}_1(\nu)$ defined as

\begin{equation}\label{eqn: Image curve gamma_1(nu)}
    \begin{aligned}
        \Tilde{\gamma}_1(\nu) = \left(s_1(\nu),\delta, z_1(\nu)\right),
    \end{aligned}
\end{equation}
with\begin{equation*}
    \begin{aligned}
        s_1(\nu) &= \nu\left(1+\mathcal{O}_1(\delta)\right)\\
        z_1(\nu) &= z_{\mathrm{in}}(\nu) + \mathcal{O}_1(\delta^2 \nu),
    \end{aligned}
\end{equation*}
where $z_{\mathrm{in}}$ is the $z$-component of the curve $\Tilde{\gamma}_{\mathrm{in}}(\nu)$ described in \eqref{eqn: transverse curve gamma_in(nu)}.

Moreover, for $\nu\in (0,\sigma)$, its tangent vector is of the form

\begin{equation}\label{eqn: Image tangent vector gamma1}
    \begin{aligned}
       \Tilde{\gamma}_1'(\nu) = \left(s_1'(\nu),0,z_1'(\nu)\right),
    \end{aligned}
\end{equation}
where
\begin{equation*}
    \begin{aligned}
        s_1'(\nu) &= 1 + \mathcal{O}_1(\delta)\\
        z_1'(\nu) &= z_{\mathrm{in}}'(\nu) + \mathcal{O}_1(\delta)
    \end{aligned}
\end{equation*}
and it has a well-defined limit as $\nu\to 0$. Thus, $\Tilde{\gamma}_1$ is transverse to $\unstableManifoldCollisionSminusmu \cap \Tilde{\Sigma}_{1}^{+}$ at $\underset{\nu\to 0^+}{\lim}\Tilde{\gamma}_1(\nu)=(0,\delta, z_{\mathrm{in}}(0))$.
\end{lemma}
The next step is to define a map from $ \Tilde{\Sigma}_{1}^{+}$ to $ \Tilde{\Sigma}_{2}^{-}$ (see \eqref{eqn: Intermediate subsets transition map C1pm and C2pm}), which we denote by $\mathcal{T}^+_{1,2}$. To this end, we consider $\varepsilon>0$ small enough and the following domain

\begin{equation*}
    U = \bigg\{(s,\Tilde{\beta},z) \colon \delta \leq \Tilde{\beta}\leq \pi -\delta, 0<s<\varepsilon, |z-z_1(\nu)|< \varepsilon, \  \nu\in (0,\sigma)\bigg\},
\end{equation*}
where $z_1$ is given in \eqref{eqn: Image curve gamma_1(nu)}. 
\begin{figure}[h!]
    \centering
    \includegraphics[width=0.75\linewidth]{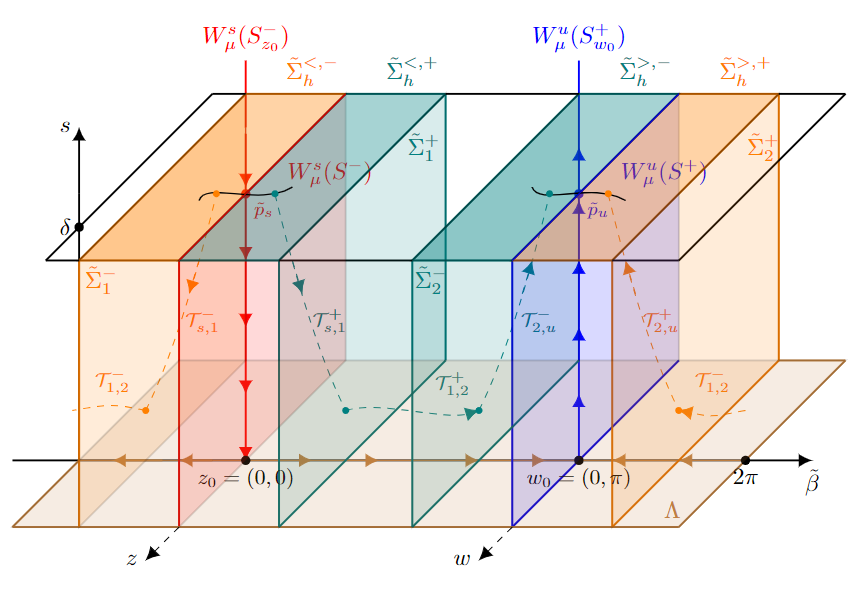}
    \caption{Representation of the dynamics near the collision manifold. The line $\Tilde{\beta} = 2\pi$ is identified with $\Tilde{\beta} = 0$. The transition maps are given by the composition of the maps $\mathcal{T}_{2,u}^{\mp} \circ \mathcal{T}_{1,2}^\pm \circ \mathcal{T}_{s,1}^\pm$. Recall that we are considering different systems coordinates in neighborhoods of the two circles $S^\pm$.}
    \label{fig:Dynamics close to collision}
\end{figure}

Note that the values of the variable $\angleStraightWsSminus$ in $U$ are no longer small. However, since  $s$  is close to $0$, the diffeomorphism $\Gamma_-$ in Proposition \ref{proposition: Straightening local diffeomorphisms and motions}, which is linear in $\Tilde{\beta}$ and  straightens the invariant manifolds of $S^-$, is still well-defined. 
Performing the changes \eqref{eqn: Change of coordinates beta-Tilde(beta)} and \eqref{eqn: Change of coordinates y-z} to  system \eqref{eqn: Eq motion (s,beta,y) taylor s}, as we did in  the proof of Proposition \ref{proposition: Straightening local diffeomorphisms and motions}, but without Taylor expanding in the $\tilde \beta$ coordinate, we obtain the system
\begin{equation}\label{eqn: Eq motion in coordinates (s,Tilde(beta),z) from C1 to C2}
    \begin{aligned}
        s' &= -\frac{m_0}{2}s\left(\cos\Tilde{\beta} + \Tilde{G_1}(s,\Tilde{\beta},z)\right)\\
        \Tilde{\beta}{}' &=\frac{m_0}{2}\sin\Tilde{\beta} + \Tilde{G_2}(s,\Tilde{\beta},z) \\
        z'&= s\left(\lambda(\mu,h)s\sin\Tilde{\beta} + \Tilde{G_3}(s,\Tilde{\beta},z)\right),
    \end{aligned}
\end{equation}
where
\begin{equation*}
    \Tilde{G_1}(s,\Tilde{\beta},z) = \mathcal{O}_2(s), \quad \Tilde{G_2}(s,\Tilde{\beta},z) = \mathcal{O}_2(s), \quad \Tilde{G_3}(s,\Tilde{\beta},z) = \mathcal{O}_3(s).
\end{equation*}

The following lemma (whose proof is done in Appendix \ref{appendix: An implicit function theorem for the transition map close to collision}), analyzes the transition map $\mathcal{T}_{1,2}^+$.

\begin{lemma}\label{lemma: Expression of gamma_2(nu) in coordinates (s,Tilde(beta),z)}
    The transition map $\mathcal{T}^{+}_{1,2}$ from $ \Tilde{\Sigma}_{1}^{+}$ to $ \Tilde{\Sigma}_{2}^{-}$, defined in \eqref{eqn: Intermediate subsets transition map C1pm and C2pm} and expressed in coordinates $(s,\Tilde{\beta},z)$,
    \begin{equation}\label{eqn: transition map T12}
        \begin{aligned}
            \mathcal{T}_{1,2}^{+}\colon  \Tilde{\Sigma}_{1}^{+} &\to  \Tilde{\Sigma}_{2}^{-}\\
            (s_1,\delta,z_1) &\mapsto (s_2,\pi -\delta,z_2),
        \end{aligned}
    \end{equation}
    maps the curve $\Tilde{\gamma}_1$ in \eqref{eqn: Image curve gamma_1(nu)} to a curve $\Tilde{\gamma}_2$ given by
    \begin{equation}\label{eqn: Curve gamma_2(nu)}
        \Tilde{\gamma}_2(\nu) = (s_2(\nu),\pi-\delta, z_2(\nu)), \quad \nu \in (0,\sigma)
    \end{equation}
    with
    \begin{equation}\label{eqn: Curve gamma_2(nu)eqn: Curve gamma_2(nu) in coordinates (s,tilde(beta),z)}
    \begin{aligned}
        s_2(\nu) &=\nu\left(1+\mathcal{O}_1(\delta)\right),\\
        z_2(\nu) &= z_{\mathrm{in}}(\nu) + \mathcal{O}_1(\delta^2\nu, \nu^2),
    \end{aligned}
    \end{equation}
    where $z_{\mathrm{in}}$ is to the $z$-component of the curve $\Tilde{\gamma}_{\mathrm{in}}$ in \eqref{eqn: transverse curve gamma_in(nu)}.
    
    Moreover, for $\nu\in (0,\sigma)$, its tangent vector is of the form
    \begin{equation}\label{eqn: Tangent vector of gamma_2}
        \Tilde{\gamma}_2'(\nu) = \left(s_2'(\nu), 0 , z_2'(\nu)\right)
    \end{equation}
    and
    \begin{equation*}
    \begin{aligned}
    s_2'(\nu) &= 1 + \mathcal{O}_1(\delta), \\
    z_2'(\nu) &= z_{\mathrm{in}}'(\nu) + \mathcal{O}_2(\delta).
    \end{aligned}
    \end{equation*} 
    and it has a well-defined limit as $\nu\to 0$. Thus, the curve $\Tilde{\gamma}_2$ is transverse to $\stableManifoldCollisionSplusmu \cap \Tilde{\Sigma}_{2}^{-}$ at  $\underset{\nu\to 0^+}{\lim}\Tilde{\gamma}_2(\nu)=(0,\pi-\delta,z_{\mathrm{in}}(0))$.
\end{lemma}
Note that $\Tilde{\gamma}_2$ is expressed in coordinates $(s,\Tilde{\beta},z)$, whereas the section $ \Tilde{\Sigma}_{2}^{-}$ is defined in coordinates $(s,\Tilde{\iota},w)$. The following corollary, whose proof is straightforward from Remark \ref{remark: Translation from (s,Tilde(beta),z) to (s,Tilde(iota),w)}, gives an expression of the curve $\Tilde{\gamma}_2$ of Lemma \ref{lemma: Expression of gamma_2(nu) in coordinates (s,Tilde(beta),z)} in the latter variables.

\begin{corollary}\label{corollary: transition map from C1+ to C2-}
     The curve $\Tilde{\gamma}_2$ given in \eqref{eqn: Curve gamma_2(nu)} for $\nu \in (0,\sigma)$ can be expressed in coordinates $(s,\Tilde{\iota},w)$ as

    \begin{equation}\label{eqn: Curve gamma_2(nu) in coordinates (s,tilde(iota),w)}
        \Tilde{\gamma}_2(\nu) = (s_2(\nu),\Tilde{\iota}_2(\nu), w_2(\nu))
    \end{equation}
    with
    \begin{equation*}
    \begin{aligned}
        s_2(\nu) &= \nu\left(1+\mathcal{O}_1(\delta)\right),\\
        \Tilde{\iota}_2(\nu) &= -\delta + \mathcal{O}_2(\nu),\\
        w_2(\nu) &= z_{\mathrm{in}}(\nu) + \mathcal{O}_1(\delta^2\nu, \nu^2),
    \end{aligned}
    \end{equation*}
    and $z_{\mathrm{in}}$ is to the $z$-component of the curve $\Tilde{\gamma}_{\mathrm{in}}$ in \eqref{eqn: transverse curve gamma_in(nu)}.
    
    Moreover,  its tangent vector is of the form

    \begin{equation}
        \Tilde{\gamma}_2'(\nu) = \left(s_2'(\nu), 0 , w_2'(\nu)\right)
    \end{equation}
    where

    \begin{equation*}
    \begin{aligned}
    s_2'(\nu) &= 1 + \mathcal{O}_1(\delta) \\
    w_2'(\nu) &= z_{\mathrm{in}}'(\nu) + \mathcal{O}_1(\delta)
    \end{aligned}
    \end{equation*}
    and has a well-defined limit as $\nu\to 0$. Thus, the curve $\Tilde{\gamma}_2$ is transverse to $\stableManifoldCollisionSplusmu \cap \Tilde{\Sigma}_{2}^{-}$ at  $\underset{\nu\to 0^+}{\lim}\Tilde{\gamma}_2(\nu)=(0,-\delta,z_{\mathrm{in}}(0))$.
\end{corollary}

Finally, the following lemma (whose proof is done in Appendix \ref{appendix: A lambda lemma for the transition map close to collision}), defines the transition map from $ \Tilde{\Sigma}_{2}^{-}$ to $\Tilde{\Sigma}_{h}^{>,-}$ in an analogous way as in Lemma \ref{lemma: A lambda lemma for the transition map close to collision}.

\begin{lemma}\label{lemma: transition map from C2- to Cu}
   The transition map 
   \begin{equation}\label{eqn: map from C2 to Cu}
       \begin{aligned}
           \mathcal{T}^-_{2,u}\colon  \Tilde{\Sigma}_{2}^{-} &\to \Tilde{\Sigma}_{h}^{>,-}\\
           (s_2,\Tilde{\iota}_2,w_2) &\mapsto (\delta, \Tilde{\iota}_u,w_u) 
       \end{aligned}
   \end{equation}
 (see \eqref{eqn: Intermediate subsets transition map C1pm and C2pm} and \eqref{eqn: definition of Tilde(Csu+-)})   takes the curve $\Tilde{\gamma}_2$ in \eqref{eqn: Curve gamma_2(nu) in coordinates (s,tilde(iota),w)} with $\nu\in (0,\sigma)$ to a curve $\Tilde{\gamma}_{\mathrm{out}}$ defined as

\begin{equation}\label{eqn: Image curve gamma_out(nu)}
    \begin{aligned}
       \Tilde{\gamma}_{\mathrm{out}}(\nu) = \left(\delta,\Tilde{\iota}_{\mathrm{out}}(\nu), w_{\mathrm{out}}(\nu)\right)
    \end{aligned}
\end{equation}
such that

\begin{equation*}
    \begin{aligned}
        \Tilde{\iota}_{\mathrm{out}}(\nu) &= -\nu + \mathcal{O}_1(\delta\nu)\\
        w_{\mathrm{out}}(\nu) &= z_{\mathrm{in}}(\nu) + \mathcal{O}_1(\delta^2 \nu,\nu^2).
    \end{aligned}
\end{equation*}
%
Moreover, its tangent vector is of the form
\begin{equation*}
    \begin{aligned}
        \Tilde{\gamma}_{\mathrm{out}}'(\nu) = \left(0,\Tilde{\iota}_{\mathrm{out}}'(\nu),w_{\mathrm{out}}'(\nu)\right)
    \end{aligned}
\end{equation*}
where
\begin{equation*}
    \begin{aligned}
        \Tilde{\iota}_{\mathrm{out}}'(\nu) &=-1 + \mathcal{O}_1(\delta)\\
        w_{\mathrm{out}}'(\nu) &= z_{\mathrm{in}}'(\nu) + \mathcal{O}_1(\delta)
    \end{aligned}
\end{equation*}
and its limit as $\nu\to 0$ is well-defined and, therefore,  $\Tilde{\gamma}_{\mathrm{out}}$ is transverse to $\unstableManifoldSplusmu$ at $\underset{\nu\to0}{\lim} \Tilde{\gamma}_{\mathrm{out}}(\nu)=(\delta,0,z_{\mathrm{in}}(0)) = \Tilde{p}_u$.
\end{lemma} 

To complete the proof of Theorem \ref{thm: Transition map close to collision}, we observe that the transition map $\Tilde{f}$ is obtained as the composition of the maps from Lemmas \ref{lemma: A lambda lemma for the transition map close to collision}, \ref{lemma: Expression of gamma_2(nu) in coordinates (s,Tilde(beta),z)} and \ref{lemma: transition map from C2- to Cu} respectively, i.e.

\begin{equation*}
\Tilde{f} =  \mathcal{T}^-_{2,u} \circ \mathcal{T}^{+}_{1,2}\circ \mathcal{T}^+_{s,1}.
\end{equation*}

\section{Proof of Theorem \ref{thm: Existence of parabolic, oscillatory and periodic orbits for small energies}}\label{sec: Proof of Theorem parabolic orbits}
The purpose of this section is to prove Theorem \ref{thm: Existence of parabolic, oscillatory and periodic orbits for small energies}. To this end, we recall the following notation which is used along the section:
\begin{itemize}
    \item The curves $\Delta_\infty^{s,u}(\mu), \Delta_{\circleSplus}^u(\mu)$ and $\Delta_{\circleSminus}^s(\mu)$ defined in \eqref{eqn: Notation for the intersection of the invariant manifolds with the section r=2} are the intersection of the corresponding invariant manifolds with $\Sigma_h$. They admit a graph parameterization in polar coordinates $(\theta,\Theta)$, as shown in \eqref{eqn: definition of the intersection of the invariant manifolds of collision with delta2 as graphs} and \eqref{eqn: definition of the intersection of the invariant manifolds of infinity with delta2 as graphs}.
    
    \item The sets $I_D^{\pm}(w_\Sigma)$, defined in \eqref{eqn: domain Ip} and \eqref{eqn: domain Im} (with $w_\Sigma$ as in \eqref{eqn: definition of wSigma}), where the graphs of the curves $\Delta_\infty^{s,u}(\mu)$ are defined.
    
    \item The graph parameterizations $\Theta^{s}_{\infty}(\theta,\hat{\Theta}_0), \Theta_{\circleSplus}^u(\theta)$ (defined in \eqref{eqn: Value of the angular momentum of the Ws(infty) at r=2} and \eqref{eqn: Value of the angular momentum of Wu(S+) at r=2} respectively) and $\Theta^{u}_{\infty}(\theta,\hat{\Theta}_0)$, $ \Theta_{\circleSminus}^s(\theta)$ (obtained from the symmetries in \eqref{eqn: Angular momentum of Wu(infty) as a graph of (w,theta)} and \eqref{eqn: Angular momentum of Ws(S-) as a graph of theta}).
    
\end{itemize}
The key step in the proof of Theorem \ref{thm: Existence of parabolic, oscillatory and periodic orbits for small energies} is Proposition \ref{prop: triple intersection} below. To state it, we introduce the terminology \emph{triple intersection} between invariant manifolds. Certainly, this is an abuse of language, since it is well known that two stable (or unstable) manifolds cannot intersect. Therefore, let us explain what we mean by that. 

\begin{definition}\label{def:tripleinter}
We say that the invariant manifolds $\unstableManifoldSplusmu$, $\Tilde{W}_\mu^u(\Alpha_{\hat{\Theta}_0^*})$ and $\Tilde{W}_\mu^s(\Alpha_{\hat{\Theta}_0^*})$ have a triple intersection  at  $p_>^* = (\theta_>, \Theta_>) \in \Sigma_{h^*}^>$ (see \eqref{eqn: definition of Sigma2Rpos and Sigma2Rneg}) if 
\begin{itemize}
    \item $\Delta_{S^+}^u(\mu)$ and $\Delta_\infty^s(\mu)$  intersect  at  $p_>^* = (\theta_>, \Theta_>) \in \Sigma_{h^*}^>$,
    \item $\Delta_{S^-}^s(\mu)$ and $\Delta_\infty^u(\mu)$  intersect  at  $p_<^* = (\theta_<, \Theta_<) \in \Sigma_{h^*}^<$ (in the usual sense),
    \item The point $p_>^*$ belongs to the unstable  fiber within $\unstableManifoldSplusmu$ of a point $(0,\bar\theta,\pi/2)\in S^+$ and the point $p_<^*$ belongs to the stable fiber within $\stableManifoldSminusmu$ of a point $(0,\bar\theta,-\pi/2)\in S^-$.
\end{itemize}
\end{definition}
Note that this can be phrased as saying that the continuous extension of the local map $\tilde f$ given in Theorem \ref{thm: Transition map close to collision} (see Remark \ref{rem:extensioftilde}) maps $p_<^*$ to  $p_>^*$. 
Moreover, note that $\tilde f$ maps $\Delta_\infty^u(\mu)\setminus \{p_<^*\}$ to a curve in $\Sigma^>_{h^*}$ which as $p_>^*$ as endpoint. Then, we say that the triple intersection is transverse if this curve and the curves $\Delta_\infty^{s}(\mu), \Delta_{\circleSplus}^u(\mu)$ intersect pairwise transversally at $p_>^*$.

Next proposition proves the existence of a transverse triple intersection at a suitable energy level. Note that, since we parameterize the curves as graphs, the angles are taken in $[-\pi/2,\pi/2]$.

\begin{proposition}\label{prop: triple intersection}
    There exist $\mu_0,\delta_0 > 0$ such that, for any $\mu\in(0,\mu_0)$, $\delta\in (0,\delta_0)$, there is an angular momentum $\hat{\Theta}_0^*$ satisfying 
    \[\hat{\Theta}_0^* = \hat{\Theta}_0 ^* (\mu,\delta) = \mu\left(\mathcal{M}_+(0) + \mathcal{O}_1(\mu,\delta^2)\right)\] (where $\mathcal{M}_+$ is defined in \eqref{eqn: melnikov function}) such that, for $h= h^*(\mu,\delta) = -\hat{\Theta}_0^*$
    \begin{itemize}
    \item The invariant manifold $\Tilde{W}_\mu^u(\Alpha_{\hat{\Theta}_0^*})$ intersects the section  $\Sigma_{h^*}^>$ and this intersection can be written locally as graph as $(\theta, \Theta^{u,>}_\infty(\theta))$ in a neighborhood of the image of the point $p_<^*\in \Delta_\infty^u(\mu) \transv \Delta_{S^-}^s(\mu)$ under the (continuous extension of the) local map provided by Theorem \ref{thm: Transition map close to collision}.
        \item The invariant manifolds $\unstableManifoldSplusmu$, $\Tilde{W}_\mu^u(\Alpha_{\hat{\Theta}_0^*})$ and $\Tilde{W}_\mu^s(\Alpha_{\hat{\Theta}_0^*})$  intersect transversally at  $p_>^* = (\theta_>, \Theta_>) \in \Sigma_{h^*}^>$ (see \eqref{eqn: definition of Sigma2Rpos and Sigma2Rneg}) where
        \begin{equation*}
            \theta_> = \mathcal{O}_1(\mu,\delta^2),\quad \Theta_> = \Theta_\infty^s(\theta_>,\hat{\Theta}_0^*) = \Theta_{\circleSplus}^u(\theta_>) = \Theta_\infty^{u,>}(\theta_>).
        \end{equation*}
        \item For a fixed $\varepsilon > 0$ small enough independent of $\mu$, denote by $B_\varepsilon(p_>^*) \subset \Sigma_{h^*}^>$ a $\varepsilon$-neighborhood of $p_>^*$, and define the following $C^1$-curves parameterized by $\theta$
        \begin{equation}\label{eqn: transversecurves}
        \begin{aligned}
            \gamma_\infty^{s,>}(\theta)=& \Tilde{W}_\mu^s(\Alpha_{\hat{\Theta}_0^*}) \cap B_\varepsilon(p_>^*),\\
            \gamma_\infty^{u,>}(\theta)= &\Tilde{W}_\mu^u(\Alpha_{\hat{\Theta}_0^*}) \cap B_\varepsilon(p_>^*),\\
            \gamma_{\circleSplus}^{u,>}(\theta)=& \unstableManifoldSplusmu\cap B_\varepsilon(p_>^*).
        \end{aligned}
        \end{equation}
        Hence, the angles
        \begin{equation}\label{eqn: angles trasnversality}
            A  = \measuredangle (\gamma_{\infty}^{u,>}(\theta_>)', \gamma_{\infty}^{s,>}(\theta_>)'),\quad B  = \measuredangle (\gamma_{\infty}^{u,>}(\theta_>)', \gamma_{\circleSplus}^{u,>}(\theta_>)')
        \end{equation}
        (taken in $[-\pi/2,\pi/2]$) satisfy
        \begin{equation}\label{eqn: Inequality angles}
         -\frac{\pi}{2}<  A <  B < 0,
        \end{equation}
        leading to the configuration depicted in Figure \ref{fig: trans intersection}.
    \end{itemize} 
\end{proposition}
The proof of this proposition is divided into two parts: in Section \ref{subsec: First part of the proof of Theorem 1.4} we prove that the triple intersection $p_>^*$ (in the sense of Definition \ref{def:tripleinter})  exists and in Section \ref{subsec: Second part of the proof of Theorem 1.4} we establish the transversality between the curves in \eqref{eqn: transversecurves} in $p_>^*$ as well as the ordering given by \eqref{eqn: Inequality angles}. Now we prove Theorems \ref{thm: Existence of parabolic, oscillatory and periodic orbits for small energies} and  \ref{thm:chaos} relying on Proposition \ref{prop: triple intersection}.

\begin{proof}[Proof of Theorems \ref{thm: Existence of parabolic, oscillatory and periodic orbits for small energies} and Theorem \ref{thm:chaos}]

At $h=h^*$, there exist transverse homoclinic point between the stable invariant manifold of infinity and the unstable manifold of the collision set $\mathcal{S}$ and between the unstable manifold of infinity and the stable manifold of $\mathcal{S}$ which  guarantees the existence of a parabolic ejection-collision orbit. Moreover, this transversality ensures that, for $h= h^*$, the manifolds $\Tilde{W}_\mu^s(\Alpha_{\hat{\Theta}_0}), \Tilde{W}_\mu^u(\Alpha_{\hat{\Theta}_0})$ also intersect transversally (in the sense of Definition \ref{def:tripleinter}). 

Then, one can adapt the construction done by Moser in \cite{moser2001stable} to our setting (see also the work of Simó and Llibre in \cite{MR0573346}; and the work of Guardia, Martin and Seara in \cite{MR3455155}) to prove the existence of a hyperbolic set whose dynamics is conjugated to the shift of infinite symbols and which accumulates to the invariant manifolds of infinity. This leads to the existence of any combination of  hyperbolic, parabolic, periodic and oscillatory orbits in forward and backward time at the energy level $h^*$.

The classical Moser approach requires a transverse homoclinic point of the invariant manifolds of infinity. This does not exist in the present setting since we only have a transverse homoclinic to infinity in the sense of Definition \ref{def:tripleinter}. However, the Moser approach can be easily adapted to this setting.

First, recall that at the energy level $h= h^*$, the curves $\gamma_\infty^{s,>}$, $\gamma_\infty^{u,>}$ and $\gamma_{\circleSplus}^{u,>}$ intersect transversally at $p_>^*$ (in the sense of Definition \ref{def:tripleinter}) and have the ordering depicted in Figure \ref{fig: trans intersection}. 

\begin{figure}[h!]
        \centering
        \begin{overpic}[scale=0.7]{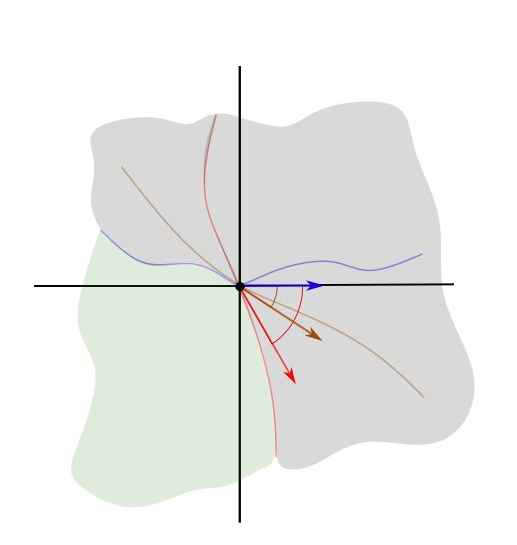}
            \put (10,90){$\Sigma_{h^*}^>$}
            \put (80,44) {$\theta$}
            \put (45,85) {$\Theta$}
            \put (39,44) {$p_>^*$}
            \put (58,44) {\color{blue} \footnotesize $\gamma_\infty^{u,>}{}'(\theta_>)$}
            \put (52,26) {\color{red} \footnotesize $\gamma_{\infty}^{s,>}{}'(\theta_>)$}
            \put (55,34)  {\color{RawSienna} \footnotesize $\gamma_{\circleSplus}^{u,>}{}'(\theta_>)$}
            \put (51,44)  {\color{brown} \footnotesize $B$}
            \put (52,37) {\color{red}\footnotesize $A$}
        \end{overpic}
        \caption{Representation of the ``ordering of the invariant manifolds'' given by the angles $A$ and $B$ in \eqref{eqn: angles trasnversality}. The region shaded in green represents the set $\Tilde{D}^<\cap D^>$  where the forward and backward return maps to the section $\Sigma_{h^*}^>$ are well defined. That is,  the region where the construction made by Moser in \cite{moser2001stable} is performed.}
        \label{fig: trans intersection}
    \end{figure}
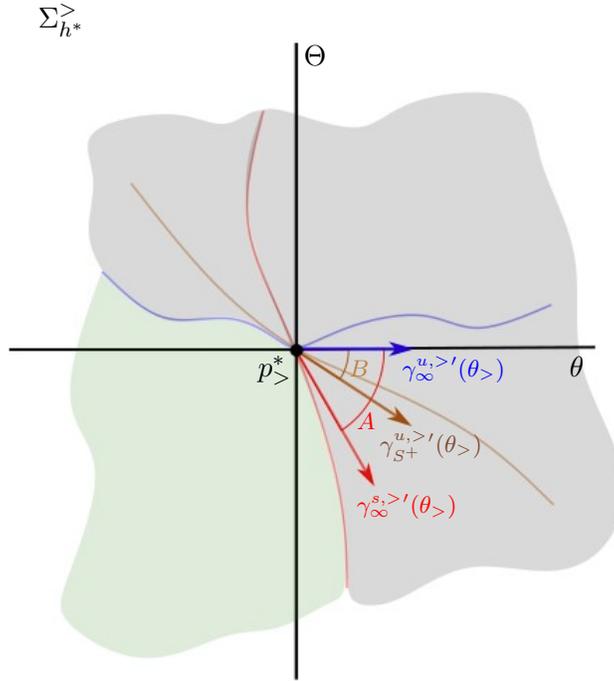
    
To apply the construction 
made by Moser in \cite{moser2001stable} we need to characterize the domain in  $\Sigma_{h^*}^>$ where one can define a return map and the domain such that if one runs  the flow (either forward or backward) with initial condition in it, one hits again $\Sigma_{h^*}^>$.
To this end, we start by defining the domains $D^\diamond\subset \Sigma_{h^*}^\diamond$, $\diamond=<,>$ as follows. Fix $\varepsilon>0$, then
\begin{itemize}
    \item $D^>$ contains points in $\Sigma_{h^*}^>$, $\varepsilon$-close to $p_>^*$ (see \eqref{eqn: Point p>*}), whose forward orbit hits $\Sigma_{h^*}^<$.
    \item  $D^<$ contains points in $\Sigma_{h^*}^<$, $\varepsilon$-close to $p_<$ (see \eqref{eqn: Point p<}), whose backward orbit hits $\Sigma_{h^*}^>$.
\end{itemize}
These domains, proceeding as in \cite{moser2001stable}, can be characterized as 
\begin{equation*}
    \begin{aligned}
        D^> &= \Big\{(\theta,\Theta) \in B_\varepsilon(p_>^*) \colon R(\delta^2,\theta,\Theta;h^*) < R_\infty^{s}(\theta)\Big\} \subset \Sigma_{h^*}^>,\\
        D^< &= \Big\{(\theta,\Theta) \in B_\varepsilon(p_<) \colon R(\delta^2,\theta,\Theta;h^*) > R_\infty^u(\theta)\Big\}\subset \Sigma_{h^*}^<.
    \end{aligned}
\end{equation*}
where 
\begin{equation*}
 R_\infty^{s}(\theta) = R(\delta^2, \theta, \Theta_{\infty}^s(\theta), h^*) > 0 , \quad R_\infty^{u}(\theta) = R(\delta^2, \theta, \Theta_{\infty}^u(\theta), h^*) < 0 
\end{equation*}
are obtained from the Hamiltonian \eqref{eqn: Hamiltonian function in rotating polar coordinates centered at $P_1$}. Alternatively, by the conservation of the Hamiltonian \eqref{eqn: Hamiltonian function in rotating polar coordinates centered at $P_1$} they can be also defined as  
\begin{equation*}
    \begin{aligned}
        D^> &= \Big\{(\theta,\Theta) \in B_\varepsilon(p_>^*) \colon \Theta < \Theta_\infty^{s}(\theta)\Big\} \subset \Sigma_{h^*}^>,\\
        D^< &= \Big\{(\theta,\Theta) \in B_\varepsilon(p_<) \colon \Theta > \Theta_\infty^u(\theta)\Big\}\subset \Sigma_{h^*}^<.
    \end{aligned}
\end{equation*}
As the points in $D^<$ are close  to $p_<$, one can map this domain to $\Sigma_{h^*}^>$ by means of the (continuous extension of the) local map analyzed in Theorem \ref{thm: Transition map close to collision}. This gives the domain $\tilde D^<$, which can be characterized as
\begin{equation}\label{set Tilde D}
\tilde D^< = \Big\{(\theta,\Theta) \in B_\varepsilon(p_<) \colon \Theta < \Theta_\infty^{u,>}(\theta)\Big\}\subset \Sigma_{h^*}^>.
\end{equation}
Hence, proceeding as   Moser in \cite{moser2001stable} one has to analyze the return map in domain $D^>\cap\Tilde{D}^{<}$ and prove that one can construct a hyperbolic set in this domain, which accumulates to the transverse homoclinic point to infinity. Note that in this domain the return map is well defined. The only difference between the classical setting and the present one is the passage close to collision. However, by Theorem \ref{thm: Transition map close to collision}  we have a very precise description of local map close to collision, which in good coordinates is close to the identity (in $\mathcal{C}^1$ regularity). This allows to apply the approach in \cite{moser2001stable}, which leads to a hyperbolic set homeomorphic to $\mathbb{N}^\mathbb{Z}$ whose dynamics is conjugated to the shift of infinite symbols. Moreover, this set contains the homoclinic point in its closure. Indeed, the symbols in the sequence $\mathbb{N}^\mathbb{Z}$ labels the closeness of the corresponding orbit to the invariant manifolds. The higher the symbol, the closer is the point to the invariant manifold. 

Moreover, in the present setting the homoclinic point is a triple intersection (in the sense of Definition \ref{def:tripleinter}). This implies that the hyperbolic set also accumulates to the invariant manifolds of collision. 

Then, unbounded (forward and backward) sequences in $\mathbb{N}^\mathbb{Z}$ correspond to oscillatory orbits which accumulate also to collision. Periodic sequences correspond to periodic orbits which by taking high enough symbols can be as large and  as close to collision as prescribed. This completes the proof.


\end{proof}

\subsection{Existence of the triple intersection 
}\label{subsec: First part of the proof of Theorem 1.4}

Fix $\mu>0$ small enough. Then, for any $\theta^+ \in U_+$, fix an energy level $h=-\hat \Theta _0= -\mu \bar \Theta_0=\mu M_+(\theta^+-w_\Sigma)$ so that Theorem \ref{thm: transverse intersection of the invariant manifolds} holds and consider the point
\begin{equation}\label{eqn: Point p<}
    p_< = p_<(\theta^+,\mu) = (\theta_<(\theta^+,\mu),\Theta_<(\theta^+,\mu)) \in \Delta_{\circleSminus}^s(\mu) \transv \Delta_{\infty}^u(\mu),
\end{equation}
where $\Theta_<(\theta^+,\mu) = \Theta_\infty^u(\theta_<(\theta^+,\mu),\hat{\Theta}_0) = \Theta_{\circleSminus}^s(\theta_<(\theta^+,\mu))$.

Since the changes of coordinates $\psi, \Tilde{\psi}$ (defined in \eqref{eqn:McGehee map Collision} and \eqref{eqn:polar change of coordinates for u,v into rho alpha}) and $\Gamma_-$ (defined in \eqref{eqn: Diffeo (s,alpha,theta)-(s,Tilde(beta),z)}) are diffeomorphisms, we can apply Theorem \ref{thm: Transition map close to collision} and Remark \ref{rem:extensioftilde} 
so that the (continuous extension of the) transition map $\Tilde{f}$ 
sends 
$\Tilde{p}_s = \left(\Gamma_ - \circ \Tilde \psi \circ \psi \right)(p_<)$ to a point  $\Tilde{p}_u\in\Tilde{\Sigma}_h^> \cap W^u_\mu (S^+)$ defined in \eqref{eqn: Image of the intersection point Theorem 5.3}. Then, applying the change of coordinates $\Gamma_+$ in \eqref{eqn: Diffeo (s,alpha,theta)-(s,iota,w)}, this point can be written as $p^u_> = (\theta^u_>,\Theta^u_>) =  \left(\psi^{-1}\circ \Tilde \psi^{-1}\circ \Gamma_+ ^{-1}\right)(\tilde p_u)$, and therefore  $\Theta^u_>=\Theta ^u_{S^+}(\theta_>^*)$. Abusing notation (in the same sense as in Definition \ref{def:tripleinter})   we say that $p^u_>  \in \Delta_{\circleSplus}^u(\mu) \transv \reducedUnstableManifoldInfinitymu$. In particular, the $\theta$-component of both points can be related as
\begin{equation}\label{eqn: Relation theta>-theta<}
    \theta_>^u = \theta_>^u(\theta^+,\mu,\delta) = \theta_<(\theta^+,\mu) + \mathcal{O}_2(\delta).
\end{equation}
To guarantee the triple intersection we need that $p_>^u = p_>(\theta^+,\mu) \in \Delta_{\circleSplus}^u(\mu) \transv \Delta_\infty^s(\mu)$ from Theorem \ref{thm: transverse intersection of the invariant manifolds}. Namely, we look for $\theta^+$ such that
\begin{equation}\label{eqn: thetas}
\theta_>^u(\theta^+,\mu,\delta) = \theta_>(\theta^+,\mu)
\end{equation}
which is equivalent to solve
\begin{equation}\label{eqn: IFT triple intersection}
  \theta^+ + \mathcal{O}_1(\mu,\delta^2) = 0,  
\end{equation}
given that $w_\Sigma$ in \eqref{eqn: definition of wSigma} is of order $\delta^3$.

Hence, the Implicit Function Theorem ensures that there exist $\mu_0,\delta_0 > 0$ such that, for all $(\mu,\delta) \in (0,\mu_0)\times (0,\delta_0)$, there exists a unique $\theta^+(\mu,\delta)$ with $\theta^+(0,0) = 0$ satisfying \eqref{eqn: IFT triple intersection} and therefore \eqref{eqn: thetas}. Since $\theta^+(\mu,\delta) \in U^+$ for $\mu,\delta > 0$ small enough (see Theorem \ref{thm: transverse intersection of the invariant manifolds}), the point
\begin{equation}\label{eqn: Point p>*}
    p_>^* = (\theta_>(\mu,\delta), \Theta_>(\mu,\delta)),\quad \Theta_>(\mu,\delta) = \Theta_{\circleSplus}^u(\theta_>(\mu,\delta))
\end{equation}
belongs to $\Tilde{W}_\mu^s(\Alpha_{\hat{\Theta}_0^*}) \cap \Tilde{W}_\mu^u(\Alpha_{\hat{\Theta}_0^*}) \cap \unstableManifoldSplusmu$ at 
\begin{equation}\label{eqn: h transversality}
h^* = h(\mu,\delta) = - \hat{\Theta}_0^*(\mu,\delta) = \mu M_+(\theta^+(\mu,\delta)-w_\Sigma) = \mu(\mathcal{M}_+(0) + \mathcal{O}_1(\mu,\delta^2)).
\end{equation}

\subsection{Transversality and ordering at  the triple intersection
}\label{subsec: Second part of the proof of Theorem 1.4}


From now on, we consider $h^*$ as in \eqref{eqn: h transversality}, and use the expression for $\theta_>$ given in Theorem \ref{thm: transverse intersection of the invariant manifolds} with the value of $\theta^+$ obtained in \eqref{eqn: IFT triple intersection} and denote by 

\begin{equation}\label{eqn: Notation transv}
\begin{aligned}
\theta_> &= \theta_>(\mu,\delta) = \theta^+(\mu,\delta) - w_\Sigma + \mathcal{O}_1(\mu) = \mathcal{O}_1(\mu,\delta^2) \in I_{\frac{3}{10}}^+(w_\Sigma)\cap I^-_{\frac{3}{10}}(w_\Sigma),\\
\theta_< &= \theta_<(\mu,\delta) = \theta_>(\mu,\delta) + \mathcal{O}_2(\delta) = \theta_> + \mathcal{O}_2(\delta) = \mathcal{O}_1(\mu,\delta^2) \in I_{\frac{3}{10}}^+(w_\Sigma)\cap I_{\frac{3}{10}}^-(w_\Sigma),\\
d_\pm(\theta)&=d_\pm(\theta,\hat{\Theta}_0^*), \quad \Theta_\infty^{u,s}(\theta) = \Theta_\infty^{u,s}(\theta,\hat{\Theta}_0^*),
\end{aligned}
\end{equation}
the $\theta$-components of $p_<$ and $p_>^*$ in \eqref{eqn: Point p<} and \eqref{eqn: Point p>*} respectively, the distances in \eqref{eqn: Asymptotic formula for the distances} and the values of the angular momenta of the points in $\Delta_\infty^{u,s}(\mu)$. 
For a fixed $\varepsilon > 0$ small enough, we consider an $\varepsilon$-neighborhood of $\theta_>$ containing $\theta_<$ (see \eqref{eqn: Notation transv}) which we denote $B_\varepsilon(\theta_>)$, and  the following $C^1$-curves
\begin{equation}\label{eqn: Curves as graphs of theta in coordinates (theta,Theta)}
    \begin{aligned}
        \gamma^{u,<}_\infty &= \Big\{\left(\theta,\Theta^{u}_\infty(\theta)\right)\colon \theta \in B_\varepsilon(\theta_>)\Big\} \subset \Delta^{u}_\infty(\mu)\subset\Sigma_{h^*}^<,\\
        \gamma^{s,<}_{\circleSminus} &= \Big\{\left(\theta,\Theta^{s}_{\circleSminus}(\theta)\right)\colon \theta \in B_{\varepsilon}(\theta_>)\Big\} \subset \Delta^{s}_{\circleSminus}(\mu)\subset\Sigma_{h^*}^<,\\
        \gamma^{s,>}_\infty &= \Big\{\left(\theta,\Theta^{s}_\infty(\theta)\right)\colon \theta \in B_\varepsilon(\theta_>)\Big\} \subset \Delta^{s}_\infty(\mu)\subset\Sigma_{h^*}^>,\\
        \gamma^{u,>}_{\circleSplus} &= \Big\{\left(\theta,\Theta^{u}_{\circleSplus}(\theta)\right)\colon \theta \in B_{\varepsilon}(\theta_>)\Big\} \subset \Delta^{u}_{\circleSplus}(\mu)\subset\Sigma_{h^*}^>.
    \end{aligned}
\end{equation}
Observe that $\gamma^{u,<}_\infty$, $\gamma^{s,>}_\infty$ and $\gamma^{u,>}_{\circleSplus}$ are the curves in \eqref{eqn: transversecurves}.

Note that, by Theorem \ref{thm: transverse intersection of the invariant manifolds}, these curves satisfy

\begin{equation*}
    p_<(\mu,\delta) \in \gamma^{u,<}_\infty \transv \gamma^{s,<}_{\circleSminus}, \quad p_>(\mu,\delta) \in \gamma^{s,>}_\infty\transv \gamma^{u,>}_{\circleSplus}.
\end{equation*}
Moreover, in Section \ref{subsec: First part of the proof of Theorem 1.4} we have seen that
\begin{equation}\label{eqn: Points of intersection}
\begin{aligned}
    p_>(\mu,\delta) = p_>^* = (\theta_>, \Theta_{\circleSplus}^u(\theta_>)) &\in  \Delta_{\circleSplus}^u(\mu)\cap \Delta_{\infty}^s(\mu) \cap\Tilde{W}_\mu^u(\Alpha_{\hat{\Theta}_0^*}) \in \Sigma_{h^*}^>.
\end{aligned}
\end{equation}
Then, the following lemma (whose proof is done in Appendix \ref{appendix: From straightening coordinates to synodical polar coordinates}) gives us information regarding the image of $\gamma^{u,<}_\infty$ through the transition map $\Tilde{f}$ obtained in Theorem \ref{thm: Transition map close to collision}, translated into coordinates $(\theta,\Theta)$.

\begin{lemma}\label{lemma: tangent vector at Sigma delta,h as a graph}
Denote by 
\begin{equation*}
 \gamma^{u,>}_\infty = \left(\psi^{-1} \circ \Tilde{\psi}^{-1} \circ \Gamma_+^{-1} \circ \Tilde{f} \circ \Gamma_- \circ \Tilde{\psi} \circ \psi\right) \circ \gamma^{u,<}_\infty \subset \Sigma_{h^*}^>
\end{equation*}
where $ \gamma^{u,<}_\infty $ is the curve in \eqref{eqn: Curves as graphs of theta in coordinates (theta,Theta)} and
\begin{itemize}    
    \item The diffeomorphisms $\psi$ and $\Tilde{\psi}$ are defined in \eqref{eqn:McGehee map Collision} and \eqref{eqn:polar change of coordinates for u,v into rho alpha} respectively. 
    
    \item The diffeomorphisms $\Gamma_-$ and $\Gamma_+$ are given in Proposition \ref{proposition: Straightening local diffeomorphisms and motions}.
    
    \item The transition map $\Tilde{f}$ from $\Sigma_{h^*}^<$ to $\Sigma_{h^*}^>$ is given in Theorem \ref{thm: Transition map close to collision}.
\end{itemize}
Then, the curve $\gamma^{u,>}_\infty$ can be written as a graph as

\begin{equation*}
    \gamma^{u,>}_\infty = \Big\{\left(\theta, \Theta^{u,>}_\infty(\theta)\right)\colon \theta \in B_\varepsilon(\theta_>)\Big\},
\end{equation*}
where the function $\Theta_\infty^{u,>}(\theta)$ satisfies

\begin{equation}\label{eqn: Expression of Thetau> infty through f}
    \begin{aligned}
     \Theta^{u,>}_\infty(\theta) &= \Theta^u_{\circleSplus}(\theta) + \mathcal{O}_1(\mu \delta),\\
     \Theta^{u,>}_\infty{}'(\theta_>) &= \Theta^{u}_{\circleSplus}{}'(\theta_>) + d_-'(\theta_>) + \mathcal{O}_1(\mu\delta^2),
    \end{aligned}
\end{equation}
where $\distanceneg(\theta)$ is defined in \eqref{eqn: Notation transv}.
\end{lemma}
As a result, the transversality condition of the triple intersection is guaranteed once we prove that
\begin{equation*}
    \begin{aligned}
        \Theta^u_{\circleSplus}{}'(\theta_>)-\Theta^{s}_{\infty}{}'(\theta_>)&\neq 0\\
        \Theta^s_\infty{}'(\theta_>)-\Theta^{u,>}_{\infty}{}'(\theta_>)  &\neq 0\\
        \Theta^u_{\circleSplus}{}'(\theta_>)-\Theta^{u,>}_{\infty}{}'(\theta_>)&\neq 0.
    \end{aligned}
\end{equation*}
By \eqref{eqn: Explicit formula for the distance} and \eqref{eqn: Expression of Thetau> infty through f}, these inequalities can be written as
\begin{equation}\label{eqn: Inequalities transversality}
    \begin{aligned}
    \Theta^u_{\circleSplus}{}'(\theta_>)-\Theta^{s}_{\infty}{}'(\theta_>) = -d_+'(\theta_>) \neq 0,\\
    \Theta^u_{\circleSplus}{}'(\theta_>)-\Theta^{u,>}_{\infty}{}'(\theta_>) = -d_-'(\theta_>) + \mathcal{O}_1(\mu\delta^2) \neq 0,\\
    \Theta^s_\infty{}'(\theta_>)-\Theta^{u,>}_{\infty}{}'(\theta_>)  = d_+'(\theta_>) - d_-'(\theta_>) + \mathcal{O}_1(\mu\delta^2) \neq 0.\\
    \end{aligned}
\end{equation}
Using \eqref{eqn: Notation transv}, the first and second inequalities are given by the first and second items of Theorem \ref{thm: transverse intersection of the invariant manifolds} respectively. Finally; by \eqref{eqn: Asymptotic formula for the distances}, \eqref{eqn: translation melnikov} and \eqref{eqn: Relation between M- and M+} we have
\begin{equation}\label{eqn: d_+' and d_'}
    d_+'(\theta_>) = -d_-'(\theta_>) = \mu \left(\mathcal{M}_+'(0) + \mathcal{O}_1(\mu,\delta^2)\right),
\end{equation}
and therefore the third inequality can be written as
\begin{equation*}
    \mu\left(2\mathcal{M}_+{}'(0) + \mathcal{O}_1(\mu,\delta^2)\right) \neq 0.
\end{equation*}
Lemma \ref{lemma: derivative of the melnikov functions not 0 for almost every theta} implies this inequality for $\mu,\delta$ small enough.

Following \eqref{eqn: Inequalities transversality} and \eqref{eqn: d_+' and d_'}, we compute the angles in \eqref{eqn: angles trasnversality} as follows
\begin{equation*}
\begin{aligned}
    \sin A =& \frac{\Theta_\infty^s{}'(\theta_>) - \Theta_\infty^{u,>}{}'(\theta_>)}{\sqrt{\left(1+ \left(\Theta_\infty^s{}'(\theta_>)\right)^2\right)\left(1+ \left(\Theta_\infty^{u,>}{}'(\theta_>)\right)^2\right)}} =\frac{\mu\left(2\mathcal{M}_+'(0) + \mathcal{O}_1(\mu,\delta^2)\right)}{\sqrt{\left(1+ \left(\Theta_\infty^s{}'(\theta_>)\right)^2\right)\left(1+ \left(\Theta_\infty^{u,>}{}'(\theta_>)\right)^2\right)}}\\
    \sin B =& \frac{\Theta_{\circleSplus}^u{}'(\theta_>) - \Theta_\infty^{u,>}{}'(\theta_>)}{\sqrt{\left(1+ \left(\Theta_{\circleSplus}^u{}'(\theta_>)\right)^2\right)\left(1+ \left(\Theta_\infty^{u,>}{}'(\theta_>)\right)^2\right)}} = \frac{\mu\left(\mathcal{M}_+'(0) + \mathcal{O}_1(\mu,\delta^2)\right)}{\sqrt{\left(1+ \left(\Theta_{\circleSplus}^u{}'(\theta_>)\right)^2\right)\left(1+ \left(\Theta_\infty^{u,>}{}'(\theta_>)\right)^2\right)}}.
\end{aligned}
\end{equation*}
Lemma \ref{lemma: derivative of the melnikov functions not 0 for almost every theta} ensures that $\mathcal{M}_+'(0) < 0$, leading to \eqref{eqn: Inequality angles}. See Figure \ref{fig: trans intersection} to see the relative position of the curves.

\appendix
\section{Proof of Lemma \ref{lemma: derivative of the melnikov functions not 0 for almost every theta}}\label{sec: Computer Assisted Proofs}

The proof of Lemma \ref{lemma: derivative of the melnikov functions not 0 for almost every theta} is computer-assisted. We rely on the CAPD libraries developed in \cite{MR4283203}.

We compute the integrals involved in the derivative of the Melnikov function $\mathcal{M}_+(\theta)$ in \eqref{eqn: melnikov function} (the result for $\mathcal{M}_-(\theta)$ follows from (\ref{eqn: Relation between M- and M+})), which can be written as 



\begin{equation}\label{eqn: Expression of the derivative of the Melnikov function}
\begin{aligned}
    \frac{d}{d\theta}\mathcal{M}_+(\theta) = \mathcal{I}_1(\theta) - \mathcal{I}_2(\theta),   
\end{aligned}
\end{equation}
with
\begin{equation}\label{eqn: definition of I1 and I2}
    \begin{aligned}
        \mathcal{I}_1(\theta) &= \constantValue \bigintss_0^{+\infty} \left(\frac{s^{\frac 2 3}\cos \left(\theta - s\right)}{\left(1+\constantValue^2s^{\frac 4 3} - 2\constantValue s^{\frac 2 3}\cos(\theta-s)\right)^{\frac 3 2}} - \frac{3\constantValue s^{\frac 4 3}\sin^2(\theta-s)}{\left(1+\constantValue^2s^{\frac 4 3} - 2\constantValue s^{\frac 2 3}\cos(\theta-s)\right)^{\frac 5 2}}\right)\; ds,\\
        \mathcal{I}_2(\theta) &= \sqrt{\frac 2 \constantValue}\bigintss_0^{+\infty} \frac{\sin(\theta-s)}{s^{\frac 1 3}}\;ds,
    \end{aligned}
\end{equation}
where $\constantValue = \frac{3^{\frac 2 3}}{2^{\frac 1 3}}$.

Note that $\mathcal{I}_2(\theta)$ can be computed explicitly (by means of a Laplace transform)



\begin{equation}\label{eqn: Expression for I2}
    \mathcal{I}_2(\theta) = \sqrt{\frac 2 \constantValue} \Gamma\left(\frac 2 3\right) \left(\frac{\sin\theta}{2} - \frac{\sqrt{3}}{2}\cos\theta\right).
\end{equation}
To compute $\mathcal{I}_1(\theta)$ we proceed as follows. We consider two constants $c,C>0$ and we split it as

\begin{equation}
    \constantValue^{-1}\mathcal{I}_1(\theta) = i(0,c;\theta) + i(c,C;\theta) + i(C,+\infty;\theta)
\end{equation}
with
\begin{equation}\label{eqn: bound notation for I1}
    \begin{aligned}
        i(0,c;\theta) &=\bigintss_0^{c} \left(\frac{s^{\frac 2 3}\cos \left(\theta - s\right)}{\left(1+\constantValue^2s^{\frac 4 3} - 2\constantValue s^{\frac 2 3}\cos(\theta-s)\right)^{\frac 3 2}} - \frac{3\constantValue s^{\frac 4 3}\sin^2(\theta-s)}{\left(1+\constantValue^2s^{\frac 4 3} - 2\constantValue s^{\frac 2 3}\cos(\theta-s)\right)^{\frac 5 2}}\right)\; ds,\\
        i(c,C;\theta) &= \bigintss_c^{C} \left(\frac{s^{\frac 2 3}\cos \left(\theta - s\right)}{\left(1+\constantValue^2s^{\frac 4 3} - 2\constantValue s^{\frac 2 3}\cos(\theta-s)\right)^{\frac 3 2}} - \frac{3\constantValue s^{\frac 4 3}\sin^2(\theta-s)}{\left(1+\constantValue^2s^{\frac 4 3} - 2\constantValue s^{\frac 2 3}\cos(\theta-s)\right)^{\frac 5 2}}\right)\; ds,\\
        i(C,+\infty;\theta) &= \bigintss_C^{+\infty} \left(\frac{s^{\frac 2 3}\cos \left(\theta - s\right)}{\left(1+\constantValue^2s^{\frac 4 3} - 2\constantValue s^{\frac 2 3}\cos(\theta-s)\right)^{\frac 3 2}} - \frac{3\constantValue s^{\frac 4 3}\sin^2(\theta-s)}{\left(1+\constantValue^2s^{\frac 4 3} - 2\constantValue s^{\frac 2 3}\cos(\theta-s)\right)^{\frac 5 2}}\right)\; ds.
    \end{aligned}
\end{equation}
Taking $c < \left(\frac 1 \constantValue\right)^{\frac 3 2} < C$,  we have the following uniform bounds for $i(0,c;\theta)$ and $i(C,+\infty;\theta)$.

\begin{equation}\label{eqn: Estimate of the tails}
    \begin{aligned}
        |i(0,c;\theta)| &\leq\bigintss_0^{c} \left(\frac{s^{\frac 2 3}}{\left(1-\constantValue s^{\frac 2 3}\right)^{3}} +  \frac{3\constantValue s^{\frac 4 3}}{\left(1-\constantValue s^{\frac 2 3}\right)^{5}}\right)\;ds \leq \frac{3}{5} \frac{c^{\frac 5 3}}{\left(1-\constantValue c^{\frac 2 3}\right)^{3}} + \frac{9\constantValue}{7} \frac{c^{\frac 7 3}}{\left(1-\constantValue c^{\frac 2 3}\right)^{5}},\\
        |i(C,+\infty;\theta)| &\leq \bigintss_C^{+\infty}\left(\frac{1}{s^{\frac 4 3}\left(\constantValue - s^{-\frac 2 3}\right)^3} + \frac{3\constantValue}{s^2\left(\constantValue - s^{-\frac 2 3}\right)^5}\right)\;ds\\
        &\leq \frac{3}{C^{\frac 1 3}\left(\constantValue - C^{-\frac 2 3}\right)^3} + \frac{3\constantValue}{C\left(\constantValue - C^{-\frac 2 3}\right)^5}.
    \end{aligned}
\end{equation}
From these estimates, we can obtain the intervals where $\mathcal{I}_1$ and $\mathcal{I}_2$ in \eqref{eqn: definition of I1 and I2} belong using computer computations\footnote{The code to compute the expressions in \eqref{eqn: bound notation for I1} and the data plotted in Figure \ref{fig: derivative_melnikov} can be found at \href{https://github.com/JoseLamasRodriguez/CAPD-Code.git}{https://github.com/JoseLamasRodriguez/CAPD-Code.git}.}.

\begin{figure}[h!]
    \centering
    \includegraphics[width=0.5\linewidth]{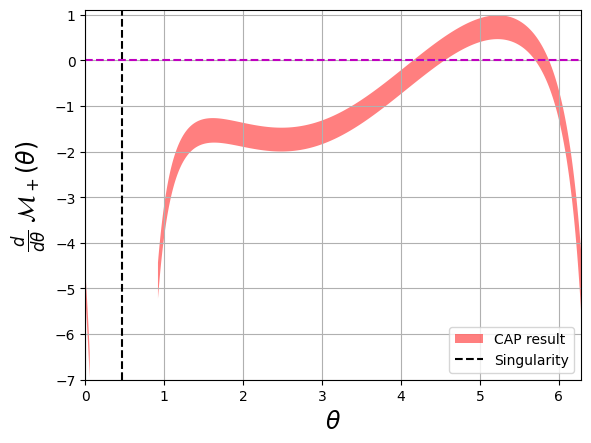}
    \caption{Representation of $\mathcal{M}_+'(\theta)$ from the computer assisted algorithm.}
    \label{fig: derivative_melnikov}
\end{figure} 
To this end note that, to have a well-defined formula for $\mathcal{M}_+(\theta)$, the conditions of Theorem \ref{thm: asymptotic formula for the distance} must be satisfied. In particular, we have to consider values of  $\theta \in I_D^+(w_\Sigma)$ defined in \eqref{eqn: domain Ip} (with $w_\Sigma$ given in \eqref{eqn: definition of wSigma}). In our case, we have taken $\theta \in \mathds{T} - \left(\frac{\sqrt{2}}{3} - 0.45, \frac{\sqrt{2}}{3}+0.45\right)\subset I_D^+(w_\Sigma)$ for $D = 0.3$.

For the estimates in \eqref{eqn: Estimate of the tails}, we have considered $c = 0.001$ and $C = 100$. Smaller values of $c$ or higher values of $C$ may result in better approximations for the derivative $\mathcal{M}_+'(\theta)$ in \eqref{eqn: Expression of the derivative of the Melnikov function} at the expense of increased computational time to evaluate $i(c,C;\theta)$, defined in \eqref{eqn: bound notation for I1}. Finally, the integral $i(c,C;\theta)$ is estimated using a computer-assisted algorithm.

The provided implementation iterates through a range of angles, dividing it into $N=10000$ equidistant intervals of size $\frac{2\pi}{N} + \mathrm{eps}$, where $\mathrm{eps} = 10^{-5}$ is introduced to ensure that these intervals intersect with one another. For each angle within the range, the code ensures that it belongs to $\mathds{T} - \left(\frac{\sqrt{2}}{3} - 0.45, \frac{\sqrt{2}}{3}+0.45\right)$ where, as previously stated, the Melnikov function can be computed. It then estimates the integrals $\mathcal{I}_1$ and $\mathcal{I}_2$ defined in \eqref{eqn: definition of I1 and I2}, where $\mathcal{I}_1$ is estimated using the previous results and $\mathcal{I}_2$ is estimated from \eqref{eqn: Expression for I2}. Finally, the difference between these integrals, and therefore the derivative $\mathcal{M}_+'(\theta)$, is computed and written in the corresponding files.

As a result, we obtain the representation of $\mathcal{M}_+'(\theta)$ depicted in Figure \ref{fig: derivative_melnikov}. Moreover,
\[
\mathcal{M}_+'(0)\in [-5.15341,-4.56572].
\]





\renewcommand{\thesection}{\Alph{section}}

\section{Proof of Lemma \ref{lemma: A lambda lemma for the transition map close to collision} and Lemma \ref{lemma: transition map from C2- to Cu}}\label{appendix: A lambda lemma for the transition map close to collision}

The first part of this section is devoted to prove Lemma \ref{lemma: A lambda lemma for the transition map close to collision}. 
The proof is divided into two parts. The first part consists on finding the curve $\Tilde{\gamma}_1$ in \eqref{eqn: Image curve gamma_1(nu)}. Then, we show that the tangent vector $\Tilde{\gamma}_1'$ is of the form \eqref{eqn: Image tangent vector gamma1}. Both proofs rely on  a fixed point argument.

To study the evolution of the curve $\Tilde{\gamma}_{\mathrm{in}}$ defined in \eqref{eqn: transverse curve gamma_in(nu)}, instead of using equation \eqref{eqn: Eq motion in coordinates (s,Tilde(beta),z)}, we consider  the equations of the orbits
\begin{equation}\label{eqn: Integral equations for s and z}
\begin{aligned}
    \frac{ds}{d\Tilde{\beta}} &= - \frac{s}{\Tilde{\beta}} \left(1+\Tilde{G}_s(s,\Tilde{\beta},z)\right)\\
    \frac{dz}{d\Tilde{\beta}} &= \frac{2}{m_0} s\Tilde{G}_z(s,\Tilde{\beta},z),
\end{aligned}
\end{equation}
where $\Tilde{G}_s$ and $\Tilde{G}_z$ satisfy
\begin{equation}\label{eqn: Estimates Gs and Gz}
\begin{aligned}
    \Tilde{G}_s(s,\Tilde{\beta},z) &= \mathcal{O}_1(s,\Tilde{\beta}), \quad \Tilde{G}_z(s,\Tilde{\beta},z) = 4\lambda(\mu,h)s + \mathcal{O}_3(s,\Tilde{\beta}),\\
    \partial_s \Tilde{G}_s(s,\Tilde{\beta},z) &= \mathcal{O}_0(s), \quad \partial_s \Tilde{G}_z(s,\Tilde{\beta},z) = 4\lambda(\mu,h) + \mathcal{O}_2(s,\Tilde{\beta}),\\
    \partial_z \Tilde{G}_s(s,\Tilde{\beta},z) &= \mathcal{O}_1(s,\Tilde{\beta}), \quad \partial_z \Tilde{G}_z(s,\Tilde{\beta},z) = \mathcal{O}_3(s,\Tilde{\beta}),
\end{aligned}
\end{equation}
and $\lambda(\mu,h) = \frac{\mu^2 + 2h + 2\mu}{4m_0}$. For simplicity we denote $\lambda =\lambda(\mu,h)$. 


We denote by $\Phi^{\Tilde{\beta}}(s_0,\Tilde\beta_0,z_0)$ the general solution of \eqref{eqn: Integral equations for s and z}. Then, 
\[
\Tilde{\gamma}_1(\nu) = (s_1(\nu),\delta,z_1(\nu))=\Phi^{\delta}(\delta,\nu,z_{\mathrm{in}}(\nu)).
\]
To compute the functions $s_1$ and $z_1$, we analyze the evolution of $\Phi^{\Tilde{\beta}}$ from $\Tilde{\beta}=\nu$ to $\Tilde{\beta}=\delta$ by a fixed point argument. 

To this end, for fixed $\nu\in (0,\sigma]$, we consider the set of continuous functions $f\colon[\nu,\delta] \to \mathds{R}$ 
and two different norms: the classical supremum norm, which we denote by $\|\cdot\|$, and 
\begin{equation}\label{eqn: Norm Zs}
    ||f||_{\mathcal{Z}_s} = \underset{\Tilde\beta\in  [\nu,\delta]}{\sup}\left(\frac{\Tilde{\beta}}{\delta \nu }\left|f(\Tilde{\beta};\nu)\right|\right),
\end{equation}
We denote by $\mathcal{C}^0$ and  $\mathcal{Z}_s$  the associated Banach spaces. 
%
Hence, the functional space $\mathcal{Z}_s \times \mathcal{C}^0$ is also a Banach space under the norm
\begin{equation*}
    ||(f_s,f_z)||_{\mathcal{Z}_s \times \mathcal{C}^0} = \max \left(||f_s||_{\mathcal{Z}_s}, ||f_z||\right).
\end{equation*}
Now we define the following operator acting on $\mathcal{Z}_s \times  \mathcal{C}^0$
\begin{equation}\label{eqn: Operator Fsz}
\begin{aligned}
    \mathcal{F}\colon \mathcal{Z}_s \times  \mathcal{C}^0 &\to \mathcal{Z}_s \times  \mathcal{C}^0\\
    (f_s,f_z) &\mapsto \mathcal{F}(f_s,f_z) = \left(\mathcal{F}^1(f_s,f_z),\mathcal{F}^2(f_s,f_z)\right),
\end{aligned}
\end{equation}
with

\begin{equation}\label{eqn: Integral form s and z}
    \begin{aligned}
        \mathcal{F}^1(f_s,f_z) &= \delta \frac{\nu}{\Tilde{\beta}} \exp\left(-\int_\nu^{\Tilde{\beta}} \frac{\Tilde{G}_s(f_s(\alpha;\nu),\alpha, f_z(\alpha;\nu))}{\alpha}\; d\alpha\right)\\
        \mathcal{F}^2(f_s,f_z) &= z_{\mathrm{in}}(\nu) + \frac{2}{m_0}\int_\nu^{\Tilde{\beta}} f_s(\alpha;\nu)\cdot \Tilde{G}_z(f_s(\alpha;\nu),\alpha,f_z(\alpha;\nu))\;d\alpha,
    \end{aligned}
\end{equation}
which corresponds to the solution $\Phi^{\Tilde \beta}(\delta,\nu,z_{\mathrm{in}}(\nu))$  of the equations of the orbits of \eqref{eqn: Integral equations for s and z} such that at the initial time $\Tilde\beta=\nu$  satisfy $s=\delta$ and $z=z_{\mathrm{in}}(\nu)$. 

Once we have defined the operator, a fixed point theorem argument allows us to estimate the components $s_1(\delta;\nu)$ and $z_1(\delta;\nu)$ in \eqref{eqn: Image curve gamma_1(nu)} and prove the first part of the lemma.

Relying on the estimates in \eqref{eqn: Estimates Gs and Gz}, we can estimate $\mathcal{F}(0,0)$ as
by
\begin{equation*}
\begin{aligned}
    ||\mathcal{F}^1(0,0)||_{\mathcal{Z}_s} &\leq \sup_{\Tilde{\beta}\in [\nu,\delta]}\exp\left(-\int_\nu^{\Tilde{\beta}} \frac{\Tilde{G}_s(0,\alpha,0)}{\alpha}\right)\leq e^{C(\delta-\nu)}\leq e^{C\delta} \\  ||\mathcal{F}^2(0,0)|| &=|z_{\mathrm{in}}(\nu)|.
\end{aligned}
\end{equation*}
for some constant $C> 0$. Therefore,
\begin{equation*}
    ||\mathcal{F}(0,0)||_{\mathcal{Z}_s\times  \mathcal{C}^0} \leq \max\left(e^{C\delta} ,\underset{\nu \in (0,\sigma]}{\max}\;|z_{\mathrm{in}}(\nu)|\right) = M.
\end{equation*}
We define now the ball
\begin{equation}\label{eqn: Ball BF}
    B_{2M} = \{(f_s,f_z) \in \mathcal{Z}_s\times  \mathcal{C}^0\colon ||(f_s,f_z)||_{\mathcal{Z}_s \times  \mathcal{C}^0} < 2M\},
\end{equation}
and we prove that the operator $\mathcal{F}$ is Lipschitz in $B_{2M}$.

Consider $(f_s,f_z),(g_s,g_z)\in B_{2M}$. 
The norm $||\mathcal{F}^1 (f_s,f_z) - \mathcal{F}^1(g_s,g_z)||_{\mathcal{Z}_s}$ can be estimated as follows

    \begin{equation}\label{eqn: Estimate F1}
        \begin{aligned}
            &||\mathcal{F}^1 (f_s,f_z) - \mathcal{F}^1(g_s,g_z)||_{\mathcal{Z}_s} \\
            &\leq \left|\exp\left(-\int_\nu^{\Tilde{\beta}} \frac{\Tilde{G}_s(f_s,\alpha,f_z)}{\alpha}\;d\alpha\right) - \exp\left(-\int_\nu^{\Tilde{\beta}} \frac{\Tilde{G}_s(g_s,\alpha,g_z)}{\alpha};d\alpha\right)\right|\\
            &\leq \left|\exp\left(-\int_\nu^{\Tilde{\beta}}\frac{\Tilde{G}_s(g_s,\alpha,g_z)}{\alpha}\;d\alpha\right)\right|\cdot\left|\exp\left(\int_\nu^{\Tilde{\beta}}\frac{\Tilde{G}_s(g_s,\alpha,g_z) - \Tilde{G}_s(f_s,\alpha,f_z)}{\alpha}\;d\alpha\right)-1\right|\\
            &\leq \left|\exp\left(-\int_\nu^{\Tilde{\beta}}\frac{\Tilde{G}_s(g_s,\alpha,g_z)}{\alpha}\;d\alpha\right)\right|\\
            &\quad\cdot\left|\int_{\Tilde{\beta}}^\nu\frac{\Tilde{G}_s(f_s,\alpha,f_z) - \Tilde{G}_s(g_s,\alpha,g_z)}{\alpha}\;d\alpha\right| \cdot \exp\left(\left|\int_{\Tilde{\beta}}^\nu\frac{\Tilde{G}_s(f_s,\alpha,f_z) - \Tilde{G}_s(g_s,\alpha,g_z)}{\alpha}\;d\alpha\right|\right).
        \end{aligned}
    \end{equation}
We now use the following properties, satisfied for any $(f_s,g_s) \in \mathcal{Z}_s$ and any $\alpha \in [\nu,\Tilde{\beta}]$ (we recall that $\Tilde{\beta}\in [\nu,\delta])$
    \begin{equation}\label{eqn: Estimate gs and fs-gs}
    \begin{aligned}
        |g_s(\alpha;\nu)| &\leq \frac{\delta \nu}{\alpha} ||g_s||_{\mathcal{Z}_s}\leq 2M \frac{\delta \nu}{\alpha}\leq 2M\delta,\\
        |f_s(\alpha;\nu) - g_s(\alpha;\nu)| &\leq \frac{\delta \nu}{\alpha}||f_s(\alpha;\nu)-g_s(\alpha;\nu)||_{\mathcal{Z}_s}.
    \end{aligned}
    \end{equation}
    The last inequality of $g_s(\alpha;\nu)$ is written to stress the fact that any function in $\mathcal{Z}_s$ is small (of size $\delta$) in terms of $|\cdot|$. Clearly, the function $f_s$ satisfies the same bound.
    
    Therefore, by \eqref{eqn: Estimates Gs and Gz}, the first factor in the last term in \eqref{eqn: Estimate F1} can be estimated as
    \begin{equation}\label{eqn: Bound 1st term F1}
        \begin{aligned}
            &\left|\exp\left(-\int_\nu^{\Tilde{\beta}}\frac{\Tilde{G}_s(g_s(\alpha;\nu),\alpha,g_z(\alpha;\nu))}{\alpha}\;d\alpha\right)\right| \leq \exp\left(C\int^\nu_{\Tilde{\beta}}\left(\frac{|g_s(\alpha;\nu)|}{\alpha}+1\right)\;d\alpha\right) \\
            &\leq \exp\left(C\int_{\Tilde{\beta}}^{\nu} \left(\frac{2M\delta \nu}{\alpha^2}+1\right)\;d\alpha\right)\leq \exp \left(2CM\delta \nu\left[\frac{1}{\nu} - \frac{1}{\Tilde{\beta}}\right] + C|\nu-\Tilde{\beta}|\right)\leq e^{C(2M+1)\delta},
        \end{aligned}
    \end{equation}
    for some constant $C>0$.

    For the rest of the estimates we use the following result, which is a consequence of \eqref{eqn: Estimates Gs and Gz} and \eqref{eqn: Estimate gs and fs-gs}. The function
    \[
    h^t(\alpha;\nu)=(h^t_s(\alpha;\nu),\alpha,h^t_z(\alpha;\nu))=(tf_s(\alpha;\nu)+(1-t)g_s(\alpha;\nu),\alpha,tf_z(\alpha;\nu)+(1-t)g_z(\alpha;\nu)),
    \]
    satisfies 
    \[
    |h^t_s|\leq 2M \frac{\delta \nu}{\alpha}\leq 2M\delta, \quad |h^t_z|\leq 2M.
    \]
     Then, there exists a constant $C > 0$ such that
    \begin{equation}\label{eqn: Tilde(Gs(f)) - Tilde(Gs(g))}
    \begin{aligned}
        \left|\Tilde{G}_s(f_s,\alpha,f_z) - \Tilde{G}_s(g_s,\alpha,g_z)\right|
        &\leq \int_0^1 \left(\left|\partial_s\Tilde{G}_s (h^t)\right|\cdot |f_s-g_s| + \left|\partial_z\Tilde{G}_s( h^t)\right|\cdot |f_z-g_z|\right) dt\\
        &\leq C\frac{\delta \nu}{\alpha}||f_s-g_s||_{\mathcal{Z}_s} + C\left(\frac{2M\delta\nu}{\alpha} + \alpha\right)||f_z-g_z||.
    \end{aligned}
    \end{equation}
    Then, the second factor can be bounded as follows
    \begin{equation}\label{eqn: Bound 2nd term F1}
    \begin{aligned}
    &\left|\int_{\Tilde{\beta}}^\nu\frac{\Tilde{G}_s(f_s,\alpha,f_z) - \Tilde{G}_s(g_s,\alpha,g_z)}{\alpha}\;d\alpha\right|\leq \int_{\Tilde{\beta}}^\nu \frac{|\Tilde{G}_s(f_s,\alpha,f_z) - \Tilde{G}_s(g_s,\alpha,g_z)|}{\alpha}\; d\alpha\\
    &\leq C\delta \nu||f_s-g_s||_{\mathcal{Z}_s} \int_{\Tilde{\beta}}^\nu \frac{1}{\alpha^2}\; d\alpha + C||f_z-g_z||\int_{\Tilde{\beta}}^\nu \left(1+ \frac{2M\delta \nu}{\alpha^2}\right)\;d\alpha \\
    &\leq C(2M+1)\delta \left(||f_s-g_s||_{\mathcal{Z}_s} + ||f_z-g_z||\right).
    \end{aligned}
    \end{equation} 
    Finally, the last factor can be bounded as

    \begin{equation}\label{eqn: Bound 3rd term F1}
        \exp\left(\left|\int_{\Tilde{\beta}}^{\nu}\frac{\Tilde{G}_s(f_s,\alpha,f_z) - \Tilde{G}_s(g_s,\alpha,g_z)}{\alpha}\;d\alpha\right|\right) \leq e^{8CM(2M+1)\delta},
    \end{equation}
    which follows from \eqref{eqn: Bound 2nd term F1}.
    
    Estimates \eqref{eqn: Bound 1st term F1}, \eqref{eqn: Bound 2nd term F1} and \eqref{eqn: Bound 3rd term F1} yield the following bound for the norm $||\mathcal{F}^1(f_s,f_z) - \mathcal{F}^1(g_s,g_z)||_{\mathcal{Z}_s}$

    \begin{equation*}
    \begin{aligned}
        ||\mathcal{F}^1(f_s,f_z) - \mathcal{F}^1(g_s,g_z)||_{\mathcal{Z}_s} \leq\delta K_1 \left(||f_s-g_s||_{\mathcal{Z}_s} + ||f_z-g_z||\right),
    \end{aligned}
    \end{equation*}
    with $ K_1 =  C(2M+1)e^{(8M+1)(2M+1)C\delta}.$
    
    On the other hand, the norm $||\mathcal{F}^2(f_s,f_z) - \mathcal{F}^2(g_s,g_z)||$ can be bounded as follows. 

First note that proceeding as for the other component,
    \begin{equation}\label{eqn: Tilde(Gz(f)) - Tilde(Gz(g))}
    \begin{aligned}
        &\left|\Tilde{G}_z(f_s,\alpha,f_z) - \Tilde{G}_z(g_s,\alpha,g_z)\right|\\
        &\leq \frac{\delta\nu}{\alpha} C\left(4\lambda + (|f_s|+|g_s|)^2 + \alpha^2 \right)||f_s-g_s||_{\mathcal{Z}_s}+ C\left((|f_s|+|g_s|)^3 +\alpha^3\right)||f_z-g_z|| \\
        &\leq C\frac{\delta\nu}{\alpha}\left(4\lambda + 16M^2\frac{\delta^2\nu^2}{\alpha^2} + \alpha^2\right)||f_s-g_s||_{\mathcal{Z}_s} + C\left(64M^3\frac{\delta^3\nu^3}{\alpha^3} + \alpha^3\right)||f_z-g_z||.
    \end{aligned}
    \end{equation}
    Therefore,
    \begin{equation*}
    \begin{aligned}
        &||\mathcal{F}^2(f_s,f_z) - \mathcal{F}^2(g_s,g_z)|| \leq \frac{2}{m_0}\left|\int_\nu^{\Tilde{\beta}}\left( f_s \cdot \Tilde{G}_z(f_s,\alpha,f_z) - g_s\cdot \Tilde{G}_z(g_s,\alpha,g_z)\right)d\alpha \right|\\
        &\leq \frac{2}{m_0}\int_\nu^{\Tilde{\beta}} |f_s-g_s|\cdot |\Tilde{G}_z(f_s,\alpha,f_z)|\;d\alpha + \frac{2}{m_0}\int_\nu^{\Tilde{\beta}}|g_s|\cdot\left|\Tilde{G}_z(f_s,\alpha,f_z) - \Tilde{G}_z(g_s,\alpha,g_z)\right|\;d\alpha.\\
        \end{aligned}
    \end{equation*}
    The first term can be estimated as
    \begin{equation*}
    \begin{aligned}
        &\frac{2}{m_0}\int_\nu^{\Tilde{\beta}} |f_s-g_s|\cdot |\Tilde{G}_z(f_s,\alpha,f_z)|\;d\alpha \leq  \frac{2}{m_0}\delta \nu ||f_s-g_s||_{\mathcal{Z}_s}\int_\nu^{\Tilde{\beta}}\frac{4\lambda |f_s| + C\left(|f_s|^3 + \alpha^3\right)}{\alpha}d\alpha\\
        &\leq \frac{2}{m_0}\delta^2 \nu \left(8M\lambda + C\left(8M^3+1\right)\delta^2\right)||f_s-g_s||_{\mathcal{Z}_{s}} \leq \delta k_1||f_s-f_g||_{\mathcal{Z}_s}
    \end{aligned}
    \end{equation*}
    and the second as
    \begin{equation*}
    \begin{aligned}
        \MoveEqLeft[1]\frac{2}{m_0}\int_\nu^{\Tilde{\beta}}|g_s|\cdot\left|\Tilde{G}_z(f_s,\alpha,f_z) - \Tilde{G}_z(g_s,\alpha,g_z)\right|\;d\alpha\\
        \leq& \Bigg(\frac{16MC}{m_0}\lambda \delta^2\nu^2\int_\nu^{\Tilde{\beta}} \frac{d\alpha }{\alpha^2} + \frac{64M^3C}{m_0}\delta^4\nu^4 \int_0^{\Tilde{\beta}}\frac{d\alpha }{\alpha^4}+ \frac{4MC}{m_0}\delta^2\nu^2\Bigg) ||f_s-g_s||_{\mathcal{Z}_s}\\
        &+ \left(\frac{256M^4 C}{m_0}\delta^4\nu^4 \int_\nu^{\Tilde{\beta}}\frac{d\alpha}{\alpha^4} + \frac{4MC}{m_0}\delta^4\nu\right)||f_z-g_z|| \leq \delta \left(k_2||f_s-g_s||_{\mathcal{Z}_s} + k_3 ||f_z-g_z||\right),
    \end{aligned}
    \end{equation*}
    where we have used \eqref{eqn: Estimates Gs and Gz}, \eqref{eqn: Estimate gs and fs-gs} and \eqref{eqn: Tilde(Gz(f)) - Tilde(Gz(g))}, and $k_1,k_2,k_3$ are defined as

    \begin{equation*}
        \begin{aligned}
            k_1 &=\frac{2}{m_0}\delta \nu \left(8M\lambda + C\left(8M^3+1\right)\delta^2\right),\\
            k_2 &=\frac{4MC}{m_0}\delta \nu\left(4\lambda + 16M^2\delta^2+\nu\right),\\
            k_3 &=\frac{4MC}{m_0}\delta^3\nu\left(64M^3+ 1\right).
        \end{aligned}
    \end{equation*}
    Denote by $K = \max \left(K_1, \max(k_1,k_2,k_3)\right)$. Then
   \begin{equation*}
    \begin{aligned}
        ||\mathcal{F}(f_s,f_z) - \mathcal{F}(g_s,g_z)||_{\mathcal{Z}_s \times \mathcal{C}^0} &\leq\delta K \left(||f_s-g_s||_{\mathcal{Z}_s} + ||f_z-g_z||\right) \\
        &\leq 2\delta K ||(f_s,f_z) - (g_s,g_z)||_{\mathcal{Z}_s\times \mathcal{C}^0}.
    \end{aligned}
    \end{equation*}
 That is, for $\delta>0$ such that $2K\delta<1/2$, the  operator $\mathcal{F}$ is contractive and satisfies $\mathcal{F}(B_{2M}) \subset B_{2M}$.

Then, the Banach fixed point theorem ensures that there exists a unique fixed point in $B_{2M}$ for the operator $\mathcal{F}$ in \eqref{eqn: Operator Fsz}, which we denote by $(s(\Tilde{\beta},\nu),z(\Tilde{\beta},\nu))$ and satisfies
\begin{equation}\label{def:fitessz}
\left|s(\Tilde{\beta},\nu)\right|\leq 2M\frac{\de\nu}{\Tilde \beta}\qquad \text{and}\qquad \left|z(\Tilde{\beta},\nu)\right|\leq 2M.
\end{equation}
Moreover
\begin{equation*}
    ||(s,z) - \mathcal{F}(0,0)||_{\mathcal{Z}_s\times \mathcal{C}^0} = ||\mathcal{F}(s,z) - \mathcal{F}(0,0)||_{\mathcal{Z}_s\times \mathcal{C}^0} \leq 4M K\delta.
\end{equation*}
Therefore, for all $\Tilde{\beta}\in [\nu,\delta]$,

\begin{equation*}
\begin{aligned}
    |s(\Tilde{\beta};\nu) - \mathcal{F}^1(0,0)(\Tilde{\beta};\nu)| &\leq \frac{\delta \nu}{\Tilde{\beta}} ||(s,z) - \mathcal{F}^1(0,0)||_{\mathcal{Z}_s} 
    \leq  4M K \delta^2 \frac{\nu}{\Tilde{\beta}}
\end{aligned}
\end{equation*}
so, using that \eqref{eqn: Estimates Gs and Gz}, we can write $s(\Tilde{\beta},\nu)$ as 
\begin{equation}\label{eqn: Estimates for s(tilde(beta),nu) and z(tilde(beta),nu)}
    \begin{aligned}
        s(\Tilde{\beta};\nu) &= \mathcal{F}^1(0,0) + \mathcal{O}\left(\delta^2\frac{\nu}{\Tilde{\beta}}\right) = \delta\frac{\nu}{\Tilde{\beta}}\exp\left(-\int_\nu^{\Tilde{\beta}}\frac{\Tilde{G}_s(0,\alpha,0)}{\alpha}\;d\alpha\right) + \mathcal{O}\left(\delta^2\frac{\nu}{\Tilde{\beta}}\right)\\
       & = \frac{\nu}{\Tilde{\beta}}\delta\left( 1+ \mathcal{O}\left(\delta\right)\right).
    \end{aligned}
\end{equation}
To obtain an expression for $z(\Tilde{\beta};\nu)$, we can replace \eqref{eqn: Estimates for s(tilde(beta),nu) and z(tilde(beta),nu)} in the integral equation for $z(\Tilde{\beta};\nu)$ in \eqref{eqn: Integral form s and z} obtaining

\begin{equation*}
\begin{aligned}
    &|z(\Tilde{\beta};\nu)-z_{\mathrm{in}}(\nu)| = \frac{2}{m_0}\left|\int_\nu^{\Tilde{\beta}}s(\alpha;\nu) \cdot \Tilde{G}_z(s(\alpha;\nu),\alpha,z(\alpha;\nu))\; d\alpha\right| \\
    &\leq \frac{2}{m_0}\int_\nu^{\Tilde{\beta}}\left(2M\frac{\delta\nu}{\alpha}\left(4\lambda|s(\alpha;\nu)| + C\left(|s(\alpha;\nu)|^3 + \alpha^3\right)\right)\right)d\alpha \leq \frac{4M\delta\nu}{m_0} \int_\nu^{\Tilde{\beta}}  \frac{8M\lambda \frac{\delta\nu}{\alpha}  + C\left(8M^3\frac{\delta^3\nu^3}{\alpha^3} + \alpha^3\right)}{\alpha} \; d\alpha \\
    &\leq \frac{32M^2}{m_0}\lambda \delta^2\nu + \frac{32M^4C}{m_0}\delta^4\nu + \frac{4MC}{m_0}\delta^5\nu,
\end{aligned}
\end{equation*}
where we used \eqref{eqn: Estimates Gs and Gz} and the fact that $s(\alpha;\nu) \in \mathcal{Z}_s$ (and therefore satisfies \eqref{eqn: Estimate gs and fs-gs}), leading to \eqref{eqn: Image curve gamma_1(nu)} and finishing the first part of the proof.

For the second part of the proof concerning the existence and expression of the tangent vector $\Tilde{\gamma}_1'(\nu) $ in \eqref{eqn: Image tangent vector gamma1} at the point $\tilde \gamma_1(\nu)$, we rake formally the corresponding derivative on the integral equations for $s(\Tilde{\beta};\nu)$ and $z(\Tilde{\beta};\nu)$ in \eqref{eqn: Integral form s and z}. We denote by $s_\partial (\Tilde{\beta};\nu) = \frac{d}{d\nu}s(\Tilde{\beta};\nu)$ and $z_\partial (\Tilde{\beta};\nu) = \frac{d}{d\nu}z(\Tilde{\beta};\nu)$ respectively, and they have the following expressions
\begin{equation*}\label{eqn: derivatives of s and z with respect to nu in terms of tilde(beta)}
    \begin{aligned}
        s_\partial(\Tilde{\beta};\nu) =& \frac{s(\Tilde{\beta};\nu)}{\nu}\left(1+ \Tilde{G}_s(\delta,\nu,z_{\mathrm{in}}(\nu)) - \nu\int_\nu^{\Tilde{\beta}} \frac{\partial_s \Tilde{G}_s(s,\alpha,z) \cdot s_\partial(\alpha;\nu) + \partial_z \Tilde{G}_s(s,\alpha,z) \cdot z_\partial(\alpha;\nu)}{\alpha}\; d\alpha\right)\\
        z_\partial(\Tilde{\beta};\nu) =& z_{\mathrm{in}}'(\nu)- \frac{2\delta}{m_0} \Tilde{G}_z(\delta,\nu,z_{\mathrm{in}}(\nu)) + \frac{2}{m_0}\bigintsss_\nu^{\Tilde{\beta}} \Bigg(s_\partial(\alpha;\nu) \cdot \Tilde{G}_z(s,\alpha,z) + s(\alpha;\nu) \bigg(\partial_s \Tilde{G}_z(s,\alpha,z)\cdot s_\partial(\alpha;\nu) \\
        &+ \partial_z\Tilde{G}_z(s,\alpha,z)\cdot z_\partial(\alpha;\nu)\bigg)\Bigg)\;d\alpha.
    \end{aligned}
\end{equation*}
where $s,z$ refer to $s(\alpha;\nu),z(\alpha;\nu)$ respectively.

To estimate the solutions, we proceed as before. To this end, for fixed $\nu \in (0,\sigma]$, we consider the set of continuous functions $f\colon [\nu,\delta] \to \mathds{R}$ and two different norms: the classical supremum norm and 

\begin{equation}\label{eqn: norm Zpartial}
    ||f||_{\mathcal{Z}_\partial} = \underset{\Tilde{\beta}\in[\nu,\delta]}{\sup}\left(\frac{\Tilde{\beta}}{\delta}\left|f(\Tilde{\beta};\nu)\right|\right),
\end{equation}
We denote by $\mathcal{C}^0$ and $\mathcal{Z}_\partial$ the associated Banach spaces. Hence, the functional space $\mathcal{Z}_\partial \times \mathcal{C}^0$ is also a Banach space under the norm

\begin{equation*}
    ||(f_s,f_z)||_{\mathcal{Z}_\partial \times \mathcal{C}^0} = \max \left(||f_s||_{\mathcal{Z}_\partial}, ||f_z||\right).
\end{equation*}
Now we define the following  affine operator acting on $\mathcal{Z}_\partial \times \mathcal{C}^0$

\begin{equation}\label{eqn: Operator D}
    \begin{aligned}
        \mathcal{F}_\partial\colon \mathcal{Z}_\partial \times \mathcal{C}^0 &\to \mathcal{Z}_\partial \times \mathcal{C}^0\\
        (f_s,f_z) &\mapsto \mathcal{F}_\partial(f_s,f_z) = \left(\mathcal{F}_\partial^1(f_s,f_z),\mathcal{F}_\partial^2(f_s,f_z)\right),
    \end{aligned}
\end{equation}
with

\begin{equation}\label{eqn: Operators for spartial, zpartial}
    \begin{aligned}
        \mathcal{F}_\partial^1(f_s,f_z) &= \frac{s(\Tilde{\beta};\nu)}{\nu}\left(1+\Tilde{G}_s(\delta,\nu,z_{\mathrm{in}}(\nu)) - \nu\int_\nu^{\Tilde{\beta}} \frac{\partial_s \Tilde{G}_s(s,\alpha,z)\cdot f_s + \partial_z \Tilde{G}_s(s,\alpha,z)\cdot f_z}{\alpha} \; d\alpha\right)\\
        \mathcal{F}_\partial^2(f_s,f_z) &= z_{\mathrm{in}}'(\nu)- \frac{2\delta}{m_0} \Tilde{G}_z(\delta,\nu,z_{\mathrm{in}}(\nu)) \\
        &\quad+ \frac{2}{m_0}\int_\nu^{\Tilde{\beta}}\bigg(f_s\cdot \Tilde{G}_z(s,\alpha,z) + s(\alpha;\nu) \bigg(\partial_s \Tilde{G}_z(s,\alpha,z) \cdot f_s+\partial_z \Tilde{G}_z(s,\alpha,z) \cdot f_z\bigg)\bigg)\;d\alpha,
    \end{aligned}
\end{equation}
which is well-defined and continuous.

Once we have defined the operator $\mathcal{F}_\partial$, we prove it has a fixed point. We start estimating $\mathcal{F}_\partial(0,0)$. Each component satisfies 
\begin{equation*}
\begin{aligned}
	&||\mathcal{F}_\partial^1(0,0)||_{\mathcal{Z}_\partial} = \underset{\Tilde{\beta}\in[\nu,\delta]}{\sup} \left(\frac{\Tilde{\beta}}{\delta} \left|\frac{s(\Tilde{\beta};\nu)}{\nu}\left(1+\Tilde{G}_s(\delta,\nu,z_{\mathrm{in}}(\nu))\right)\right|\right)\leq 2M \left(1+\mathcal{O}_1(\delta,\nu)\right) \leq 4M 
	&\\
	&||\mathcal{F}_\partial^2(0,0)|| \leq \left|z_{\mathrm{in}}'(\nu) - \frac{2\delta}{m_0} \Tilde{G}_z(\delta,\nu,z_{\mathrm{in}}(\nu))\right| \leq |z_{\mathrm{in}}'(\nu)| + \frac{2}{m_0}\delta^2\left(4\lambda + 2C\delta^2\right),
\end{aligned}
\end{equation*}
obtained from \eqref{eqn: Estimates Gs and Gz} and the estimate of $s(\Tilde{\beta};\nu)$ in \eqref{def:fitessz}. Therefore
\begin{equation*}
||\mathcal{F}_\partial(0,0)||_{\mathcal{Z}_\partial \times \mathcal{C}^0} \leq \max \left(4M, \underset{\nu \in (0,\sigma]}{\max}|z_{\mathrm{in}}'(\nu)| + \frac{2}{m_0}\delta^2\left(4\lambda + 2C\delta^2\right)\right) = M_\partial.
\end{equation*}
We define then the following ball
\begin{equation}\label{eqn: Ball BFpartial}
B_{2M^\partial} = \{(f_s,f_z)\in \mathcal{Z}_\partial \times \mathcal{C}^0\colon ||(f_s,f_z)||_{\mathcal{Z}_\partial \times \mathcal{C}^0} < 2M_\partial\},
\end{equation}
and we prove that the operator $\mathcal{F}_\partial$ is Lipschitz in $B_{2M^\partial}$. 

Take $(f_s,f_z),(g_s,g_z)\in B_{2M_\partial}$. 
Then, the norm $||\mathcal{F}_\partial^1(f_s,f_z) - \mathcal{F}_\partial^1(g_s,g_z)||_{\mathcal{Z}_\partial}$ can be bounded as follows

\begin{equation*}
\begin{aligned}
\MoveEqLeft[1]||\mathcal{F}_\partial^1(f_s,f_z) - \mathcal{F}_\partial^1(g_s,g_z)||_{\mathcal{Z}_\partial}\\ \leq & \underset{\Tilde{\beta}\in[\nu,\delta]}{\sup} \Bigg|\frac{\Tilde{\beta}}{\delta}\cdot s(\Tilde{\beta};\nu) \int_\nu^{\Tilde{\beta}} \frac{\partial_s \Tilde{G}_s(s,\alpha,z)\cdot (f_s-g_s) - \partial_z \Tilde{G}_s(f_s,\alpha,f_z)\cdot (f_s-g_s)}{\alpha}d\alpha\Bigg|\\
\leq& 2M  \nu\int_\nu^{\Tilde{\beta}}\frac{|\partial_s \Tilde{G}_s(s,\alpha,z)|\cdot |f_s-g_s| + |\partial_z \Tilde{G}_s(s,\alpha,z)|\cdot |f_z-g_z|}{\alpha}\;d\alpha\\
\leq& 2CM  \nu \int_\nu^{\Tilde{\beta}} \frac{|f_s-g_s|}{\alpha}\;d\alpha\\
&+ 2CM \nu\int_\nu^{\Tilde{\beta}}\frac{(|s(\alpha;\nu)| + \alpha)|f_z-g_z|}{\alpha}\; d\alpha \leq 2CM\delta\left(||f_s-g_s||_{\mathcal{Z}_s} + \nu\left(2M+1\right)||f_z-g_z||\right),
\end{aligned}
\end{equation*}
where we have used \eqref{eqn: Estimates Gs and Gz}, that $s(\alpha;\nu)\in \mathcal{Z}_s$ and therefore satisfies \eqref{def:fitessz} and also that, since  $f_s,g_s\in \mathcal{Z}_\partial$, for all $\alpha\in[\nu,\Tilde{\beta}]$
\begin{equation}\label{eqn: Estimates fs and fs-gs partial}
\begin{aligned}
|f_s(\alpha;\nu)-g_s(\alpha;\nu)| \leq \frac{\delta}{\alpha}||f_s-g_s||_{\mathcal{Z}_\partial}.
\end{aligned}
\end{equation}
Therefore, if we define $K_1 = 2CM \max(1,\nu(2M+1))$, we obtain the following estimate
\begin{equation*}
||\mathcal{F}_\partial^1(f_s,f_z) - \mathcal{F}_\partial^1(g_s,g_z)||_{\mathcal{Z}_\partial}\leq \delta K_1\left(||f_s-g_s||_{\mathcal{Z}_\partial} + ||f_z-g_z||\right).
\end{equation*}
On the other hand, the norm $||\mathcal{F}_\partial^2(f_s,f_z) - \mathcal{F}_\partial^2(g_s,g_z)||$ is bounded as follows
\begin{equation*}
\begin{aligned}
\MoveEqLeft[1]||\mathcal{F}_\partial^2(f_s,f_z) - \mathcal{F}_\partial^2(g_s,g_z)||\\
\leq& \frac{2}{m_0}\int_\nu^{\Tilde{\beta}}\bigg(|\Tilde{G}_z(s,\alpha,z)|\cdot|f_s-g_s| + |s(\alpha;\nu)|\bigg(|\partial_s \Tilde{G}_z(s,\alpha,z)| \cdot |f_s-g_s| +|\partial_z \Tilde{G}_z(s,\alpha,z)|\cdot|f_z-g_z|\bigg)\bigg)\;d\alpha \\
\leq& \frac{2}{m_0} \delta||f_s-g_s||_{\mathcal{Z}_\partial}\int_\nu^{\Tilde{\beta}}\frac{|\Tilde{G}_z(s,\alpha,z)|}{\alpha}\;d\alpha \\
&+ \frac{2}{m_0}\delta||f_s-g_s||_{\mathcal{Z}_\partial}\int_\nu^{\Tilde{\beta}}\frac{|s(\alpha;\nu)|}{\alpha}\cdot|\partial_s \Tilde{G}_z(s,\alpha,z)|\;d\alpha + ||f_z-g_z||\int_\nu^{\Tilde{\beta}} |s(\alpha;\nu)|\cdot |\partial_z \Tilde{G}_z(s,\alpha,z)|\;d\alpha.
\end{aligned}
\end{equation*}
Then, for the first term, we use that
\begin{equation*}
    \begin{aligned}
        \int_\nu^{\Tilde{\beta}}\frac{|\Tilde{G}_z(s,\alpha,z)|}{\alpha}\; d\alpha \leq \int_\nu^{\Tilde{\beta}}\frac{4\lambda|s(\alpha;\nu)| + C(|s(\alpha;\nu)|^3+ \alpha^3)}{\alpha} d\alpha \leq k_1,
    \end{aligned}
\end{equation*}
with $k_1 = 2M\delta^2 \left(8M\lambda  + C\delta^2\left(8M^3 + \nu\delta\right)\right)$. For the second term,
\begin{equation*}
    \begin{aligned}
        \int_\nu^{\Tilde{\beta}}\frac{|s(\alpha;\nu)|}{\alpha}\cdot |\partial_s \Tilde{G}_z(s,\alpha,z)| d\alpha \leq 2M\delta\nu \int_\nu^{\Tilde{\beta}} \frac{4\lambda + C\left(|s(\alpha;\nu)|^2+ \alpha^2\right)}{\alpha^2}d\alpha\leq k_2,
    \end{aligned}
\end{equation*}
with  $k_2 = 2M\delta\left(4\lambda + C\delta\left(4M^2\delta + \nu\right)\right)$ and, for the third,
\begin{equation*}
    \begin{aligned}
        \int_\nu^{\Tilde{\beta}}|s(\alpha;\nu)|\cdot |\partial_z\Tilde{G}_z(s,\alpha,z)|d\alpha \leq k_3\delta,
    \end{aligned}
\end{equation*}
with $k_3 = 2MC\delta^3\nu\left(8M^3 +1 \right)$. 

To obtain these estimates, we have used the results in \eqref{eqn: Estimates Gs and Gz} and the fact that $s(\alpha;\nu)$ satisfies \eqref{def:fitessz} and \eqref{eqn: Estimates fs and fs-gs partial}. Hence, if we denote by $K_2 = \max\left(\frac{2}{\constantEquilibrium}k_1,\frac{2}{\constantEquilibrium}k_2,k_3\right)$, we have
\begin{equation*}
||\mathcal{F}_\partial^2(f_s,f_z) - \mathcal{F}_\partial^2(g_s,g_z)||\leq \delta K_2  \left(||f_s-g_s||_{\mathcal{Z}_\partial} + ||f_z-g_z||\right),
\end{equation*}
leading to the following estimate of $||\mathcal{F}_\partial(f_s,f_z) - \mathcal{F}_\partial(g_s,g_z)||_{\mathcal{Z}_\partial \times \mathcal{C}^0}$

\begin{equation*}
||\mathcal{F}_\partial(f_s,f_z) - \mathcal{F}_\partial(g_s,g_z)||_{\mathcal{Z}_\partial \times \mathcal{C}^0} \leq 2\delta K ||(f_s,f_z) - (g_s,g_z)||_{\mathcal{Z}_\partial \times \mathcal{C}^0},
\end{equation*}
with $K=\max(K_1,K_2)$, proving that, taking $\delta$ small enough, it is contractive and satisfies $\mathcal{F}_\partial(B_{2M^\partial})\subset B_{2M^\partial}$.

Hence, there exists a unique fixed point in $B_{2M^\partial}$ of the operator $\mathcal{F}_\partial$, which we denote by $(s_\partial(\Tilde{\beta};\nu),z_\partial(\Tilde{\beta};\nu))$. We can estimate its value as follows
\begin{equation*}
\begin{aligned}
||(s_\partial,z_\partial) - \mathcal{F}_\partial(0,0)||_{\mathcal{Z}_\partial \times \mathcal{C}^0} = ||\mathcal{F}_\partial(s_\partial,z_\partial) - \mathcal{F}_\partial(0,0)||_{\mathcal{Z}_\partial \times \mathcal{C}^0} \leq  2\delta K ||(s_\partial,z_\partial)||_{\mathcal{Z}_\partial \times \mathcal{C}^0} \leq 4\delta M_\partial K.
\end{aligned}
\end{equation*}
Therefore, for all $\Tilde{\beta}\in [\nu,\delta]$
\begin{equation*}
\begin{aligned}
|s_\partial(\Tilde{\beta};\nu) - \mathcal{F}_\partial^1(0,0)| &\leq  4M_\partial K\frac{\delta^2}{\Tilde{\beta}}\\
|z_\partial(\Tilde{\beta};\nu) - \mathcal{F}_\partial^2(0,0)| &\leq  4M_\partial K\delta.
\end{aligned}
\end{equation*}
The expression of the tangent vector $\Tilde{\gamma}_1^\partial$ in \eqref{eqn: Image tangent vector gamma1} is obtained by taking $\Tilde{\beta} = \delta$,
\begin{equation}\label{eqn: Expressions for spartial, zpartial}
\begin{aligned}
s_\partial(\delta;\nu) =& \mathcal{F}_\partial^1(0,0) + \mathcal{O}_1\left(\delta\right) = \frac{s(\delta;\nu)}{\nu}\left(1+\Tilde{G}_s\left(\delta,\nu,z_{\mathrm{in}}(\nu)\right)\right)+\mathcal{O}_1\left(\delta\right)\\
z_\partial(\delta;\nu) =& \mathcal{F}_\partial^2(0,0) + \mathcal{O}_2(\delta) = z_{\mathrm{in}}'(\nu) + \frac{2\delta}{m_0}\Tilde{G}_z(\delta,\nu,z_{\mathrm{in}}(\nu)) + \mathcal{O}_1\left(\delta\right).
\end{aligned}
\end{equation}
Using \eqref{eqn: Estimates for s(tilde(beta),nu) and z(tilde(beta),nu)}, this leads to
\begin{equation}\label{eqn: Derivatives of s and z with respect to nu from Cs to C1}
\begin{aligned}
s_\partial(\delta;\nu) =& 1 + \mathcal{O}_1(\delta)\\
z_\partial(\delta;\nu) =& z_{\mathrm{in}}'(\nu) + \mathcal{O}_1(\delta),
\end{aligned}
\end{equation}
which implies that the curve $\Tilde{\gamma}_1$ in \eqref{eqn: Image curve gamma_1(nu)} is transverse with $\unstableManifoldCollisionSminusmu\cap \Tilde{\Sigma}_1^+$, which is given by  $\{s=0\}$, at the point $\lim_{\nu\to 0}\Tilde{\gamma}_1 (\nu)=(0,\delta, z_{\mathrm{in}}(0))$.

The proof of Lemma \ref{lemma: transition map from C2- to Cu} is completely analogous, using instead the equations of the orbits associated to system \eqref{eqn: Eq motion in coordinates (s,iota,w)} obtained in Proposition \ref{proposition: Straightening local diffeomorphisms and motions}, given by

\begin{equation}\label{eqn: Eq motion of Tilde(iota) and w in terms of s}
    \begin{aligned}
        \frac{d\Tilde{\iota}}{ds} &= -\frac{\Tilde{\iota}}{s}\left(1+J_{\Tilde{\iota}}(s,\Tilde{\iota},w)\right)\\
        \frac{dw}{ds} &= \frac{2}{m_0}\Tilde{\iota}J_w(s,\Tilde{\iota},w)
    \end{aligned}
\end{equation}
where

\begin{equation*}
J_{\Tilde{\iota}}(s,\Tilde{\iota},w) = \mathcal{O}_1(s,\Tilde{\iota}), \quad J_w(s,\Tilde{\iota},w) = -4\lambda(\mu,h)s + s\mathcal{O}_1(s,\Tilde{\iota}).
\end{equation*}

\section{Proof of Lemma \ref{lemma: Expression of gamma_2(nu) in coordinates (s,Tilde(beta),z)}}\label{appendix: An implicit function theorem for the transition map close to collision}

To prove Lemma \ref{lemma: Expression of gamma_2(nu) in coordinates (s,Tilde(beta),z)}, note that, for $s=0$, the equations of motion \eqref{eqn: Eq motion in coordinates (s,Tilde(beta),z) from C1 to C2} are reduced to

\begin{equation}\label{eqn: Eq motion for the heteroclinic connection in coordinates (s,Tilde(beta),z)}
    \begin{aligned}
        \Tilde{\beta}' &= \frac{m_0}{2}\sin\Tilde{\beta}\\
        z' &= 0,
    \end{aligned}
\end{equation}
which describe the heteroclinic connections between $\unstableManifoldCollisionSminusmu$ and $\stableManifoldCollisionSplusmu$ given in Lemma \ref{lemma: Linear part of motion close to S+-}. They can be parameterized as 
\begin{equation}\label{eqn: Heteroclinic connections Lemma 55}
    \begin{aligned}
        \Tilde{\beta}_h(\tau) &= 2\tan^{-1}\left(e^{\frac{\constantEquilibrium}{2}\tau} \tan \left(\frac{\delta}{2}\right)\right)\\
        z_h(\tau) &= z,
    \end{aligned}
\end{equation}
by fixing the initial condition  at $\tau = 0$ as $(0,\delta,z)\in\Tilde{\Sigma}_1^+$.

Denote by  $\tau_h>0$ the time such that $\Tilde{\beta}_h(\tau_h) = \pi-\delta$. 
Then, if we consider a point $(s,\delta,z)\in\Tilde{\Sigma}_1^+$, one can see that there exist a time $\bar \tau(s,z)$ so that the orbit hits $\Tilde{\Sigma}_2^-$ by an Implicit Function Theorem argument. 
Indeed, define,
\begin{equation*}
    \mathcal{F} (\tau,s;z) = \Phi_{\Tilde{\beta}}\left(\tau; s,\delta,z\right) - (\pi-\delta),
\end{equation*}
where $\Phi_{\Tilde{\beta}}$ refers to the $\Tilde{\beta}$-component of the flow $\Phi$, then
\begin{itemize}
    \item $\mathcal{F}(\tau_h,0;z) = 0$.
    \item $\frac{d}{d\tau} \mathcal{F}\left(\tau_h,0;z\right) = \frac{m_0}{2}\sin(\Tilde{\beta}_h(\tau_h)) = \frac{m_0}{2}\sin \delta \neq 0$.
\end{itemize}
Therefore, we can apply the Implicit Function Theorem to ensure that for $0<s<\delta$ and $\delta>0$ small enough, there exists a unique $\bar \tau = \bar \tau(s;z)$ with $\bar \tau(0;z) = \tau_h$ satisfying

\begin{equation}\label{eqn: Implicit function =0 App B}
    \mathcal{F}(\bar \tau(s;z),s;z) = 0.
\end{equation}
As a result, for any initial condition $(s,\delta,z) \in \Tilde{\Sigma}_1^+$, there exists a time $\bar \tau(s;z)$ satisfying $\Phi_{\Tilde{\beta}}\left(\bar \tau(s;z);s,\delta,z\right) = \pi - \delta$ for the flow $\Phi$. Therefore, we define the map $\mathcal{T}_{1,2}^+$ in \eqref{eqn: transition map T12} as
\begin{equation}\label{eqn: Relation T12-Phi}
    \mathcal{T}_{1,2}^+(s,\delta,z) = \Phi(\bar \tau(s;z);s,\delta,z) = \left(\Phi_s(\bar \tau(s;z);s,\delta,z),\de,  \Phi_z(\bar \tau(s;z);s,\delta,z) \right) \in \Tilde{\Sigma}_2^-.
\end{equation}
The remaining task is then to compute the expansion with respect to $s$ of  the map. Equivalently, $s$-expansions of $\bar\tau$ and the flow components $\Phi_s$ and $\Phi_z$.


Perform a Taylor expansion at $s=0$ of the flow $\Phi$  with initial condition at $\Tilde{\Sigma}_1^+$, we obtain
\begin{equation}\label{eqn: Taylor Expansion of the flow from C1 to C2}
    \begin{aligned}
        \Phi_s(\tau;s,\delta,z) &= s \cdot\xi_{s,0}^1(\tau;z) + \mathcal{O}_2(s)\\
        \Phi_{\Tilde{\beta}}(\tau;s,\delta,z) &= \Tilde{\beta}_h(\tau) + s \cdot\xi_{s,0}^2(\tau;z) + \mathcal{O}_2(s)\\
        \Phi_z(\tau;s,\delta,z) &= z + s\cdot  \xi_{s,0}^3(\tau;z) + \mathcal{O}_2(s),
    \end{aligned}
\end{equation}
with
\begin{equation*}
    \xi_{s,0}^1(\tau;z) =\partial_s \Phi_s(\tau;0,\delta,z), \quad  \xi_{s,0}^2(\tau;z) = \partial_s \Phi_{\Tilde{\beta}}(\tau;0,\delta,z),\quad 
\xi_{s,0}^3(\tau,z) = \partial_s\Phi_z(\tau;0,\delta,z).
\end{equation*}
To obtain $\xi_{s,0}^1(\tau;z)$, we expand each side of the equation
\begin{equation}\label{eqn: Variational s=0}
    \frac{d}{d\tau} \Phi_s(\tau;s,\delta,z) =  F_s(\Phi_s(\tau; s,\delta,z), \Phi_{\Tilde{\beta}}(\tau;s,\delta,z),\Phi_{z}(\tau;s,\delta,z)), 
\end{equation}
where $F_s$ corresponds to the $s$-component of the vector field associated to motion \eqref{eqn: Eq motion in coordinates (s,Tilde(beta),z) from C1 to C2}, around the unperturbed solution $(0,\tilde \beta_h(\tau),z)$. 
We obtain
\[
\begin{split}
   \frac{d}{d\tau} \Phi_s(\tau;s,\delta,z)&= s \cdot  \xi_{s,0}^1{}'(\tau,z) + \mathcal{O}_2(s)\\
   F_s(\Phi_s(\tau; s,\delta,z), \Phi_{\Tilde{\beta}}(\tau;s,\delta,z),\Phi_{z}(\tau;s,\delta,z)) &= -\frac{m_0}{2} s\cdot\xi_{s,0}^1(\tau;z) \cos(\Tilde{\beta}_h(\tau)) + \mathcal{O}_2(s).
\end{split}
\]
This gives the equation for $\xi_{s,0}^1{}$,
\begin{equation*}
    \begin{cases}
     \xi_{s,0}^1{}'(\tau;z) = -\frac{m_0}{2} \xi_{s,0}^1(\tau;z) \cos(\Tilde{\beta}_h(\tau))\\
      \xi_{s,0}^1(0;z) = 1.
    \end{cases}
\end{equation*}
whose solution is 
\begin{equation}\label{eqn: Expression of the first component of the s variational at r=0}
     \xi_{s,0}^1(\tau;z) = \xi_{s,0}^1(\tau) = \cosh\left(\frac{m_0}{2}\tau\right) - \cos \delta \sinh\left(\frac{m_0}{2}\tau\right).
\end{equation}
Following a similar procedure, we obtain 

\begin{equation} \label{eqn: Expression of the second/third components of the s variational at r=0}
        \xi_{s,0}^2(\tau;z) = 0, \quad \xi_{s,0}^3(\tau;z) = 0,
\end{equation}
leading to the following result

\begin{equation}\label{eqn: Image curve}
    \begin{aligned}
        \Phi_s(\bar \tau(s;z); s,\delta,z) &= s\cdot  \xi_{s,0}^1(\bar \tau(s;z);z) + \mathcal{O}_2(s)\\
        \Phi_{\Tilde{\beta}}(\tau;s,\delta,z) &= \Tilde{\beta}_h(\tau)  + \mathcal{O}_2(s)\\
        \Phi_z(\bar \tau(s;z); s,\delta,z) &= z + \mathcal{O}_2(s).
    \end{aligned}
\end{equation}
Then, by the second equation the function $\bar\tau$ obtained by the Implicit Function Theorem is of the form 
\begin{equation*}
    \bar \tau(s;z) = \tau_h + \mathcal{O}_2(s).
\end{equation*}
Hence, using the fact that $\tilde \beta_h(\tau_h)=\pi-\delta$, the expression of $\tilde \beta_h$ in \eqref{eqn: Heteroclinic connections Lemma 55} and the definition of $\xi_{s,0}^1(\tau;z)$ in \eqref{eqn: Expression of the first component of the s variational at r=0}, one can see that $\xi_{s,0}^1(\tau_h;z)=1$ and  therefore
\begin{equation}
    \xi_{s,0}^1(\bar \tau(s;z);z) = 1+ \mathcal{O}_2(s).
\end{equation}
This leads to 
\begin{equation}\label{eqn: Estimation of the image}
    \Phi(\bar \tau(s;z);s,\delta,z) = (s+\mathcal{O}_2(s), \pi-\delta, z + \mathcal{O}_2(s)) \in \Tilde{\Sigma}_2^-.
\end{equation}
Taking the curve $\Tilde{\gamma}_1(\nu) = (s_1(\delta;\nu),\delta,z_1(\delta;\nu))$ in \eqref{eqn: Image curve gamma_1(nu)} as the initial condition 
yields the curve 
\[
\Tilde{\gamma}_2(\nu)= \mathcal{T}_{1,2}^+(s_1(\delta;\nu),\delta,z_1(\delta;\nu))
 \]
satisfying \eqref{eqn: Curve gamma_2(nu)}, where $\mathcal{T}_{1,2}^+$ is defined in \eqref{eqn: Relation T12-Phi}. 

The regularity of the map $\mathcal{T}_{1,2}^+$ with respect to $s$ (consequence of \eqref{eqn: Taylor Expansion of the flow from C1 to C2}) and the curve $\tilde \gamma_1(\nu)$ (consequence of Lemma \ref{lemma: A lambda lemma for the transition map close to collision}) allows us to compute the tangent vector $\tilde \gamma_2'(\nu)$ directly from \eqref{eqn: Curve gamma_2(nu)}, resulting in \eqref{eqn: Tangent vector of gamma_2} and completing the proof.

\section{Proof of Lemma \ref{lemma: tangent vector at Sigma delta,h as a graph}}\label{appendix: From straightening coordinates to synodical polar coordinates}

This section is devoted to express the transition map provided by Theorem \ref{thm: Transition map close to collision} in polar coordinates centered at $P_1$. To this end, we follow the notation of Sections \ref{sec: Dynamics close to collision} and \ref{sec: Proof of Theorem parabolic orbits}, and we recall the following  facts.
\begin{itemize} 
    \item The parameterizations $\chi_\pm$ of the invariant manifolds $W_\mu^s(S_{y_0}^-)$ and $W_\mu^u(S_{x_0}^+)$ provided in   \eqref{eqn: Approximation of chi(s,z) in terms of s}, \eqref{eqn: Approximation of chi(s,w)} (see also \eqref{eqn: Values a0, a1}) satisfies the following properties
    \begin{equation}\label{eqn: Approximations of partial_z chi_- and partial_w chi_+}
    \begin{aligned}
        \partial_z \chi_-(\delta,z) &= \partial_w \chi_+(\delta,w) = 1 + \mathcal{O}_4(\delta)\\
        \partial_y \chi_-^{-1}(\delta,y) &= \partial_x\chi_+^{-1}(\delta,x) = 1 + \mathcal{O}_4(\delta),
    \end{aligned}
    \end{equation}
    where $\chi_\pm^{-1}$ denote the inverse of the functions $z\mapsto  \chi_\pm(\delta,z)$.
    
    \item By the definition and properties of $\psi_-$ and $\psi_+$ in \eqref{eqn: Approximation of psi-(s,y) in terms of s} and \eqref{eqn: Approximation of psi+(s,x)}, the derivatives $\partial_y\psi_-(\delta,y)$ and $\partial_x\psi_+(\delta,x)$  satisfy
    \begin{equation}\label{eqn: Approximations of partial_y psi_- and partial_x psi_+}
        \partial_y \psi_-(\delta,y) = \partial_x \psi_+(\delta,x) = \mathcal{O}_2(\delta).
    \end{equation}
\end{itemize}

To prove Lemma \ref{lemma: tangent vector at Sigma delta,h as a graph}, we first apply the changes of coordinates $\psi$ and $\Tilde\psi$ introduced in \eqref{eqn:McGehee map Collision} and \eqref{eqn:polar change of coordinates for u,v into rho alpha} respectively  to translate the curves $\gamma^{u,<}_\infty$ and $\gamma^{s,<}_{\circleSminus}$ defined in \eqref{eqn: Curves as graphs of theta in coordinates (theta,Theta)} into coordinates $(\theta,\alpha)$.

This leads to curves parameterized as graphs as
\begin{equation}\label{eqn: Curves gamma_u< and gamma_s< in coordinates (alpha,theta)}
    \begin{aligned}
        \overline{\gamma}^{u,<}_\infty = \Big\{ \left(\theta, \alpha^{u,<}_\infty(\theta)\right) \colon \theta \in B_{\varepsilon}\left(\theta_>\right)\Big\},\quad \overline{\gamma}^{s,<}_{\circleSminus} = \Big\{\left(\theta, \alpha^{s,<}_{\circleSminus}(\theta)\right) \colon \theta \in B_{\varepsilon}\left(\theta_>\right)\Big\}
    \end{aligned}
\end{equation}
such that

\begin{equation*}
    \begin{aligned}
        \alpha^{u,<}_\infty(\theta) = - \arccos\left(\frac{\Theta_\infty^u(\theta) + \mu\delta^2\cos\theta - \delta^4}{\delta \sqrt{2(1-\mu) + \rho(\theta)}}\right),\quad \alpha^{s,<}_{\circleSminus}(\theta) = - \arccos\left(\frac{\Theta_{\circleSminus}^s(\theta) + \mu\delta^2\cos\theta - \delta^4}{\delta \sqrt{2(1-\mu) + \rho(\theta)}}\right).
    \end{aligned}
\end{equation*}
To obtain approximations of $\alpha_\infty^{u,<}(\theta)$ and $\alpha_{\circleSminus}^{s,<}(\theta)$ (and their derivatives), we recall the results from Proposition \ref{proposition: Parameterization of the invariant manifolds of infinity in rotating polar coordinates centered at P1} and Proposition \ref{proposition: Perturbed invariant manifolds of collision in synodical polar coordinates}, which give us estimates for both angular momenta $\Theta_\infty^u(\theta)$ and $\Theta_{\circleSminus}^s(\theta)$ respectively. In fact, $\rho(\theta) = \rho(\delta,\theta)$ (defined in \eqref{eqn:Relation rho with (r,theta,H,mu)}) admits the following approximation

\begin{equation}\label{eqn: Approximation rho Appendix D}
    \rho(\theta) = \mathcal{O}_1(\delta^6,\mu\delta^2),\quad  \partial_\theta \rho(\theta) = \mathcal{O}_1(\mu\delta^4),
\end{equation}
since $h= -\hat{\Theta}_0(\mu) = \mathcal{O}_1(\mu)$. Note that

\begin{itemize}
    \item We take the $-$ sign for both $\alpha^{u,<}_\infty(\theta)$ and $\alpha^{s,<}_{\circleSminus}(\theta)$ since we are close to the circle $\circleSminus$ (which is characterized by $\alpha=-\frac{\pi}{2}$ as shown in \eqref{eqn: Circles of equilibrium points for mu neq 0}).
    \item For $\theta = \theta_< \in B_{\varepsilon}(\theta_>)$ it is satisfied that $\alpha^{u,<}_\infty(\theta_<) = \alpha^{s,<}_{\circleSminus}(\theta_<)$.
\end{itemize} 
For further computations it will be necessary to approximate both derivatives $\alpha^{u,<}_{\infty}{}'(\theta_<)$ and $\alpha^{s,<}_{\circleSminus}{}'(\theta_<)$, which are given by

\begin{equation}\label{eqn: Approximations of the derivatives of alpha_u< and alpha_s< at theta(mu)}
    \alpha^{u,<}_\infty{}'(\theta_<) = \frac{\Theta_{\infty}^u{}'(\theta_<)}{\sqrt{2(1-\mu)}\delta} + \mathcal{O}_1(\mu\delta),\quad \alpha^{s,<}_{\circleSminus}{}'(\theta_<) = \frac{\Theta_{\circleSminus}^{s}{}'(\theta_<)}{\sqrt{2(1-\mu)}\delta} + \mathcal{O}_1(\mu\delta)
\end{equation}
where we have considered $0<\mu\ll \delta$ small enough, and we have used the approximations in \eqref{eqn: Approximation rho Appendix D} and the fact that both $\Theta_{\infty}^u{}'(\theta_<)$ and $\Theta_{\circleSminus}^{s}{}'(\theta_<)$ are of order $\mathcal{O}_1(\mu)$ (see Propositions \ref{proposition: Parameterization of the invariant manifolds of infinity in rotating polar coordinates centered at P1} and \ref{proposition: Perturbed invariant manifolds of collision in synodical polar coordinates}).

Once in coordinates $(\theta,\alpha)$, we apply the diffeomorphism $\Gamma_-$ from Proposition \ref{proposition: Straightening local diffeomorphisms and motions} to express the curves $\overline{\gamma}^{u,<}_\infty$ and $\overline{\gamma}^{s,<}_{\circleSminus}$ in \eqref{eqn: Curves gamma_u< and gamma_s< in coordinates (alpha,theta)} in  coordinates $(\Tilde{\beta},z)$, yielding the following parameterizations in terms of $\theta \in B_\varepsilon(\theta_<)$

\begin{equation}\label{eqn: Parameterizations of gammau< and gammas< in terms of theta in coordinates (Tilde(beta),z)}
    \begin{aligned}
        \Tilde{\gamma}^{u,<}_\infty(\theta) = \left(\Tilde{\beta}^{u,<}_{\infty} (\theta), z^{u,<}_\infty(\theta)\right),\quad \Tilde{\gamma}^{s,<}_{\circleSminus}(\theta) = \left(\Tilde{\beta}^{s,<}_{\circleSminus} (\theta), z^{s,<}_{\circleSminus}(\theta)\right)
    \end{aligned}
\end{equation}
such that

\begin{equation}\label{eqn: Expression of Tilde(beta) and z}
    \begin{aligned}
       \Tilde{\beta}^{u,<}_{\infty} (\theta) &= \alpha^{u,<}_\infty(\theta) + \frac{\pi}{2} - \psi_-\left(\delta,  y^{u,<}_\infty(\theta)\right),\quad z^{u,<}_\infty(\theta) = \chi_-^{-1}\left(\delta, y^{u,<}_\infty(\theta)\right),\\
       \Tilde{\beta}^{s,<}_{\circleSminus} (\theta) &= \alpha^{s,<}_{\circleSminus}(\theta) + \frac{\pi}{2} - \psi_-\left(\delta,  y^{s,<}_{\circleSminus}(\theta)\right),\quad z^{s,<}_{\circleSminus}(\theta)= \chi_-^{-1}\left(\delta, y^{s,<}_{\circleSminus}(\theta)\right),
    \end{aligned}
\end{equation}
with
\begin{equation*}
    \begin{aligned}
        y^{u,<}_\infty(\theta) = \theta - 2\left(\alpha^{u,<}_\infty(\theta) + \frac{\pi}{2}\right),\quad y^{s,<}_{\circleSminus}(\theta) = \theta - 2\left(\alpha^{s,<}_{\circleSminus}(\theta) + \frac{\pi}{2}\right).
    \end{aligned}
\end{equation*}
The goal is to express the curve $\Tilde{\gamma}^{u,<}_\infty$ in \eqref{eqn: Parameterizations of gammau< and gammas< in terms of theta in coordinates (Tilde(beta),z)} as a graph in terms of $\Tilde{\beta}$ to apply Theorem \ref{thm: Transition map close to collision}.

To this end, we make use of the following fact: since the curve $\Tilde{\gamma}^{s,<}_{\circleSminus}(\theta) \subset \stableManifoldSminusmu$ for all $\theta \in B_\varepsilon(\theta_>)$, it satisfies that $\Tilde{\beta}^{s,<}_{\circleSminus}(\theta) = 0$ as stated in equation \eqref{eqn: straightening Wus S-}. Therefore, we can rewrite $\Tilde{\beta}^{u,<}_\infty(\theta)$ in \eqref{eqn: Expression of Tilde(beta) and z} as
\begin{equation}\label{eqn: Tilde(beta)u< - Tilde(beta)s<}
    \Tilde{\beta}^{u,<}_\infty(\theta) = \Tilde{\beta}^{u,<}_\infty(\theta) - \Tilde{\beta}^{s,<}_{\circleSminus}(\theta) = \alpha^{u,<}_{\infty}(\theta) - \alpha^{s,<}_{\circleSminus}(\theta) - \left(\psi_-\left(\delta,  y^{u,<}_\infty(\theta)\right) - \psi_-\left(\delta,  y^{s,<}_{\circleSminus}(\theta)\right)\right)
\end{equation}
and therefore 
\begin{equation}\label{eqn: Expression Trick Tilde(beta)'}
\begin{aligned}
    \Tilde{\beta}^{u,<}_\infty{}'(\theta_<) =& \left(\alpha^{u,<}_{\infty}{}'(\theta_<) - \alpha^{s,<}_{\circleSminus}{}'(\theta_<)\right)\left(1-2\partial_y \psi_-(\delta, y_*)\right),
\end{aligned}
\end{equation}
where $y_* = y^{u,<}_\infty(\theta_<) = y^{s,<}_{\circleSminus}(\theta_<)$.



Hence, using \eqref{eqn: Approximations of partial_y psi_- and partial_x psi_+} and \eqref{eqn: Approximations of the derivatives of alpha_u< and alpha_s< at theta(mu)}, $\Tilde{\beta}^{u,<}_\infty{}'(\theta_<)$ admits the following approximation

\begin{equation}\label{eqn: Approximation Tilde(beta)'(theta(mu))}
    \Tilde{\beta}^{u,<}_\infty{}'(\theta_<) = \frac{d_-'(\theta_<)}{\sqrt{2(1-\mu)\delta}} + \mathcal{O}(\mu\delta),
\end{equation}
where $d_-(\theta) = d_-\left(\theta, \hat{\Theta}_0(\mu)\right)$ is defined in \eqref{eqn: Explicit formula for the distance}.  By \eqref{eqn: Asymptotic formula for the distances}, \eqref{eqn: translation melnikov} and Lemma \ref{lemma: derivative of the melnikov functions not 0 for almost every theta} it satisfies 
\begin{equation}\label{eqn: d_-' bound}
C^{-1}\mu\leq|d_-'(\theta_<)|\leq C\mu,
\end{equation}
for some $C>0$ independent of $\mu$ and $\delta$. This implies that, for $0<\mu\ll\delta$ small enough, $\Tilde{\beta}^{u,<}_\infty{}'(\theta_<) \neq 0$.


Hence, in a sufficiently small neighborhood of $\theta_<$ (for instance, we can consider $B_\varepsilon(\theta_>)$), the Inverse Function Theorem defines an inverse for $\Tilde{\beta}^{u,<}_\infty(\theta)$, which we denote $\theta^{u,<}_\infty(\Tilde{\beta})$, in an interval $\Tilde{B}_u$
with $0 \in \Tilde{B}_u$ (since $\Tilde{\beta}^{u,<}_\infty(\theta_<) = \Tilde{\beta}^{s,<}_{\circleSminus}(\theta_<) = 0$) and 
\begin{equation}\label{eqn: Expression of theta'(Tilde(beta))}
\theta^{u,<}_{\infty}(0) = \theta_<, \quad \theta^{u,<}_{\infty}{}'(0) = \frac{1}{\Tilde{\beta}^{u,<}_\infty{}'(\theta_<)}.    
\end{equation}
Therefore, we can express the curve $\Tilde{\gamma}^{u,<}_\infty$ in \eqref{eqn: Parameterizations of gammau< and gammas< in terms of theta in coordinates (Tilde(beta),z)} as a graph in terms of $\Tilde{\beta}$ of the form
\begin{equation}\label{eqn: Curve Tilde gamma_u,<}
    \Tilde{\gamma}^{u,<}_\infty = \Big\{\left(\Tilde{\beta},z^{u,<}_\infty(\Tilde{\beta})\right)\colon \Tilde{\beta} \in \Tilde{B}_u\Big\},
\end{equation}
with $z^{u,<}_\infty(\Tilde{\beta}) = z^{u,<}_\infty\circ \theta^{u,<}_\infty (\Tilde{\beta})$ satisfying
\begin{equation}\label{eqn: Approximation of zu<'(0)}
\begin{aligned}
    z^{u,<}_\infty{}'(0) =& z^{u,<}_\infty{}'(\theta_<)\cdot \theta^{u,<}_\infty{}'(0) = \partial_y \chi_-^{-1}(\delta,y_*)\cdot y_{\infty}^{u,<}{}'(\theta_<)\cdot \theta^{u,<}_\infty{}'(0)\\
    =& \left(1+\mathcal{O}_4(\delta)\right) \left(1-2\alpha_\infty^{u,<}{}'(\theta_<)\right) \cdot \left(\Tilde{\beta}_\infty^{u,<}{}'(\theta_<)\right)^{-1}\\
    =& \left(1+\mathcal{O}_4(\delta)\right)\left(\sqrt{2(1-\mu)}\frac{\delta}{d_-{}'(\theta_<)} - 2\frac{\Theta_\infty^u{}'(\theta_<)}{ d_-{}'(\theta_<)} + \mathcal{O}_2(\delta)\right) \\
    =& \left(1 + \mathcal{O}_4(\delta)\right)\sqrt{2(1-\mu)}\frac{\delta}{d_-{}'(\theta_<)},
\end{aligned}
\end{equation}
with $d_-'(\theta_<)$ satisfying \eqref{eqn: d_-' bound} and $y_* = y^{u,<}_\infty(\theta_<) = y^{s,<}_{\circleSminus}(\theta_<)$. The result follows from \eqref{eqn: Approximations of partial_z chi_- and partial_w chi_+}, \eqref{eqn: Approximations of the derivatives of alpha_u< and alpha_s< at theta(mu)}, \eqref{eqn: Expression of Tilde(beta) and z}, \eqref{eqn: Approximation Tilde(beta)'(theta(mu))} and \eqref{eqn: Expression of theta'(Tilde(beta))}.

Denoting $\Tilde{\beta} = \nu$, we are now ready to apply the transition map $\Tilde{f}$ from Theorem \ref{thm: Transition map close to collision} to the curve \eqref{eqn: Curve Tilde gamma_u,<}, which leads the following curve
\begin{equation}\label{eqn: Tilde(gamma)uinfty in coordinates (Tilde(iota),w)}
\begin{aligned}
    \Tilde{\gamma}^{u,>}_\infty(\nu) = (\Tilde{\iota}^{u,>}_\infty(\nu), w^{u,>}_\infty(\nu)) = \left(-\nu+\mathcal{O}_1(\delta\nu), z^{u,<}_\infty(\nu) + \mathcal{O}_1(\delta^2\nu,\nu^2)\right),
\end{aligned}
\end{equation}
for  $\nu \in (0,\nu_0)$ with $\nu_0 > 0$ small enough, which satisfies
\begin{equation}\label{eqn: tangent vector Tilde(gamma)uinfty at 0}
    \Tilde{\gamma}^{u,>}_\infty{}'(0)  = (\Tilde{\iota}^{u,>}_\infty{}'(0), w^{u,>}_\infty{}'(0))= \left(-1+\mathcal{O}_1(\delta), z^{u,<}_\infty{}'(0) + \mathcal{O}_1(\delta)\right).
\end{equation}
This curve is expressed in coordinates $(\Tilde{\iota},w)$. Now, we apply the inverse of the diffeomorphism $\Gamma_+$ from Proposition \ref{proposition: Straightening local diffeomorphisms and motions} to express the curve $\Tilde{\gamma}^{u,>}_\infty$ in coordinates $(\theta,\alpha)$, which yields to the following parameterization
\begin{equation}\label{eqn: Curve gammau> in coordinates (alpha,theta) parameterized by nu}
    \overline{\gamma}^{u,>}_\infty(\nu) = \left(\theta^{u,>}_\infty(\nu), \alpha^{u,>}_\infty(\nu)\right)
\end{equation}
such that
\begin{equation*}
    \begin{aligned}
        \theta^{u,>}_\infty(\nu) &= \chi_+(\delta, w^{u,>}_\infty(\nu)) - 2\left(\Tilde{\iota}^{u,>}_\infty(\nu) - \psi_+(\delta,\chi_+(\delta,w^{u,>}_\infty(\nu)))\right)\\
        \alpha^{u,>}_\infty(\nu) &= \Tilde{\iota}^{u,>}_\infty(\nu) + \psi_+(\delta,\chi_+(\delta,w^{u,>}_\infty(\nu)) + \frac{\pi}{2}
    \end{aligned}
\end{equation*}
where $\theta^{u,>}_{\infty}(\nu)$ satisfies $\theta_\infty^{u,>}(0) =\theta_> = \theta_< + \mathcal{O}_2(\delta)$ and
\begin{equation}\label{eqn: Expression of thetau>'(0)}
\begin{aligned}
    \theta^{u,>}_{\infty}{}'(0) =& \left(1 + 2\partial_{x}\psi_+(\delta,x_*)\right)\cdot \partial_{w}\chi_+\left(\delta, w_*\right)\cdot w_{\infty}^{u,>}{}'(0) -2\Tilde{\iota}_{\infty}^{u,>}{}'(0)\\
    =&\left(1+2\partial_{x}\psi_+(\delta,x_*)\right)\cdot(1+\mathcal{O}_4(\delta))\cdot\left(z^{u,<}_{\infty}{}'(0) + \mathcal{O}_1(\delta)\right) +2 + \mathcal{O}_1(\delta) \\
    =& z^{u,<}_{\infty}{}'(0) + 2 + \mathcal{O}_1(\delta),
\end{aligned}
\end{equation}
where $x_* = \chi_+(\delta,w_*)$ and $w_* = w^{u,>}_{\infty}(0)$. This expression is obtained from \eqref{eqn: Approximations of partial_z chi_- and partial_w chi_+}, \eqref{eqn: Approximations of partial_y psi_- and partial_x psi_+} and \eqref{eqn: tangent vector Tilde(gamma)uinfty at 0}. It follows from \eqref{eqn: Approximation of zu<'(0)} that, taking $0<\mu \ll \delta$ small enough, $\theta^{u,>}_{\infty}{}'(0) \neq 0$. 

Therefore, the Inverse Function Theorem ensures the existence of an inverse $\nu_\infty^{u,>}(\theta)$ for $\theta \in B_{\varepsilon}(\theta_>)$ satisfying
\begin{equation}\label{eqn: nu(theta)}
    \nu_\infty^{u,>}(\theta_>) = 0,\quad \nu_{\infty}^{u,>}{}'(\theta_>) = \frac{1}{\theta_\infty^{u,>}{}'(0)},
\end{equation}
allowing us to express the curve $\overline{\gamma}^{u,>}_\infty (\nu)$ in \eqref{eqn: Curve gammau> in coordinates (alpha,theta) parameterized by nu} as a graph of the form

\begin{equation}\label{eqn: Curve gammau> as a graph in coordinates (alpha,theta)}
    \overline{\gamma}^{u,>}_\infty = \Big\{\left(\theta, \alpha^{u,>}_{\infty}(\theta)\right) \colon \theta \in B_\varepsilon(\theta_>)\Big\},
\end{equation}
such that
\begin{equation}\label{eqn: Expression of alphau> in terms of theta}
    \alpha^{u,>}_\infty(\theta) = \alpha^{u,>}_\infty \circ \nu^{u,>}_\infty(\theta) = \Tilde{\iota}^{u,>}_\infty \circ \nu^{u,>}_\infty(\theta) + \psi_+(\delta,\chi_+(\delta,w^{u,>}_\infty\circ \nu^{u,>}_\infty(\theta))) + \frac{\pi}{2}.
\end{equation}
In order to get a proper approximation of both $\alpha^{u,>}_\infty(\theta_>)$ and $\alpha^{u,>}_\infty{}'(\theta_>)$, we use an analogous argument to \eqref{eqn: Tilde(beta)u< - Tilde(beta)s<}, using in this case the curve $\gamma^{u,>}_{\circleSplus}$ in \eqref{eqn: Curves as graphs of theta in coordinates (theta,Theta)} as a reference. To this end, following the changes \eqref{eqn:McGehee map Collision}, \eqref{eqn:polar change of coordinates for u,v into rho alpha} and the diffeomorphism $\Gamma_+$ in Proposition \ref{proposition: Straightening local diffeomorphisms and motions}, we obtain the curve $\gamma^{u,>}_{\circleSplus}$ expressed in coordinates $(\Tilde{\iota},w)$, parameterized by $\theta \in B_{\varepsilon}(\theta_>)$ as follows
\begin{equation*}
    \Tilde{\gamma}^{u,>}_{\circleSplus} = \Big\{\left(\Tilde{\iota}_{\circleSplus}^{u,>}(\theta), w^{u,>}_{\circleSplus}(\theta)\right) \colon \theta \in B_\varepsilon(\theta_>)\Big\},
\end{equation*}
such that
\begin{equation}\label{eqn: Tilde(iota)S+ and wS+}
    \begin{aligned}
        \Tilde{\iota}_{\circleSplus}^{u,>}(\theta) = \alpha_{\circleSplus}^{u,>}(\theta) - \frac{\pi}{2} - \psi_+\left(\delta, \chi_+(\delta,w_{\circleSplus}^{u,>}(\theta))\right),\quad w_{\circleSplus}^{u,>}(\theta) = \chi_+^{-1}(\delta, x_{\circleSplus}^{u,>}(\theta))
    \end{aligned}
\end{equation}
with 
\begin{equation}\label{eqn: alphaS+ and xS+}
    \alpha^{u,>}_{\circleSplus}(\theta) = \arccos\left(\frac{\Theta_{\circleSplus}^u(\theta) + \mu\delta^2\cos \theta - \delta^4}{\delta \sqrt{2(1-\mu) + \rho(\theta)}}\right),\quad x_{\circleSplus}^{u,>}(\theta) = \theta - 2\left(\alpha_{\circleSplus}^{u,>}(\theta) - \frac{\pi}{2}\right).
\end{equation}
We recall that $w_* = w^{u,>}_\infty(0) = w^{u,>}_{\circleSplus}(\theta_>)$ corresponds to the $w$-component of the intersection $p_>^*$ in \eqref{eqn: Points of intersection}.

Since $\Tilde{\gamma}_{\circleSplus}^{u,>}\subset \unstableManifoldSplusmu$, by \eqref{eqn: straightening Wus S+} the component $\Tilde{\iota}_{\circleSplus}^{u,>}(\theta) = 0$, yielding the following identity 
\begin{equation*}
    \alpha^{u,>}_{\circleSplus}(\theta) = 
    \frac{\pi}{2} + \psi_+\left(\delta, \chi_+(\delta,w_{\circleSplus}^{u,>}(\theta))\right).
\end{equation*}
Then, we can rewrite $\alpha^{u,>}_\infty(\theta)$ in \eqref{eqn: Expression of alphau> in terms of theta} as
\begin{equation}\label{eqn: Relation alpha_infty>, alpha_S+>}
\begin{aligned}
    \alpha^{u,>}_{\infty}(\theta) =& \alpha^{u,>}_{\circleSplus}(\theta) + \Tilde{\iota}^{u,>}_\infty \circ \nu^{u,>}_\infty(\theta) + \left(\psi_+(\delta,\chi_+(\delta,w^{u,>}_\infty\circ \nu^{u,>}_\infty(\theta)))-  \psi_+\left(\delta, \chi_+(\delta,w_{\circleSplus}^{u,>}(\theta))\right)\right)
\end{aligned}
\end{equation}
Moreover, the derivative at $\theta_>$ is given by
\begin{equation*}
\begin{aligned}
    \alpha^{u,>}_\infty{}'(\theta_>) =&  \alpha^{u,>}_{\circleSplus}{}'(\theta_>) + \Tilde{\iota}_{\infty}^{u,>}{}'(0)\cdot \nu_{\infty}^{u,>}{}'(\theta_>) \\
    &+ \left(\partial_{x}\psi_+(\delta,x_*)\cdot \partial_w \chi_+(\delta,w_*)\right)\left(w^{u,>}_\infty{}'(0)\cdot \nu_{\infty}^{u,>}{}'(\theta_>) - w_{\circleSplus}^{u,>}{}'(\theta_>)\right).
\end{aligned}
\end{equation*}
The expressions $\Tilde{\iota}^{u,>}_\infty{}'(0)$ and $w^{u,>}_\infty{}'(0)$ from \eqref{eqn: tangent vector Tilde(gamma)uinfty at 0} and $\nu_{\infty}^{u,>}{}'(\theta_>)$ from \eqref{eqn: nu(theta)} allow us to obtain
\begin{equation*}
\begin{aligned}
\Tilde{\iota}_{\infty}^{u,>}{}'(0)\cdot \nu_{\infty}^{u,>}{}'(\theta_>) =& (-1+\mathcal{O}_1(\delta))\cdot \frac{1}{\theta_\infty^{u,>}{}'(0)} = (-1+\mathcal{O}_1(\delta))\cdot \frac{1}{z_{\infty}^{u,<}{}'(0) + 2 + \mathcal{O}_1(\delta)} \\
=& (-1+\mathcal{O}_1(\delta))\cdot \frac{1}{\frac{\delta\sqrt{2(1-\mu)}}{d_-'(\theta_<)}-2\frac{\Theta_\infty^{u}{}'(\theta_<)}{d_-'(\theta_<)}+2 + \mathcal{O}_1(\delta)}\\
=& (-1+\mathcal{O}_1(\delta)) \cdot \frac{1}{\frac{\sqrt{2(1-\mu)}\delta}{d_-'(\theta_<)}\left(1 - 2\frac{\Theta_\infty^u{}'(\theta_<) - d_-'(\theta_<)}{\delta \sqrt{2(1-\mu)}} + \mathcal{O}_1(\mu)\right)}\\
=& (-1+\mathcal{O}_1(\delta))\left[\frac{d_-{}'(\theta_<)}{\delta\sqrt{2(1-\mu)}}\left(1 + \frac{2\Theta^{s}_{\circleSminus}{}'(\theta_<)}{\delta\sqrt{2(1-\mu)}} + \mathcal{O}_1(\mu)\right)\right]\\
=& (-1+\mathcal{O}_1(\delta))\left[\frac{d_-{}'(\theta_>)}{\delta\sqrt{2(1-\mu)}}\left(1 + \frac{2\Theta^{s}_{\circleSminus}{}'(\theta_>)}{\delta\sqrt{2(1-\mu)}} + \mathcal{O}_1(\mu)\right)\right],
&\\
w^{u,>}_\infty{}'(0)\cdot \nu_{\infty}^{u,>}{}'(\theta_>) =& \frac{z^{u,<}_\infty{}'(0) + \mathcal{O}_1(\delta)}{z^{u,<}_{\infty}{}'(0) + 2 + \mathcal{O}_1(\delta)} = 1 + \mathcal{O}_1\left(\frac{\mu}{\delta}\right).
\end{aligned}
\end{equation*}
since $\theta_< = \theta_> + \mathcal{O}_2(\delta)$ so 
\begin{equation*}
\begin{aligned}
d_-'(\theta_<) &= d_-'(\theta_> + \mathcal{O}_2(\delta)) = d_-'(\theta_>) + \mathcal{O}_1(\mu\delta^2),\\
\Theta^{s}_{\circleSminus}{}'(\theta_<) &= \Theta^{s}_{\circleSminus}{}'(\theta_> + \mathcal{O}_2(\delta)) = \Theta^{s}_{\circleSminus}{}'(\theta_>) + \mathcal{O}_1(\mu\delta^2).
\end{aligned}
\end{equation*}
The computation of $w_{\circleSplus}^{u,>}{}'(\theta_>)$ is obtained from \eqref{eqn: Tilde(iota)S+ and wS+} and \eqref{eqn: alphaS+ and xS+} as
\begin{equation*}
    w^{u,>}_{\circleSplus}{}'(\theta_>)  = \partial_x \chi_+^{-1}(\delta, x_*)\cdot x_{\circleSplus}^{u,>}{}'(\theta_>) = \left(1+\mathcal{O}_4(\delta)\right)\left(1-2\alpha^{u,>}_{\circleSplus}{}'(\theta_>)\right) = 1 + \mathcal{O}_1\left(\frac{\mu}{\delta}\right)
\end{equation*}
given $0<\mu\ll\delta$ small enough, where $x_*$ corresponds to the $x$-component of the intersection $p_>(\mu)$ in \eqref{eqn: Points of intersection}, $\partial_x \chi_+^{-1}(\delta,x_*)$ follows from the estimate \eqref{eqn: Approximations of partial_z chi_- and partial_w chi_+} and $\alpha^{u,>}_{\circleSplus}{}'(\theta_>)$ is obtained from \eqref{eqn: alphaS+ and xS+} (and whose approximation is equivalent to those in \eqref{eqn: Approximations of the derivatives of alpha_u< and alpha_s< at theta(mu)}). This result, along with the approximations of $\partial_w\chi_+(\delta,w)$ and $\partial_x\psi_+(\delta,x)$ in \eqref{eqn: Approximations of partial_z chi_- and partial_w chi_+} and \eqref{eqn: Approximations of partial_y psi_- and partial_x psi_+} respectively, give us the following result

\begin{equation}\label{eqn: Approximation alpha_u>infty in terms of alphau>S+}
\alpha^{u,>}_\infty{}'(\theta_>) = -\frac{\Theta^{u}_{\circleSplus}{}'(\theta_>) + d_-{}'(\theta_>)}{\sqrt{2(1-\mu)}\delta} +\mathcal{O}(\mu\delta).
\end{equation}
Finally, applying the inverse of the changes \eqref{eqn:McGehee map Collision} and \eqref{eqn:polar change of coordinates for u,v into rho alpha} on the curve $ \overline{\gamma}^{u,>}_\infty$ in \eqref{eqn: Curve gammau> as a graph in coordinates (alpha,theta)}, we obtain the curve $\gamma^{u,>}_\infty$ in coordinates $(\theta,\Theta)$ as a graph of the form

\begin{equation*}
    \gamma^{u,>}_\infty = \Big\{(\theta,\Theta^{u,>}_\infty(\theta))\colon \theta \in B_\varepsilon(\theta_>)\Big\},
\end{equation*}
where
\begin{equation*}
\begin{aligned}
    \Theta^{u,>}_\infty(\theta) = \delta \sqrt{2(1-\mu) + \rho(\theta)} \cos \left(\alpha^{u,>}_\infty(\theta)\right) + \delta^4 - \mu \delta^2 \cos\theta
\end{aligned}
\end{equation*}
with $\rho(\theta)$ as defined in \eqref{eqn: Approximation rho Appendix D}. This expression and its derivative can be approximated by \eqref{eqn: Approximation alpha_u>infty in terms of alphau>S+} respectively, yielding \eqref{eqn: Expression of Thetau> infty through f} and completing the proof.
\printbibliography

\end{document}

%% file: notation.tex
\newcommand\Alpha{\mathrm{A}}
\newcommand\de{\delta}

\newcommand{\constantValue}{\kappa}

\newcommand{\constantEquilibrium}{m_0}

\newcommand{\timeParametrization}{w}

\newcommand{\mcGeheeTimeCollision}{\tau}

\newcommand{\hamiltonianPolarRotatingCoordinatesCenteredSun}{\mathcal{H}}
\newcommand{\hamiltonianPolarRotatingCoordinatesCenteredCM}{\hat{\mathcal{H}}}

\newcommand{\hamiltonianCartesianRotatingCoordinates}{\mathcal{H}}

\newcommand{\jacobiConstant}{\mathcal{J}}

\newcommand{\mcGeheeFirstIntegral}{M}

\newcommand{\initialtheta}{\overline{\theta}}

\newcommand{\rOrbitPolarRotatingUnperturbedProblem}{ r_0}
\newcommand{\thetaOrbitPolarRotatingUnperturbedProblem}{ \theta_0}
\newcommand{\ROrbitPolarRotatingUnperturbedProblem}{ R_0}
\newcommand{\ThetaOrbitPolarRotatingUnperturbedProblem}{ \Theta_0}

\newcommand{\orbitPolarRotatingUnperturbedProblemRpos}{ \varphi_0^+ }
\newcommand{\orbitPolarRotatingUnperturbedProblemRneg}{ \varphi_0^- }

\newcommand{\rOrbitPolarRotatingUnperturbedProblemRpos}{ \rOrbitPolarRotatingUnperturbedProblem^+ }
\newcommand{\thetaOrbitPolarRotatingUnperturbedProblemRpos}{ \thetaOrbitPolarRotatingUnperturbedProblem^+ }
\newcommand{\ROrbitPolarRotatingUnperturbedProblemRpos}{ \ROrbitPolarRotatingUnperturbedProblem^+ }
\newcommand{\ThetaOrbitPolarRotatingUnperturbedProblemRpos}{ \ThetaOrbitPolarRotatingUnperturbedProblem^+ }

\newcommand{\rOrbitPolarRotatingUnperturbedProblemRneg}{ \rOrbitPolarRotatingUnperturbedProblem^- }
\newcommand{\thetaOrbitPolarRotatingUnperturbedProblemRneg}{ \thetaOrbitPolarRotatingUnperturbedProblem^- }
\newcommand{\ROrbitPolarRotatingUnperturbedProblemRneg}{ \ROrbitPolarRotatingUnperturbedProblem^- }
\newcommand{\ThetaOrbitPolarRotatingUnperturbedProblemRneg}{ \ThetaOrbitPolarRotatingUnperturbedProblem^- }

\newcommand{\rRotatingCenteredCM}{\hat{r}}
\newcommand{\thetaRotatingCenteredCM}{\hat{\theta}}
\newcommand{\RRotatingCenteredCM}{\hat{R}}
\newcommand{\ThetaRotatingCenteredCM}{\hat{\Theta}}

\newcommand{\stableManifoldInftymu}{W_\mu^s(\Alpha_{\hat{\Theta}_0})}
\newcommand{\unstableManifoldInftymu}{W_\mu^u(\Alpha_{\hat{\Theta}_0})}
\newcommand{\stableManifoldInftymuO}{W_0^s(\Alpha_0)}
\newcommand{\unstableManifoldInftymuO}{W_0^u(\Alpha_0)}

\newcommand{\stableManifoldSminusmu}{ W_\mu^s(S^-) }

\newcommand{\unstableManifoldSplusmu}{W_\mu^u(S^+) }

\newcommand{\stableManifoldCollisionSplusmu}{W_{\mu}^s(\circleSplus)}
\newcommand{\unstableManifoldCollisionSminusmu}{W_{\mu}^u(\circleSminus)}

\newcommand{\mcGeheeRadiusInfinity}{\xi}

\newcommand{\mcGeheeRadialMomentum}{v}
\newcommand{\mcGeheeAngularMomentum}{u}

\newcommand{\mcGeheePolarRadius}{\rho}
\newcommand{\mcGeheePolarAngle}{\alpha}
\newcommand{\mcGeheeRegularRadius}{s}

\newcommand{\mcGeheeTransformationCollision}{\psi}

\newcommand{\generalMcGeheeInvariantManifold}{\mathcal{M}}

\newcommand{\collisionManifold}{\Omega}
\newcommand{\circleSplus}{S^+}
\newcommand{\circleSminus}{S^-}
\newcommand{\circlesSpm}{S^\pm}

\newcommand{\equilibriumPointSplus}{S_{\initialtheta}^+}
\newcommand{\equilibriumPointSminus}{S_{\initialtheta}^-}

\newcommand{\heteroclinicConnectionCollisionmu}{\gamma_h}
\newcommand{\thetaHeteroclinicConnectionCollisionmu}{\theta_h}
\newcommand{\alphaHeteroclinicConnectionCollisionmu}{\alpha_h}

\newcommand{\vectorFieldMcGeheePolarCoordinatesmu}{F_\mu}

\newcommand{\eigenvaluerSpmGeneralVectorField}{\lambda_s}
\newcommand{\eigenvalueThetaSpmGeneralVectorField}{\lambda_\theta}
\newcommand{\eigenvalueAlphaSpmGeneralVectorField}{\lambda_\mcGeheePolarAngle}

\newcommand{\sectionrTwo}{\Sigma}

\newcommand{\sectionrTwohRpos}{\Sigma_{h}^{>}}
\newcommand{\sectionrTwohRneg}{\Sigma_{h}^{<}}

\newcommand{\circleSpm}{S^\pm}

\newcommand{\invariantManifoldsInfinity}{W_\mu^{s,u}(\Alpha_{\hat{\Theta}_0})}
\newcommand{\reducedInvariantManifoldsInfinity}{\Tilde{W}_\mu^{s,u}(\Alpha_{\hat{\Theta}_0})}
\newcommand{\reducedStableManifoldInfinitymu}{\Tilde{W}_\mu^{s}(\Alpha_{\hat{\Theta}_0})}
\newcommand{\reducedUnstableManifoldInfinitymu}{\Tilde{W}_\mu^{u}(\Alpha_{\hat{\Theta}_0})}
\newcommand{\invariantManifoldsSpm}{W_\mu^{u,s}(\circlesSpm)}

\newcommand{\closeCollision}{\delta}

\newcommand{\ballCollisionSminus}{\mathcal{B}_\closeCollision^-}

\newcommand{\angleStraightWsSminus}{\Tilde{\angleSminus}}
\newcommand{\straightenFiberWsSminus}{z}

\newcommand{\angleStraightWuSplus}{\Tilde{\angleSplus}}

\newcommand{\straightenFiberWuSplus}{w}

\newcommand{\angleSminus}{\beta}
\newcommand{\angleSplus}{\iota}

\newcommand{\localStableInvariantManifoldFiberSminusReferencey}{W_\mu^s(\circleSminus_{y_0})}

\newcommand{\localUnstableInvariantManifoldFiberSplusReferencey}{W_\mu^{u}(\circleSplus_{x_0})}

\newcommand{\graphOfStableManifoldSminus}{\psi_{-}}

\newcommand{\graphOfFiberStableManifoldSminus}{\chi_{-}}

\newcommand{\graphOfUnstableManifoldSplus}{\psi_{+}}

\newcommand{\graphOfFiberUnstableManifoldSplus}{\chi_{+}}

\newcommand{\curveIntersectionStableInfinitySectionrTwohRPos}{\Delta_\infty^{s}}
\newcommand{\curveIntersectionUnstableInfinitySectionrTwohRNeg}{\Delta_\infty^{u}}

\newcommand{\curveIntersectionUnstableSplusSectionrTwohRPos}{\Delta_{\circleSplus}^{u}}

\newcommand{\distancepos}{d_+}
\newcommand{\melnikovpos}{M_+}
\newcommand{\distanceneg}{d_-}
\newcommand{\melnikovneg}{M_-}

\newcommand{\poincareMapRposRneg}{\mathcal{P}}

\newcommand{\mapStraightenSminus}{\Gamma_-}
\newcommand{\mapStraightenSplus}{\Gamma_+}